\documentclass[a4paper,12pt]{article}

\usepackage{amsthm,amsmath,stmaryrd,bbm,hyperref,geometry,color,authblk}
\usepackage[utf8]{inputenc}
\usepackage{amssymb}
\usepackage[english]{babel}
\usepackage{graphicx}
\usepackage{amsfonts,amssymb}
\usepackage{verbatim}
\usepackage{enumitem}

\newcommand{\po}{\left(}
\newcommand{\pf}{\right)}
\newcommand{\co}{\left[}
\newcommand{\cf}{\right]}
\newcommand{\cco}{\llbracket}
\newcommand{\ccf}{\rrbracket}
\newcommand{\R}{\mathbb R}

\newcommand{\N}{\mathbb N} 
\newcommand{\dd}{\text{d}}
\newcommand{\na}{\nabla}
\newcommand{\1}{\mathbbm{1}} 
\newcommand{\pierre}[1]{\textcolor{blue}{#1}}

\newtheorem{thm}{Theorem}
\newtheorem{assu}{Assumption}
\newtheorem*{assu*}{Assumption}
\newtheorem{lem}[thm]{Lemma}
\newtheorem{defi}[thm]{Definition}
\newtheorem{cor}[thm]{Corollary}
\newtheorem{prop}[thm]{Proposition}
\newtheorem{rem}{Remark}

\title{Stochastic moments dynamics: a flexible finite-dimensional random perturbation of Wasserstein gradient descent}
\author{Pierre Germain, Pierre Monmarché}

\begin{document}
\maketitle

\begin{abstract}
For optimizing a non-convex function in finite dimension, a method is to add Brownian noise to a gradient descent, allowing for transitions between basins of attractions of different minimizers. To adapt this for optimization over a space of probability distributions requires a suitable noise. For this purpose, we introduce here a simple stochastic process where a number of moments of the distribution are following a chosen finite-dimensional diffusion process, generalizing some previous studies where the expectation of the measure is subject to a Brownian noise. The process may explode in finite time, for instance when trying to force the variance of a distribution to behave like a Brownian motion. We show, up to the possible explosion time, well-posedness and  propagation of chaos  for the system of mean-field interacting particles with common noise approximating the process.
\end{abstract}

\section{Introduction}

\subsection{Motivation}

 In finite dimension, in order to minimize a non-convex energy $U\in\mathcal C^1(\R^d)$, a possibility is to consider the overdamped Langevin diffusion,
 \begin{equation}
 \label{eq:overdamped}
\dd X_t = -\na U\po X_t\pf \dd t + \sqrt{2\varepsilon} \dd B_t
 \end{equation}
where $\varepsilon>0$ is a temperature parameter. It can be interpreted as a noisy perturbation of a gradient descent (to be distinguished from the stochastic gradient method where the noise comes from a random approximation of the gradient $\na U$). Thanks to the additional noise, the process exits from its 
initial energy well and visits all the space, while staying preferentially in low-energy regions, avoiding an exhaustive isotropic search in high dimension (provided $\varepsilon$ is small enough). More precisely, under mild assumptions, the process is ergodic with respect to the Gibbs measure with density proportional to $\exp(- \frac{1}{\varepsilon} U)$, which means that indeed the whole space will be visited in the long-time, while the time spent on average in a region is proportional to its probability with respect to the Gibbs measure, ensuring that most time is spent around minimizers. The simulated annealing algorithm is a variant where $\varepsilon$ depends on time and goes to zero. If it decays sufficiently slowly, the process converges in probability to the \emph{global} minimizers of $U$.

Our primary motivation is to adapt this in the context of Wasserstein gradient flows, for optimization over $\mathcal P(\R^d)$ the set of probability measures over $\R^d$.  Although this is our underlying objective, in fact the theory of Wasserstein gradient flow will not play any role in the rest of our study so that we simply refer the interested reader to \cite{ambrosio2005gradient} for theoretical background, and to \cite{Szpruch,yao2022mean,li2022sampling,chizat2018global,lambert2022variational} for a far from exhaustive sample of recent applications. To explain our motivation it is sufficient to consider the example of the granular media equation
\begin{equation}\label{eq:granularmedia}
\partial_t \rho_t = \sigma^2 \Delta \rho_t + \na\cdot \po \rho_t \po \na U + \rho_t\star \na W\pf\pf\,,
\end{equation}
where $\rho_t$ is a probability density over $\R^d$, $\na\cdot$ stands for the divergence operator, $U,W \in \mathcal C^1(\R^d)$ are respectively called the confining and interaction potentials and $\sigma>0$ is a diffusivity parameter. Under suitable conditions on $U$ and $W$, the trajectory $(\rho_t)_{t\geqslant 0}$ can be shown to be the gradient flow of a so-called free energy $\mathcal F$ defined for $\mu \in \mathcal P(\R^d)$ by $\mathcal F(\mu) =   +\infty$ if $\mu$ doesn't have a density with respect to the Lebesgue measure and otherwise, still writing $\mu$ for the density,

\begin{equation} \label{eq:energy}
    \mathcal F \po \mu \pf = \sigma^2 \int_{\R^d} \mu(x) \ln \mu(x) \dd x + \int_{\R^d} U(x)  \mu(x)\dd x  +\frac12 \int_{\R^d\times \R^d} W(x-y) \mu(x)\mu(y)\dd x \dd y\,. 
\end{equation}

Stationary solutions of \eqref{eq:granularmedia}  are then critical points of the free energy, and there may be several of them. This is for instance the case with $d=1$, the double-well confining potential $U(x) = \frac{x^4}{4} - \frac{x^2}{2}$ and an attracting quadratic interaction $W(x) = x^2/2$. Then, there exists a critical parameter $\sigma_*>0$ such that, for $\sigma\geq \sigma_*$, \eqref{eq:granularmedia} admits a unique stationary solution (which is the global minimizer of $\mathcal F$) while for $\sigma < \sigma_*$ there are three such stationary solutions (two of which being global minimizers, the third being an unstable saddle point), see \cite{Tugautdoublewell}. More generally, in $\R^d$ with $W(x)=|x|^2/2$, then for any non-degenerate local minimizer $x_*$ of $U$ which is the global minimizer of $x\mapsto U(x) + |x-x_*|^2/2$,  when $\sigma$ is sufficiently low, there exists a local minimizer of $\mathcal F$ which is close to a Dirac mass at $x_*$ \cite{tugaut2014self,MonmarcheReygner}. 

Facing such an optimization problem over $\mathcal P(\R^d)$ with several minimizers, by analogy with the finite-dimensional case, we wish to add some noise to \eqref{eq:granularmedia}.

In a practical implementation, \eqref{eq:granularmedia} is approximated by the empirical distribution $\pi(X):= \frac1N\sum_{i=1}^N \delta_{X_i}$ of a particle system $X=(X^1,\dots,X^N)$ solving

\begin{equation} \label{eq:mcKeanPS}
    \forall i\in\cco 1,N\ccf,\qquad \dd X_t^i = -\na U(X_t^i) \dd t - \frac1N \sum_{j=1}^N \na W(X_t^i - X_t^j) \dd t + \sqrt{2}\sigma\dd B_t^i\,,
\end{equation}
where $B^1,\dots,B^N$ are independent Brownian motions on $\R^d$. In other words $X$ is an overdamped Langevin process with energy $U_N(x) = \sum_{i=1}^N U(x_i) + \frac12\sum_{i,j=1}^N W(x_i-x_j)$. In particular it is ergodic and thus, along an arbitrarily long trajectory, its empirical distribution is going to get close at some points to each of the stationary solutions of \eqref{eq:granularmedia}. More precisely, in the double well potential in dimension 1 for instance, the law of $\pi(X)$ at equilibrium is approximately $\frac12\delta_{\rho_-} + \frac12\delta_{\rho_+}$ where $\rho_-,\rho_+$ are the two global minimizers of $\mathcal F$.  However, when $N$ is large, the process is metastable and the transitions between the two minimizers are exponentially large with $N$, and are thus not seen in practice \cite{Monmarchemetastable}. In some sense, in a first order approximation, the evolution of $\pi(X_t)$  is given by \eqref{eq:granularmedia} with an additional noise, but the intensity of this noise vanishes as $N\rightarrow \infty$ \cite{FERNANDEZ199733}. By contrast, we are interested in noise which doesn't vanish in this mean-field limit.

Random dynamics, and in particular diffusion processes, over spaces of probability measures have been considered for long for modeling purpose, as they naturally arise as the mean field limit of interacting agents or particles subject to a common environmental noise. More recently, apart from this modeling motivation, the question of  designing stochastic processes which are reversible with respect to some natural distribution over $\mathcal P(\R^d)$ have gained interest, see \cite{von2009entropic,Shao} and references within (having in mind the finite-dimensional case~\eqref{eq:overdamped}, notice that  defining a suitable notion of Gibbs measure over $\mathcal P(\R^d)$ associated to $\mathcal F$ is unclear in the absence of a Lebesgue measure). Another motivation for the construction of diffusion processes over $\mathcal P(\R^d)$ is to take advantage of the regularization properties of the noise to get well-posed equations, as in \cite{delarue2024rearranged}. Finally, closely related to the present study, the viewpoint in \cite{ANGELI2023127301,maillet2023note,delarue2024ergodicity} is that the noise can induce ergodicity (i.e. uniqueness of a stationary distribution) in cases where the deterministic flow has several stationary solutions.

The Wasserstein space being a metric space, an approach to define diffusion processes in this setting is to work within the theory of Dirichlet forms \cite{dello2022dirichlet,overbeck1995analytic,ren2024diffusion,von2009entropic,Shao,ren2024markov}. Another possibility is to work in a space of functions in bijection with (a subset of) $\mathcal P(\R^d)$, for instance repartition functions in one dimension \cite{delarue2024rearranged,delarue2024rearranged2,von2009entropic} or densities  in some Hilbert space with a trace-class Gaussian distribution as the reference measure  \cite{EberleBouRabee}. However, the viewpoint taken in the present paper is different, following the idea that the fact that the Gibbs measure is invariant for the overdamped Langevin diffusion is nice but not crucial from an optimization point of view. Basically, what we need is that the process spends most of its time in low-energy regions (which holds if the parameter $\varepsilon$ controlling the intensity of the noise is small), and some controllability property that allows for transitions between minimizers.

As in other cases of interest, in the granular media equation \eqref{eq:granularmedia}, typically, different critical points of $\mathcal F$ are not distinguished by fine microscopic informations (like regularity properties, high-frequency Fourier modes) but rather by large-scale features. For instance, when $W(x)=|x|^2$, they are characterized by their expectation (although their basins of attraction are not, cf. \cite[Proposition 14]{MonmarcheReygner}). In other words, to allow for transitions between minimizers, it would seem relevant to add random perturbation to the expectations of $\rho_t$ in \eqref{eq:granularmedia}. Creating a stochastic flow on $\mathcal P(\R^d)$ whose expectation is a Brownian motion is rather simple, taking $t\mapsto \mu \po \cdot + B_t\pf$. Using this elementary motion to perturb  the initial motion amounts to consider
\[\dd X_t^i = -\na U(X_t^i) \dd t - \frac1N \sum_{j=1}^N \na W(X_t^i - X_t^j) \dd t + \sqrt{2}\sigma\dd B_t^i + \sqrt{\varepsilon} \dd B_t\,,\]
where $B$ is a $d$-dimensional Brownian motion, common to all particles. This is precisely the process studied in~\cite{delarue2024ergodicity,maillet2023note}, where it is indeed shown that this finite-dimensional noise is sufficient to ensure the ergodicity of the process (implying in particular that transitions between neighborhoods of the minimizers of the free energy occur).

In more general situations, however, randomizing the expectation alone may not be enough. A natural extension is to fix a collection of observables $f_1,\dots,f_p:\R^d \rightarrow \R$ and to design a process $\mu_t$ over $\mathcal P(\R^d)$ so that $(\mu_t(f_1),\dots,\mu_t(f_p))$ follows a prescribed stochastic differential equation over $\R^p$. Following this idea, in the next section, we introduce a process that we call stochastic moment dynamics (SMD) and which is the main topic of this work.  This definition, the well-posednes analysis of the process (cf. Proposition~\ref{prop:stochasticmomentswelldefined}) and the convergence of a particle approximation towards the mean-field limit  (cf. Proposition~\ref{prop:PoC}) are our main contributions. In view of our initial objective,  the ergodicity of the process, under suitable conditions, in the spirit of \cite{ANGELI2023127301,maillet2023note,delarue2024ergodicity},  should be addressed in the future, but for now this is beyond the scope of the present work.

\subsection{Derivation of the process}\label{sec:derivation}

In this section we forget the gradient part of the dynamics, and we focus on informally designing the stochastic dynamics over $\mathcal P(\R^d)$ that will eventually play a role similar to the Brownian motion in~\eqref{eq:overdamped}.

Let $f=(f_1,\dots,f_p)\in\mathcal C^2(\R^d,\R^p)$. Given a $p$-dimensional Brownian $W$, a drift  $a:\R^p \rightarrow \R^p$ and a diffusion coefficient $s:\R^p \rightarrow \R^{p\times p}$, the goal is to find a process $(\mu_t)_{t\geqslant 0}$ on $\mathcal P(\R^d)$ such that $Z_t:= \mu_t(f) = \int_{\R^d} f(x) \mu_t(\dd x)$ solves
\begin{equation}
\label{eq:EDSZ}
\dd Z_t = a(Z_t) \dd t + s(Z_t) \dd W_t\,.
\end{equation} 
For convenience for readers not (yet) familiar with SPDE or conditional McKean-Vlasov diffusion processes, in this informal derivation we work first with a system of interacting particles (the rigorous definition of the process in general, together with the necessary theoretical background, is exposed in Sections~\ref{sec:definition} and \ref{sec:def-process}), which is besides of interest as it will eventually be what is implemented in practice.

Consider a particle system $X=(X^1,\dots,X^N)$ in $(\R^d)^N$ solving an SDE of the form
\begin{equation}\label{eq:particule-1}
\dd X^i_t = b_{\pi(X_t)}(X_t^i) \dd t + \sigma_{\pi(X_t)}(X_t^i) \dd W_t
\end{equation}
for some drift and diffusion coefficients $(\mu,x)\mapsto b_{\mu}(x),\sigma_{\mu}(x)$ to be determined (with, for $\mu\in\mathcal P(\R^d)$, $b_\mu:\R^d\rightarrow \R^d$ and $\sigma_\mu:\R^d \rightarrow \R^{d\times p}$). Notice that the Brownian motion $W$ is the same for all particles. Applying It$\bar{\text{o}}$ formula and writing $\pi_t=\pi(X_t)$,  for $k\in\cco 1,p\ccf$,
\begin{eqnarray*}
\dd \frac1N\sum_{i=1}^N f_k(X_t^i) &=& \frac1N \sum_{i=1}^N \co\na f_k(X_t^i) \cdot b_{\pi_t}(X_t^i) + \frac12 \sigma_{\pi_t}(X_t^i)  \sigma_{\pi_t}^T(X_t^i) : \na^2 f_k(X_t^i)   \cf \dd t \\
& & \ + \po \frac1N\sum_{i=1}^N\na f_k (X_t^i) \cdot \sigma_{\pi_t}(X_t^i) \pf \dd W_t\,.
\end{eqnarray*} 
In other words,  
\begin{equation}\label{eq:pit}
\dd \pi_t( f) = \pi_t  \co(\na f)^T b_{\pi_t} + \frac12 \sigma_{\pi_t}  \sigma_{\pi_t}^T: \na^2 f   \cf \dd t  \ + \pi_t \po (\na f)^T  \sigma_{\pi_t}\pf \dd W_t\,,
\end{equation}
where we denoted the Jacobian matrix $\na f=(\partial_{x_i} f_j) \in \R^{d\times p}$ (with $i$ the line and $j$ the column) and  wrote
\[ \sigma_{\mu}  \sigma_{\mu}^T: \na^2 f = \begin{pmatrix}
\sigma_{\mu}  \sigma_{\mu}^T: \na^2 f_1\\
\vdots\\
\sigma_{\mu}  \sigma_{\mu}^T: \na^2 f_p
\end{pmatrix}\,.\]
Identifying \eqref{eq:pit} with \eqref{eq:EDSZ} for $Z_t=\pi_t(f)$, for $\mu \in \mathcal P(\R^d)$ we want $\sigma_{\mu}$ to satisfy 
\[\mu \po (\na f)^T  \sigma_{\mu}\pf = s\po \mu(f)\pf \,.  \]
Fix some $\sigma^0_\mu(x) \in\R^{d\times p}$ and set $\sigma_{\mu}(x) = \sigma_\mu^0(x) M_\mu$ with some $M_{\mu}\in\R^{p\times p}$ to be determined. Then 
\[\mu \po (\na f)^T  \sigma_{\mu}\pf =\mu \po (\na f)^T  \sigma_\mu^0 \pf M_{\mu}\,.  \] 
Hence we chose
\[M_{\mu} = \po \mu \po (\na f)^T  \sigma_\mu^0 \pf\pf^{-1} s\po \mu(f)\pf \,, \] 
provided this inverse is well-defined (which is clearly not always the case, but we will come back to this later on; for now, let us assume that it is well-defined for this particular $\mu$). 

Once $\sigma_{\mu}$ fixed, to get \eqref{eq:EDSZ},  we are left to define $b_{\mu}$ in order that 
\[ \mu  \co(\na f)^T b_{\mu} + \frac12 \sigma_{\mu}  \sigma_{\mu}^T: \na^2 f   \cf = a \po \mu(f)\pf\,,\]
namely
\[ \mu  \co(\na f)^T b_{\mu}    \cf = R_{\mu} :=  - \frac12 \co \sigma_{\mu}  \sigma_{\mu}^T: \na^2 f   \cf +  a \po \mu(f)\pf\,.\]
Given some $\sigma^1_\mu(x) \in\R^{d\times p}$, we look for a drift of the form $b_{\mu}(x) = \sigma_{\mu}^1(x) N_{\mu}$ with  $N_{\mu} \in \R^p$ to determine. The previous condition then reads
\[\mu  \co(\na f)^T \sigma_{\mu}^1   \cf N_{\mu} = R_{\mu}\]
and thus, provided the inverse is well-defined, we set
\[N_\mu = \po \mu  \co(\na f)^T \sigma_{\mu}^1   \cf\pf^{-1} R_{\mu}\,.\]

The choices of  $\sigma^0$ and $\sigma^1$ have been left free. Notice that a natural candidate is $\sigma^0_\mu =  \sigma_\mu^1 = \na f$. Indeed, in this case $(\na f)^T \na f$ is a   semi-positive symmetric matrix, and $\mu \co (\na f)^T \na f\cf$ is positive definite  if, for instance, $\mu$ has a positive Lebesgue density and the family $\{(\na f)^T(x), x\in\R^d\}$ has rank $p$ (this last condition being necessary to hope for \eqref{eq:EDSZ} in general, otherwise there exists a non-zero linear combination of the $f_j$'s which is constant; for instance if  $f_3=f_2+f_1$, the evolution of  $\mu_t(f_3)$ is fixed by the two other coordinates). 

To sum up, for $\mu \in \mathcal P(\R^d)$ such that  $\mu(\nabla f^T \nabla f)$ is non-singular, we set

\begin{equation} \label{eq:coeffsp}
     \left\{
        \begin{array}{lll}
        \sigma_{\mu}(x) &=& \nabla f (x)\po \mu[\nabla f^T \nabla f]\pf ^{-1} s(\mu(f))\\
    
        b_{\mu}(x) &=& \nabla f (x) \po \mu[\nabla f^T \nabla f]\pf ^{-1} [a(\mu(f)-\frac{1}{2}\mu(\sigma_{\mu}\sigma_{\mu}^T :\nabla^2 f)] \,.
        \end{array}\right. 
\end{equation}
This determines the dynamics~\eqref{eq:particule-1}.

    \begin{rem}
   In the case where we want the expectation of the law to be a Brownian motion, we retrieve the natural process discussed in the previous section. Indeed, in that case, $f(x)=x$, $s(x)= I_d$ and $a(x)=0$, from which $b_\mu=0$ and $\sigma_\mu = I_d$ for all $\mu \in\mathcal P(\R^d)$, resulting in
   \[\forall i\in\cco 1,N\ccf,\qquad \dd X_t^i = \dd W_t\,,\]
   and then $\pi_t = \frac1N \sum_{i=1}^N \delta_{X_0^i + W_t}$ is the image of $\pi_0$ by $x\mapsto x+W_t$.
    \end{rem}

 Applying It$\bar{\text{o}}$ formula as in \eqref{eq:EDSZ} but now for any smooth observable $g\in \mathcal C^\infty_c(\R^d,\R)$, notice that
\[\dd \pi_t( g) = \pi_t  \co(\na g)^T b_{\pi_t} + \frac12 \sigma_{\pi_t}  \sigma_{\pi_t}^T: \na^2 g   \cf \dd t  \ + \pi_t \po (\na g)^T  \sigma_{\pi_t}\pf \dd W_t\,.\]
In other words, $\pi_t$ is (at least informally) a weak solution of the SPDE
    \begin{equation}\label{eq:EDPS}
            d\mu_t = -\na \cdot \po b_{\mu_t} \mu_t\pf \dd t -  \na \cdot \po    \mu_t \sigma_{\mu_t}\pf   dW_t^0+ \frac{1}{2}\sum_{i,j=1}^d \partial_{i,j}^2 \po  (\sigma_{\mu_t}\sigma_{\mu_t}^T)_{i,j} \mu_t \pf  \dd t \,.
    \end{equation}

    Before proceeding, let us highlight that there are clearly cases where the process is not defined for all times. For instance, take $d=p=1$, $f(x) = x^2$, $a(z)=0$ and $s(z)=1$. Since we are forcing the second moment of the distribution to follow a Brownian motion, a problem will definitely arise when the latter tries to become negative. Since $(\na f)^T(x)\na f(x) = 4 x^2$,  $\mu[\nabla f^T \nabla f]$ is strictly positive for all $\mu \in \mathcal P(\R)$ except $\mu = \delta_0$. On the other hand, $\mu(f)=0$ implies that $\mu=\delta_0$. Considering the stopping time
    \[\tau = \inf\{t\geqslant 0, \mu_0(f) + W_t = 0\}\,,\]
    we expect the solution $\mu_t$ of \eqref{eq:EDPS} to be well defined up to time $ \tau$, with $\mu_t\neq \delta_0$ for all $t<\tau$ so that $\mu_t(f)$ solves \eqref{eq:EDSZ}, i.e. $\mu_t(f) = \mu_0(f) + W_t$ for all $t<\tau$. At time $\tau$, $\mu_{\tau}=\delta_0$, $\mu_{\tau}[\nabla f^T \nabla f]=0$ and we cannot make sense of \eqref{eq:EDPS}.
    
    Even when combining this dynamics with \eqref{eq:granularmedia}, the  regularizing properties of the latter may not be enough to avoid this explosion in finite time. Take for instance in \eqref{eq:granularmedia} the case where $U=W=0$ and $\sigma^2=1$, which amounts to add a Laplacian in \eqref{eq:EDPS}. The second moment $Z_t = \mu_t(f)$ with $f(x)=x^2$ now solves
    \[\dd Z_t = 2 \dd t + \dd W_t\,.\]
  Even with the positive drift due to the elliptic regularization, the probability to reach $\mu_t=\delta_0$ in finite time is still not zero.
 
 \subsection{Overview of the work}

 This paper is organized as follows. Useful theoretical background on conditional McKean-Vlasov equations is recalled in Section~\ref{sec:definition}. In Section~\ref{sec:def-process}, which is the core of our work, we introduce rigorously  the SMD, establish its well-posedness under suitable conditions and a propagation of chaos result for its approximation by a particle system. The proofs of these results are then given in Section~\ref{sec:proof-main}.
 Section~\ref{sec:explosion} discusses further the explosive behavior of the process, first by providing a Lyapunov criterion to prevent it in Section~\ref{subsec:Lyap}, and then by introducing a non-explosive regularized variation of the SMD in Section~\ref{ref:regularise}. The SMD is illustrated first on two simple examples  in Section~\ref{sec:exemples} and then with numerical experiments in Section~\ref{sec:numerique}. Finally, auxiliary technical lemmas are gathered in an appendix.

\subsection{General notations}

Throughout the paper, we identify the space of real \(k \times m\) matrices with the Euclidean space \(\mathbb{R}^{k \times m}\), and for square matrices we use the notation \(A : B\) to denote the standard Hilbert–Schmidt inner product:
\[
A : B := \sum_{i,j=1}^k A_{i,j} B_{i,j}\,.
\]
We denote by $\mathcal{P}(\R^d)$ the space of Borel probability measures on $\R^d$, and by $m_p(\mu)$ the $p$-th moment of $\mu\in \mathcal{P}(\R^d)$:
\[m_p(\mu) =\int_{\R^d}|x|^pd\mu(x)\]
and $\mathcal{P}_p(\R^d)$ be the subset of probability measures with finite $p$-th moment:
\[
\mathcal{P}_p(\R^d) := \left\{ \mu \in \mathcal{P}(\R^d) \;\middle|\;m_p(\mu)< \infty  \right\}.
\]
The \emph{Wasserstein distance of order $p$} between $\mu, \nu \in \mathcal{P}_p(E)$ is defined by
\[
\mathcal{W}_p(\mu, \nu) := \left( \inf_{\pi \in \Gamma(\mu, \nu)} \int_{\R^d\times \R^d} |x-y|^p \, d\pi(x, y) \right)^{1/p},
\]
where $\Gamma(\mu, \nu)$ denotes the set of couplings (i.e., probability measures on $E \times E$ with marginals $\mu$ and $\nu$).

The metric space $(\mathcal{P}_p(\R^d), \mathcal{W}_p)$ is complete and separable \cite{villani2008optimal}. Moreover, convergence in $\mathcal{W}_p$ metrizes the weak convergence of probability measures together with convergence of $p$-th moments. $(\mathcal{P}_p(\R^d), \mathcal{W}_p)$ is a geodesic space: if $\mu, \nu \in \mathcal{P}_p(E)$ and $\pi \in \Gamma_{\mathrm{opt}}(\mu, \nu)$ is an optimal coupling, then the curve $(\mu_t)_{t \in [0,1]}$ defined by
\[
\mu_t := ((1 - t)x + t y)_\# \pi
\]
is a constant-speed geodesic connecting $\mu$ and $\nu$:

\[ \mathcal{W}_p(\mu_s,\mu_t)=|t-s| \,  \mathcal{W}_p(\mu_0,\mu_1)  \quad \forall s,t \in [0,1]\,.\]

\section{On conditional McKean-Vlasov equations} \label{sec:definition}

We will define the process using the framework of conditional McKean-Vlasov equations. The latter is meant to describe the limit as $N\rightarrow \infty$ of a system of interacting particles of the form
\begin{equation} \label{eq:commonnoisePS}
    dX_t^{i,N}=b\po t,X_t^{i,N},\pi(X_t^N)\pf \dd t+\sigma \po t,X_t^{i,N},\pi(X_t^N)\pf \dd W_t^{i}+\sigma^0(t,X_t^{i,N},\pi(X_t^N))\dd W_t^0\,,
\end{equation}
where $(W^0,\dots,W^N)$ are independent Brownian motions and $b,\sigma,\sigma^0$ are suitable coefficients. Without the common noise $W^0$,  the particles are expected to be approximately independent so that their empirical distribution converges to their deterministic common law (which is the so-called propagation of chaos phenomenon). With a common noise, this will only happen conditionally to $W^0$. Hence we expect each particle to behave similarly to a solution of 
\begin{equation} \label{eq:commonnoise}
    dX_t=b(t,X_t,\mu_t)dt+\sigma (t,X_t,\mu_t) dW_t+\sigma^0(t,X_t,\mu_t)dW_t^0\,,
\end{equation}
with $W,W^0$ two independent Brownian motions and $\mu_t$ the conditional law of $X_t$ given $W^0$.

In the rest of this section, for the reader's convenience, we gather the known facts  about conditional McKean-Vlasov equations that will prove useful in our study. Most of the material here is from \cite{carmona2018probabilistic} and \cite{carmona2016mean}, to which we refer for further details and proofs. Notice that, for clarity, in this section we only consider \emph{non-explosive} processes, defined for all times. As usual for Markov processes, explosive processes are then defined through a localization procedure (see Section~\ref{sec:def-process}).

\subsection{Probabilistic framewok}

Conditional McKean Vlasov equations are a refinement of classical MacKean Vlasov equations: the dynamic is similar to the classical case up to an additional common noise term acting on all the interacting particles. These equations naturally arise in models where the particles interact with a random environment: the common noise term can be interpreted as random shocks on the particles coming from the random evolution of the environment. This new noise term induce a strong correlation on the different particles: considering the particle system  \eqref{eq:commonnoisePS}, the particles do not only interact through their empirical measure $\pi(X_t^N)$ but are also impacted by the environment (directed by the Brownian $W^0$).

In order to generalise classical results from the McKean Vlasov theory, it is fundamental to adapt the probabilistic framework in order to deal with this strong dependency between the particles. Indeed, the propagation of chaos classically quantify the diminution of the correlations of the particles when the size of the system goes to infinity: the interactions between the particles progressively becomes an interaction with a continuum of independent particles represented by the law of a typical particle. However, in the conditional McKean Vlasov theory, the particles are still strongly correlated at the limit because of the influence of the common noise term (this term does not disappear like the personal noise terms $W^i$ at the limit). Moreover, the distribution of the particles $\mu_t$ is also subject to the effect of the environment, and thus is expected to be random (while it is deterministic in the classical case of study). However, if one can assume that the realisation of the random environment is known, it seems that the framework of the classical theory is valid: particles evolve in a fixed environment and interact through their empirical distribution. This simple observation is central for the following and justify the name of ``conditional" McKean Vlasov equations. The probabilistic framework of the theory will be constructed in order to justify this ``conditioning with respect to the realisation of the environment", by clearly distinguishing the randomness coming from the environment from the randomness modeling the diffusion of the particles. The mean field limit of the empirical distribution $\pi(X_t^N)$ will then correspond to the conditional law of a typical particle with respect to the common noise term $W^0$. The flow of conditional laws $(\mu_t)$ will be understood as a random process of probability measure living in a Wasserstein space (this flow is determistic in the classical case).

\medskip{}

In the classical theory, the probabilistic framework naturally allows to deal with the linear case, meaning equations of the type:

\begin{equation*} 
    dX_t=b(t,X_t,\mu_t)dt+\sigma (t,X_t,\mu_t) dW_t
\end{equation*}
for $(\mu_t)$ a deterministic flow of probability measure (which may be different from $\mu_t=\mathcal{L}(X_t)$). Then, in the same spirit the probabilistic framework of the conditional McKean Vlasov theory should be able to deal with the equations of the type:

 \begin{equation} \label{eq:flotquelconque}
    dX_t=b(t,X_t,\mu_t)dt+\sigma (t,X_t,\mu_t) dW_t+\sigma^0(t,X_t,\mu_t)dW_t^0\,
\end{equation}
for $(\mu_t)$ a random flow of probability measure seen as a fixed input. 

\medskip{}

The general framework is built in order to manage the case where $(\mu_t)$ involves an additional source of randomness than $W^0$ (this won't be the case in our study). This setting requires assumptions on the random inputs $X_0,\mu,W,W^0$ and the filtered probability space on which the solutions are built. It is classically called ``admissible probabilistic set-up" and is formally detailed in \cite{carmona2018probabilistic} Chapters 1.1 and 1.2. This framework is necessary to give a notion of weak solutions to conditional McKean Vlasov equations but will not be of use in our study (see \cite{hammersley2021weak} for more details on weak solutions). This assumption on the probabilistic set-up is based on the notion of ``immersion" between two filtrations and has been extensively studied to give a definition of weak solutions of various stochastic models \cite{jacod2006weak}. Heurisitcally, this general probabilistic framework justify that given the knowledge of the processes $X^0,\mu,W, W^0$ up to time $t$, observing the filtration $\mathcal{F}_t$ should not introduce a bias in the evolution of the inputs after that time $t$. It is worth noticing that theses questions doesn't appear in the general case of study when the input $\mu_t$ is deterministic (see \cite{carmona2018probabilistic} Remark 1.11). The general study of conditional McKean Vlasov equations is then based on ``admissible probabilistic set-up", from which it is possible to develop a general theory. However, we will see that we don't need this level of generality in our study, which will use a simpler (though less general) probabilistic framework (because we will have enough regularity in the coefficients of the equation).

\medskip{}

Two results of the general theory are central for our study and justify the use of a simplified framework. First, it exists a general result of transfer of solutions into a canonical space (\cite[Lemma 1.27]{carmona2018probabilistic}): if it exists a solution in a general probabilistic set-up, it is possible to transfer this solution into a canonical set-up. Second, a Yamada-Watanabe type result (see \cite[Theorem 1.33]{carmona2018probabilistic}) justifies that if an equation is uniquely solvable (which will be the case when the coefficients $b$, $\sigma$, $\sigma^0$ verify a Lipschitz continuity property), the choice of the probabilistic set-up for the search of a solution has no influence on the law of the solution. Combining the two previous results, it justifies that the use of a ``canonical" set-up is possible and doesn't reduce the generality of the study. This will allow to better distinguish the different sources of noise than in a general setting and will simplify the study of the conditional laws. Note that this approach is standard when one deals with conditional McKean Vlasov equations with coefficients regular enough (to ensure path-wise uniqueness of the solution). The canonical probabilistic set-up is defined as follows.

\begin{defi}(Canonical probabilistic set-up). \label{def:setup}
    Let $(\Omega^0,\mathcal{F}^0,\mathbb{P}^0)$ and $(\Omega^1,\mathcal{F}^1,\mathbb{P}^1)$ be two complete probability spaces endowed with two complete and right continuous filtrations $\mathbb{F}^0=(\mathcal{F}^0_t)_{0 \leq t \leq T}$ and $\mathbb{F}^1=(\mathcal{F}^1_t)_{0 \leq t \leq T}$. Let $W^0=(W_t^0)_{0\leq t \leq T}$ be a $p$ dimensional $\mathbb{F}^0$ Brownian motion and $W=(W_t)_{0\leq t \leq T}$ be a $d$ dimensional $\mathbb{F}^1$ Brownian motion. We define the canonical probabilistic set-up the filtered probability space $(\Omega, \mathcal{F}, \mathbb{F}, \mathbb{P})$ defined as:
    \begin{itemize}
        \item $\Omega=\Omega^0\times \Omega^1$
        \item $(\mathcal{F}, \mathbb{P})$ is the completion of $(\mathcal{F}^0\otimes \mathcal{F}^1, \mathbb{P}^0\otimes \mathbb{P}^1)$
        \item $\mathbb{F}$ is the complete and right augmentation of $(\mathcal{F}_t^0\otimes \mathcal{F}_t^1)_{0\leq t\leq T}$ for $\mathbb{P}$.
    \end{itemize}
    
    On this set-up, we say that a process $(\mu_t)_{0\leq t \leq T}$ is an admissible input if it is defined on $\Omega^0$, is $\mathbb{F}^0$ progressively measurable and has continuous trajectories into a Wasserstein space. 
\end{defi}

In this set-up, generic elements of $\Omega$ will usually be denoted $\omega=(\omega^0, \omega^1)$. We shall usually use the notation $W$ instead of $W^1$ for the personal noise term. We will use the following  abuse of notation: a random variable $Z$ defined on $(\Omega^0, \mathcal{F}^0,\mathbb{F}^0, \mathbb{P}^0)$ will be identified with its natural extension $\tilde{Z}:(\omega_0,\omega^1) \mapsto Z(\omega^0)$ defined on $\Omega$. Likewise, a sub $\sigma$-algebra $\mathcal{G}^0 \subset \mathcal{F}^0$, will be identified with its natural extension $\tilde{\mathcal{G}}^0=\mathcal{G}^0\times \{ \emptyset , \Omega^1\}$. We will use the notation $\mathcal{P}_\alpha(\R^d)$ for the Wasserstein space of order $\alpha\geq1$ in $\R^d$. During the rest of this paper, we will only work on a canonical probabilistic set-up. One should keep in mind that this framework does not impact the generality of the following study because the coefficients of interest $b, \sigma, \sigma^0$ will be regular enough (locally Lispchitz). If one wishes to work with less regular coefficients, this specific probabilistic set-up cannot be used without loss of generality and one must use the general framework of admissible probabilistic set-ups.

\medskip{}

This specific framework is very useful because it separates with ease the different sources of noise. Indeed, all the information corresponding to the evolution of the random environment (directed by the Brownian motion $W^0$) is contained in $\mathcal{F}^0$ while the information corresponding to the personal diffusion of the particle is contained in $\mathcal{F}^1$. It is easier in this setting to formalize the idea of conditioning with respect to the environment. As explained previously, the mean field limit of the empirical distribution is expected to be the flow $(\mathcal{L}(X_t | \mathcal{F}^0))_{0\leq t \leq T}$. 
However, in order to fit the definition of ``admissible input", it is easier to understand the flow $(\mathcal{L}(X_t | \mathcal{F}^0))_{0\leq t \leq T}$ not as probability kernels but more like a stochastic process with values in a probability measure space (see \cite{carmona2018probabilistic} I Prop. 5.7). Once again, the structure of the probabilistic set-up is very useful to formalize this idea and leads to the following result (see \cite[Section 2.1.3]{carmona2018probabilistic} for more details):

\begin{prop} \label{prop:existcondlaw}(Existence of a regular flow of conditional laws)  Let $(X_t)_{0\leq t\leq T}$ be a $\mathbb{F}$ adapted process, with continuous paths in $\R^d$, such that it exists $\alpha\geq1$ such that $\mathbb{E}(\sup_{0\leq t\leq T}|X_t|^\alpha)< \infty$.
Then it exists a random process $(\mu_t)_{0\leq t \leq T}$ defined on $\Omega^0$ with values in the Wasserstein space $\mathcal{P}_\alpha(\R^d)$, which is $\mathbb{F}^0$ adapted and has continuous paths such that for all $t\in [0,T]$, $\mu_t(\omega^0)(\cdot)$ is a conditional law of $X_t$ with respect to $\mathcal{F}^0$. 
\end{prop}

\begin{rem}
    For any stochastic process with values in a metric space, if it exists a continuous modification of this process, then this modification is unique (up to indistinguishability). Then, given a process $(X_t)_{0\leq t\leq T}$ verifying the assumptions of Proposition \ref{prop:existcondlaw}, the process $(\mu_t)_{0\leq t \leq T}$ given by the proposition is pathwise unique. 
\end{rem}

This proposition ensures that we may consider a flow of conditional laws of a typical particle with respect to the common noise $\mathcal{F}^0$ which is an ``admissible input". Intuitively, a measure valued process $(\mu_t)$ is said to be an ``admissible input" if it has enough regularity (both in  time and space) and does not anticipate in the future (because it is adapted to $\mathbb{F}^0$). The previous proposition formally justifies that the mean field distribution knowing the common noise has enough regularity if the particles verify a locally uniform in time moment bound (which will be true in when the coefficients $b,\sigma,\sigma^0$ are Lipschitz). It is thus meaningful to consider the non-linear equation \eqref{eq:flotquelconque} with $(\mu_t)$ corresponding to the conditional law of the particle knowing the common noise.

\medskip{}

It is possible to give a notion of solution for conditional McKean Vlasov equations in a canonical probabilistic setting for particles verifying a moment bound condition:

\begin{defi}(Weak solution) \label{def:strongexistence}
    Given a canonical probability set-up (Definition \ref{def:setup}) and an $\mathcal{F}^1_0$ measurable initial condition $X_0$ with values on $\R^d$ verifying $\mathbb{E}(|X_0|^\alpha)< \infty$ for some $\alpha\geq1$, we call a weak solution to the conditional McKean Vlasov equation with coefficients $b,\sigma,\sigma^0$ on the interval $[0,T]$ of order $\alpha\geq1$ an $\mathbb{F}$ adapted process $(X_t)_{0\leq t \leq T}$ with values in $\R^d$ verifying the following conditions:
    \begin{enumerate}
        \item $(X_t)_{0\leq t \leq T}$ has continuous paths and $\mathbb{E}(\underset{0\leq t\leq T}{\text{sup}}|X_t|^\alpha)< \infty$
        \item $\mathbb{E}[\hspace{2pt} \int_0^T (|b(t,X_t,\mu_t)|+ |\sigma(t,X_t,\mu_t)|^2+|\sigma^0(t,X_t,\mu_t)|^2)dt ] < \infty$
        \item $\mathbb{P}$ a.s, for all $t\in [0,T]$:  
        \begin{equation}
            \displaystyle X_t=X_0+\int_0^tb(s,X_s,\mu_s)ds+\int_0^t\sigma(s,X_s,\mu_s)dW_s+\int_0^t\sigma^0(s,X_s,\mu_s)dW^0_s
        \end{equation}
        
    \end{enumerate}
    with $(\mu_t)_{0\leq t \leq T}$ given by Proposition \ref{prop:existcondlaw}.
\end{defi}

In the following we will usually call solution the couple $X,\mu$, with $\mu$ be the regular flow of conditional distribution of the solution $X$ given by Proposition \ref{prop:existcondlaw}. As with the study of classical stochastic differential equation, the adjective ``weak" refers to the fact that the outputs $X$ and $\mu$ are not necessarily adapted to the filtration generated by the inputs $X_0$, $W$ and $W^0$. Thus, it is natural to define ``strong" solutions, that verify this specific property.

\begin{defi} (Strong solution) \label{def:strongexistencefin}
    Let $(X_t)_{0\leq t\leq T}$ be a weak solution of order $\alpha\geq 1$ and $(\mu_t)_{0\leq t\leq T}$  be the the regular flow of conditional distribution given by Proposition \ref{prop:existcondlaw}. This solution is said to be ``strong" if $(X_t)_{0\leq t\leq T}$ is $\mathbb{F}^{X_0,W,W^0}$ adapted and $(\mu_t)_{0\leq t\leq T}$ is $\mathbb{F}^{W^0}$ adapted, with $\mathbb{F}^{X_0,W,W^0}$ (resp. $\mathbb{F}^{W^0}$) be the right continuous and completed augmentation of the filtration generated by $X^0,W,W^0$ (resp. $W^0$).
\end{defi}

It is also possible to define strong uniqueness for these equations.

\begin{defi}(Strong uniqueness)
    We say that strong uniqueness holds for a conditional McKean Vlasov equation with coefficients $b,\sigma,\sigma^0$ on the interval $[0,T]$ of order $\alpha\geq1$ if for any canonical set-up $(\Omega,\mathcal{F},\mathbb{F},\mathbb{P})$, for any initial condition $X_0$ verifying $\mathbb{E}(|X_0|^\alpha)< \infty$, if $(X_t)_{0\leq t \leq T}$ and $(X'_t)_{0\leq t \leq T}$ are two weak solutions, then $$\mathbb{P} \text{    a.s, for all   } t \in [0,T], \quad X_t=X'_t $$
\end{defi}

Whenever an equation is strongly solvable and strong uniqueness holds, we say that this equation is strongly uniquely solvable. 

\medskip{}

If $(X_t)_{0\leq t\leq T}$ is a solution of a conditional McKean Vlasov equation, its flow of conditional distribution given the common noise $(\mu_t)_{0\leq t \leq T}$ is solution in a weak sense to the stochastic Fokker-Plank equation.

\begin{prop}(Stochastic Fokker-Plank equation) \label{prop:stochasticfokkerplank}
    Let $(X_t)_{0\leq t\leq T}$ be a solution to the conditional McKean Vlasov equation with coefficients $b,\sigma,\sigma^0$, and $(\mu_t)_{0\leq t \leq T}$ be the flow of conditional distribution given the common noise. Then, $\mathbb{P}^0$ a.s, $(\mu_t)$ satisfies the following equation in the distributional sense when acting on $\mathcal{C}^\infty_0(\R^d,\R)$ the space of smooth functions from $\R^d$ to $\R$ going to $0$ at infinity:
    \begin{equation} \label{eq:stochasticFP}
        \begin{split}
            d\mu_t =& -\na_x \cdot (b(t, \cdot, \mu_t) \mu_t)dt+ \frac{1}{2}\mathrm{trace}[\na^2_{xx}((\sigma \sigma^T(t,\cdot, \mu_t)+\sigma^0 (\sigma^0)^T(t,\cdot, \mu_t))\mu_t]dt \\
            &- \na_x \cdot ( \sigma^0(t, \cdot, \mu_t)  \mu_t) dW_t^0
        \end{split}
    \end{equation}
\end{prop}

Note that in the classical case ($\sigma^0=0$), we recover the usual (deterministic) Fokker-Plank equation. This proposition justifies the use of conditional McKean Vlasov equations in our study, whose framework leads to the derivation of a stochastic partial differential equation that will be interpreted as a noisy version of a Wasserstein gradient flow.

\medskip{}
 
To sum up, the probabilistic set-up is made for the study of conditional McKean Vlasov equation with coefficients $b,\sigma, \sigma^0$ regular enough. It allows to give a definition of the existence and uniqueness of conditional McKean Vlasov processes, and of the underlying stochastic Fokker-Plank equation. The next subsection will be dedicated to an exposition of the classical results when the coefficients are Lipschitz continuous, for which the previous framework is perfectly adapted. We will present the main results of the theory: strong existence and uniqueness of the solutions, a mesurability result justifying in this setting the adjective ``strong", a Yamada-Watanabe type result on the conditional laws and finally a propagation of chaos result.

\subsection{Study of the Lipschitz case}

The goal of this subsection is to present some of the most important results of the theory for Lipschitz coefficients. These results are well known in the classical McKean Vlasov theory and are very similar in the conditional case. The difficulty of conditional McKean Vlasov equations rely more on the construction of an appropriate theoretical framework than in the proofs of these results. Like in the classical theory, the Lipschitz case consists in a simple (yet very useful) framework in which many results are true in a strong sense. During the rest of this subsection, we will present results on equations with coefficients verifying the following hypothesis.

\begin{assu}(Lipschitz and locally bounded coefficients) \label{hyp:lipschitzbornés}
    It exists $\alpha\geq1$ such that the coefficients $b,\sigma, \sigma^0$ defined on $[0,T]\times \R^d\times \mathcal{P}_\alpha(\R^d)$ with values respectively on $\R^d$, $\R^{d\times d}$ and $\R^{d\times p}$ verify the following conditions:
    \begin{enumerate}
        \item The coefficients are bounded in bounded subsets of $[0,T]\times \R^d\times \mathcal{P}_\alpha(\R^d)$
        \item The coefficients are Lipschitz continuous on $x\in \R^d$ and $\mu\in \mathcal{P}_\alpha(\R^d)$ uniformly in time $t\in[0,T]$, $\R^d$ being equipped with the euclidean norm and $\mathcal{P}_\alpha(\R^d)$ with the Wasserstein distance of order $\alpha$.
    \end{enumerate}
\end{assu}

Under this hypothesis, conditional McKean Vlasov equations are strongly uniquely solvable.

\begin{prop} (Strong Existence and Uniqueness) \label{prop:existcoefflipschitz}
    Suppose that the coefficients $b,\sigma,\sigma^0$ verify Assumption \ref{hyp:lipschitzbornés} with $\alpha\geq1$. Then, given a canonical probabilistic set-up $ (\Omega, \mathcal{F}, \mathbb{F}, \mathbb{P})$ and an $\mathcal{F}_0^1$ measurable initial condition $X_0$ satisfying $\mathbb{E}(|X_0|^\alpha)< \infty$ strong existence hold on this probabilistic set-up and the solution is strongly unique.
\end{prop}

The proof of this statement is based on a fixed point iteration in the Banach space of $\mathbb{F}$ adapted continuous processes verifying a uniform in time moment bound. This result will be central for the study presented in Section \ref{sec:def-process}, whose objective is to prove the existence and the uniqueness for the coefficients defined by Equation \eqref{eq:coeffsp}. In this study, the coefficients will be singular in the general case, so it will not possible to apply directly the results of Proposition \ref{prop:existcoefflipschitz}. Nevertheless, it will be possible to use a localization procedure based on it, to ensure existence and uniqueness of the solution up to a blowing time.

\medskip{}

The regular flow of conditional measure $\mu$ is built as a version of the conditional laws of the solution given the $\sigma$-field $\mathcal{F}^0$. However, when this flow is adapted to the input $W^0$, it seems that knowing the entire $\sigma$-field $\mathcal{F}^0$ is rather unnecessary: one may think that knowing the information contained in the common noise is enough. This intuition is true when the equation is strongly uniquely solvable.

\begin{prop}
    Suppose that the coefficients verify Assumption \ref{hyp:lipschitzbornés} with $\alpha\geq1$. Given a probabilistic set-up $(\Omega,\mathcal{F},\mathbb{F},\mathbb{P})$, an $\mathcal{F}_0^1$ measurable initial condition $X_0$ verifying $\mathbb{E}(|X_0|^\alpha)< \infty$, let $(X_t)_{0\leq t \leq T}$ be the unique solution and $(\mu_t)_{0\leq t \leq T}$ be its regular conditional flow given $\mathcal{F}^0$. Then for any $t\in [0,T]$, $\mathbb{P}^0$ a.s, $\mu_t$ is a version of a conditional distribution of $X_t$ given the $\sigma$-field $\sigma( W^0_s, 0\leq s\leq T)$.
\end{prop}

This proposition ensures that in the Lipschitz case, the flow of the conditional distributions of a typical particle $X_t$ given $\mathcal{F}^0$ only depends on the common noise $W^0$ and is adapted to its natural filtration. 

\medskip{}

Before presenting the propagation of chaos set-up, it is important to state a ``uniqueness in law" result, to justify the intuition that $\mu_t$ represent the (conditional) distribution of any ``typical" particle evolving in the same random environment. More precisely, two ``typical" particles $X$ and $X'$, whose initial states are independent with the same distribution, and whose evolution is directed by the same common noise $W^0$ and by two independent personal noises $W$ and $W'$, must have the same mean field distribution $\mu_t=\mu'_t$. This ``Yamada-Watanabe" type result is true when the coefficients are Lipschitz continuous.

\begin{prop} (Uniqueness of the mean field distribution) \label{prop:uniqueMFD}
    Suppose that the coefficients verify Assumption \ref{hyp:lipschitzbornés} for some $\alpha\geq1$. Given a probabilistic set-up $(\Omega,\mathcal{F},\mathbb{F},\mathbb{P})$, let $X_0$ and $X_0'$ be two $\mathcal{F}_0^1$ measurable initial conditions, independent and identically distributed verifying $\mathbb{E}(|X_0|^\alpha)< \infty$. Let $W'$ be a d dimensional $\mathbb{F}^1$ Brownian motion independant of $W$. Let $X$ (resp. $X'$) be the solution of the equation directed by $(X_0,W,W^0)$ (resp. $(X_0,W',W^0)$ ) and $\mu$ (resp. $\mu'$) the flow of conditional distribution associated. 
    
    Then $\mathbb{P}^0$ a.s, for all $t\in [0,T]$, $\mu_t=\mu'_t$ 
\end{prop}

This result is necessary for a rigorous formulation of the propagation of chaos result, by ensuring that the mean field repartition $\mu_t$, although random, is common to all typical particles. Classically, ``propagation of chaos" refers to the convergence in a suitable sense of the particle system \eqref{eq:commonnoisePS} to the solution of the McKean Vlasov SDE \eqref{eq:commonnoise}. The probabilistic set-up for this result is made to enable the use of a coupling argument: it is basically a set-up in which it is possible to construct solutions of both equations \eqref{eq:commonnoise} and \eqref{eq:commonnoisePS} directed by the same Brownian motions and is defined as follows.

\begin{defi}(Set-up for the propagation of chaos)\label{def:setuppropchaos}.  Let $(\Omega^0,\mathcal{F}^0,\mathbb{P}^0)$ and $(\Omega^1,\mathcal{F}^1,\mathbb{P}^1)$ be two complete probability spaces endowed with two complete and right continuous filtrations $\mathbb{F}^0=(\mathcal{F}^0_t)_{0 \leq t \leq T}$ and $\mathbb{F}^1=(\mathcal{F}^1_t)_{0 \leq t \leq T}$. Let $W^0=(W_t^0)_{0\leq t \leq T}$ be a $p$ dimensional $\mathbb{F}^0$ Brownian motion defined on $\Omega^0$ and let $(W^n)_{n\geq 1}$ be a sequence of independent $d$ dimensional $\mathbb{F}^1$ Brownian motion defined on $\Omega^1$. Let $(X_0^n)_{n \geq 1}$ be a sequence of independent and identically distributed $\mathcal{F}_0^1$ random variables defined on $\Omega^1$, that verify $\mathbb{E}(|X_0^1|^\alpha)< \infty$ for some $\alpha \geq 1$. Let $\Omega$ be the canonical probabilistic set-up associated to the two probabilistic spaces $(\Omega^0,\mathcal{F}^0,\mathbb{P}^0)$ and $(\Omega^1,\mathcal{F}^1,\mathbb{P}^1)$ and $b,\sigma,\sigma^0$ coefficients verifying Assumption \ref{hyp:lipschitzbornés}.

Because of Proposition \ref{prop:existcoefflipschitz}, it is possible to define for any $n\geq 1$ $\bar{X}^n=(\bar{X_t}^n)_{0\leq t \leq T}$ the solution of the conditional McKean Vlasov equation defined on $(\Omega,\mathcal{F},\mathbb{F},\mathbb{P},X_0^n,W^n,W^0)$ and $\bar{\mu}_t^n$ the regular flow of conditional distribution associated. By Proposition \ref{prop:uniqueMFD}, for all $n\geq 1$, the process $(\bar{\mu}^n)$ is indistinguishable from $(\bar{\mu}^1)$, so we can use the notation $\bar{\mu}$ for the mean field distribution common for all typical particles $\bar{X}^n$.

Because the coefficients $b,\sigma,\sigma^0$ are Lipschitz continuous, the classical theory of SDE ensures that for all $N\geq 1$, the particle system \eqref{eq:commonnoisePS} made of $N$ particles admits unique strong solution $(X^{n,N})_{1\leq n\leq N}$ on the probabilistic set-up $( \Omega,\mathcal{F},\mathbb{F},\mathbb{P},(X^{n,N}_0)_{1\leq n\leq N}, (W^k)_{0\leq k\leq N})$. 

We define $ \pi^N_t=\pi((X_t^{n,N})_{1\leq n \leq N})=\frac{1}{N}\sum_{n=1}^{N}\delta_{X_t^{n,N}}$ the empirical distribution of the particle's system.
\end{defi}

It is then possible to state the propagation of chaos result for conditional McKean Vlasov equations with Lipschitz continuous coefficients:

\begin{prop}(Propagation of chaos)
   \label{prop:propchaos}
   Within the above framework:
    \begin{equation}
        \underset{N\rightarrow \infty}{\text{lim}} \quad \underset{1\leq n \leq N}{\text{max}} \quad \mathbb{E}[\underset{0\leq t \leq T} \sup|X_t^{n,N}-\bar{X}^n_t|^\alpha]=0 
    \end{equation}
    and 
     \begin{equation}
        \underset{N\rightarrow \infty}{\text{lim}} \quad \underset{0\leq t \leq T}{\text{sup}} \quad \mathbb{E}[ \mathcal{W}_\alpha(\pi_t^N,\bar{\mu}_t)^\alpha]=0\,. 
    \end{equation}
\end{prop}

To sum up, most of the classical results for McKean Vlasov equations with Lipschitz continuous coefficients still hold in the conditional case. The main specificity of the study of the conditional case consists in the choice of an ``admissible probabilistic set-up", which can be chosen as a specific canonical set-up in the Lipschitz case. Using the adapted framework, strong existence and uniqueness hold as well as propagation of chaos (in its usual form). In the following, we will be interested in evolutions where coefficients are not globally Lipschitz continuous (see Equation \eqref{eq:coeffsp} defining the stochastic moment processes dynamic). The resulting processes will not be well defined for all time $t \leq T$, but only up to a random blowing time in the general case (see for instance the example presented at the end of Section \ref{sec:derivation}). The objective of the next section will be to use a localization procedure to justify the well posed-ness of the equation despite the explosive behavior of its solution.

 \section{Definition of the stochastic moments dynamics}\label{sec:def-process}

The goal of this section is to justify that the SMD informally derived in Section~\ref{sec:derivation} is well defined for general observables $f=(f_1,..,f_p)$ and coefficients $a$, $s$ directing the moment dynamic, corresponding to the following singular conditional McKean Vlasov equation:

\begin{equation} \label{eq:stomomentdyn}
    dX_t=\tilde{b}_{\mu_t}(X_t)\dd t+ \tilde{\sigma}_{\mu_t}(X_t)\dd W_t  + b_{\mu_t}(X_t) \dd t +\sigma_{\mu_t}(X_t) \dd W_t^0\,,
\end{equation}
where $\mu_t$ stands for the conditional law of the particle $X_t$ knowing the common noise, the coefficients $\tilde{b}$ and $\tilde{\sigma}$ correspond to the coefficients of the McKean Vlasov equation associated to the Wasserstein gradient flow (for instance, in the examples in Sections~\ref{sec:exemples} and \ref{sec:numerique}, these coefficients will be defined as $\tilde{b}_\mu(x)=-\nabla U(x) -\mu \ast \nabla W(x)$ and $\tilde{\sigma}_\mu(x)=constant$) and the coefficients $b$ and $\sigma$ correspond to coefficients of the term that will play the role of the noise in $\mathcal{P}(\mathbb{R}^d)$, as determined in Section~\ref{sec:derivation}:
\begin{equation} \label{eq:defcoefs}
     \left\{
        \begin{array}{lll}
        \sigma_{\mu}(x) &=& \nabla f (x)\po \mu[\nabla f^T \nabla f]\pf ^{-1} s(\mu(f))\\
    
        b_{\mu}(x) &=& \nabla f (x) \po \mu[\nabla f^T \nabla f]\pf ^{-1} [a(\mu(f))-\frac{1}{2}\mu(\sigma_{\mu}\sigma_{\mu}^T :\nabla^2 f)] \,.
    \end{array}\right. 
\end{equation}

\medskip{}

 As explained at the end of Section \eqref{sec:derivation}, it is not possible to ensure the definition of the process for all $t\leq T$ in the general case: for instance it is not possible to define a probability measure valued process with a negative second order moment. This is characterized by the fact that in the general case the coefficients $b$ and $\sigma$ directing the evolution of the McKean Vlasov particles (Equation \eqref{eq:coeffsp}) can be singular. For instance, considering the simple example in dimension $d=1$, where the stochastic moment process is defined so that its second moment order is a Brownian motion ($d = p = 1$, $\tilde{b}=\tilde{\sigma}=0$, $f (x) = x^2$, $a(z) = 0$ and $s(z) = 1$ in Equation \eqref{eq:coeffsp}), the coefficient $\sigma_\mu(x)$ is defined as:
$$\sigma_\mu(x)=\frac{4x}{m_2(\mu)}$$ where $m_2(\mu)=\int_\R x^2\mu(dx)$ is the second order moment of $\mu$. This coefficient is singular in $\mu=\delta_0$, so one cannot directly use the result of existence and uniqueness for Lipschitz coefficients. However, it is still possible to define the process as long as its conditional distribution $\mu_t$ differs from $\delta_0$ (because $\sigma$ is Lipschitz outside of every neighborhood of $\delta_0$), following a standard localization procedure as for explosive processes in finite dimension,   based on the construction on a sequence of regular truncations of the singular coefficients built to coincide with the original coefficients outside of an increasingly smaller neighborhood of the singularity. 

Depending on the hypothesis made on the observable $f$, we will define a parameter $\alpha \geq 2$ that will correspond to the index of Wasserstein space in which the solution is defined. The stronger the assumptions on the observable, the closer the parameter will be to $2$, and the stronger the result will be. Indeed, the dynamic will be well defined until a blowing time corresponding partially to the exit time from the Wasserstein space on which the solution is defined (the moment of the characteristic order may explode). Because of the natural inclusion between the Wasserstein spaces, the smaller the parameter $\alpha$ will be, the longer the solution will be defined.

The result will be stated under both strong hypothesis (that will be of practical use in the examples considered in Sections~\ref{sec:exemples} and \ref{sec:numerique}), and a rather general one that apply to a broad class of observables of interest in practice.
\begin{assu} \label{hyp:observableregu1}
    The observable \( f \in \mathcal{C}^2(\mathbb{R}^d, \mathbb{R}^p) \) must satisfy one of the following assumptions.
    
    \begin{enumerate}
        \item[(A.1)]\label{assu:hyp1} \( f \) is quadratic, meaning that for all \( k \in \{1, \dots, p\} \), there exist a matrix \( A \), a vector \( B \), and a constant \( C \) such that:
        \[
        f_k(x) = x^T A x + B^T x + C.
        \]
        
        \item[(A.2)]\label{assu:hyp2} For all \( k \in \{1, \dots, p\} \), the functions \( f_k \), \( \nabla f_k \), and \( \nabla^2 f_k \) are Lipschitz and bounded.
        
        \item[(A.3)]\label{assu:hyp3} The functions \( f \), \( \nabla f \) and for all \( k \in \{1, \dots, p\} \), \( \nabla^2 f_k \) are locally Lipschitz with at most a polynomial growth, meaning that there exist constants \( \alpha_1, \alpha_2, \alpha_3 \geq 0 \) and \( C_1, C_2, C_3 \geq 0 \) such that for all \( x, y \in \mathbb{R}^d \),
        
        \begin{equation*} 
        \left\{
        \begin{array}{ll}
            |f(x) - f(y)| &\leq C_1 |x - y| (1 + |x|^{\alpha_1} + |y|^{\alpha_1}), \\
            |\nabla f(x) - \nabla f(y)| &\leq C_2 |x - y| (1 + |x|^{\alpha_2} + |y|^{\alpha_2}), \\
            |\nabla^2 f_k(x) - \nabla^2 f_k(y)| &\leq C_3 |x - y| (1 + |x|^{\alpha_3} + |y|^{\alpha_3}).
        \end{array}
        \right.
        \end{equation*}
    \end{enumerate}

    If \( f \) satisfies Assumption (A.1) or (A.2) we set \( \alpha = 2 \). If \( f \) satisfies Assumption (A.3), we set:
    \[
    \alpha := \max(\alpha_1 + 1, 2\alpha_2 + \alpha_3 + 3).
    \]

    Thus, depending on which case holds, the appropriate value of \( \alpha \) is used.
\end{assu}

In addition to these hypothesis on the observable $f$, we require some regularity on the coefficients $a$ and $s$, which correspond respectively to the drift and diffusion matrix of the SDE satisfied by $\mu_t(f)$ in the case where $\tilde{b}=\tilde{\sigma}=0$, as discussed in Section~\ref{sec:derivation}.

\begin{assu}\label{hyp:coeffaets}
    The coefficients $a:\R^p\rightarrow \R^p$ and $s:\R^p\rightarrow \R^{p\times p}$ are supposed to be locally Lipschitz and bounded outside of a singularity set $\mathcal{S}$. Formally, we assume that it exists $\mathcal{S}\subset{\R^p}$ such that for all $\epsilon>0$ and $K\geq 0$, the restrictions of $a$ and $s$ to $(\mathcal{S}_\epsilon)^c \cap B(0,K)$ are Lipschitz continuous and bounded, with $(\mathcal{S}_\epsilon)^c:=\{z \in \R^p, \ d(z,\mathcal{S} )>\epsilon \}$.
\end{assu}

Finally, we need some control on the coefficients $\tilde{b}$ and $\tilde{\sigma}$.

\begin{assu} \label{hyp:driftetdiffusionenplus}
    For $\alpha$ defined in Assumption \ref{hyp:observableregu1}, the coefficients $\tilde{b}:\R^d\times \mathcal{P}_\alpha(\R^d) \rightarrow \R^d$ and $\tilde{\sigma}: \R^d \times \mathcal{P}_\alpha(\R^d)\rightarrow \R^{d\times d}$ satisfy (B.1) and either (B.2) or (B.2'):
    \begin{enumerate}
        \item[(B.1)] $\tilde{b}$ and $\tilde{\sigma}$ are locally Lipschitz,  which means that for all $M>0$, it exists $C_M\geq 0$ such that for all $x,x'\in B(0,M)$ and $\mu,\mu'\in B_{\mathcal W_\alpha}(\delta_0,M)$:
    \begin{equation*}
        |\tilde{b}(x,\mu)-\tilde{b}(x',\mu')|+|\tilde{\sigma}(x,\mu)-\tilde{\sigma}(x',\mu')|\leq C_M\bigl(|x-x'|+\mathcal W_\alpha(\mu,\mu')\bigr)
    \end{equation*}
    \item[(B.2)] $|\tilde{b}|^2$ and $|\tilde{\sigma}|^2$ have at most a polynomial growth of order $\alpha$ in position: for all $M>0$, it exists $C_M\geq 0$ such that for all $x\in \R^d$ and $\mu\in \mathcal{P}_\alpha(\R^d)$ such that $m_\alpha(\mu) \leq M$:
    \begin{equation*}
        |\tilde{b}(x,\mu)|^2+|\tilde{\sigma}(x,\mu)|^2 \leq C_M ( 1+ |x|^\alpha)
    \end{equation*}
    with $m_\alpha(\mu)=\int_{\R^d} |x|^\alpha \mu(dx)$
    \item[(B.2')] Let $\beta\geq0$ such that $\nabla f$ as at most a polynomial growth of order $\beta$. Then, $\tilde{\sigma}$ must have at most a controlled polynomial growth and $\tilde{b}$ must be coercive and must have at most a polynomial growth of order $q>\alpha$: it exists $q>\alpha$ such that for all $M>0$, it exists $c_M, \, C_M, \, \bar{C}_M > 0$ such that for all $x\in \R^d$ and $\mu\in \mathcal{P}_\alpha(\R^d)$ such that $m_\alpha(\mu) \leq M$:

    \begin{equation*} 
        \left\{
        \begin{array}{lll}
            |\tilde{\sigma}(x,\mu)|^2 \leq C_M ( 1+ |x|^{2 \beta}) \\
            |\tilde{b}(x,\mu)|\leq C_M ( 1+ |x|^{q})\\
            \langle \tilde{b}_\mu(x), x\rangle \leq -c_M|x|^{q}+ \bar{C}_M
        \end{array}
        \right.
        \end{equation*}
    Note that the value of $\beta$ only depend on which assumption is verified by $f$, and respectively for (A.1), (A.2) and (A.3), $\beta=1$ $\beta=0$ and $\beta=\alpha_2+1$
    \end{enumerate}

\end{assu}

 Defining a solution of an SDE up to a stopping time $\tau$ is standard in the literature \cite{hsu2002stochastic}, and this notion can easily be adapted for conditional McKean Vlasov equations. 

\begin{defi} \label{def:soljusqueT.A} (Strong Solution up to a stopping time) 
    Given a canonical probability set-up, an $\mathcal{F}^1_0$ measurable initial condition $X_0$ with values on $\R^d$ verifying $\mathbb{E}(|X_0|^\alpha)< \infty$, a $\mathbb{F}^{W^0}$ stopping time $\tau$ and $T\geq 0$, we call a strong solution to the conditional McKean Vlasov equation with coefficients $b, \tilde{b}, \sigma,   \tilde{\sigma}$ up to the stopping time $\tau$ an $\mathbb{F}^{X_0,W,W^0}$ adapted process $(X_{t\wedge \tau})_{0\leq t \leq T}$ with continuous paths  in $\R^d$ verifying the following conditions:
    \begin{enumerate}
        \item It exists a sequence of $\mathbb{F}^{W^0}$ stopping times $(\tau^M)_{M\in \N^*}$ such that $\tau^M \uparrow \tau$ a.s
        \item For all $M \in \N^*$, $\mathbb{E}[\hspace{2pt} \int_0^{T\wedge \tau^M} (|b(X_t,\mu_t)|+ |\tilde{b}(X_t,\mu_t)|+ |\sigma(X_t,\mu_t)|^2+|\sigma^0(X_t,\mu_t)|^2)dt ] < \infty$
        \item For all $M \in \N^*$, $\mathbb{P}$ a.s, for all $t\in [0,T]$:  
        \begin{multline*}
            X_{t\wedge \tau^M}=X_0+\int_0^{t\wedge \tau^M} \big( b(X_s,\mu_s)+\tilde{b}(X_t,\mu_t) \big)ds+\int_0^{t\wedge \tau^M}\sigma(s,X_s,\mu_s)dW_s\\+\int_0^{t\wedge \tau^M}\sigma^0(s,X_s,\mu_s)dW^0_s
        \end{multline*}
        
    \end{enumerate}
    and for all $M\in \N^*$, $(\mu_{t \wedge \tau^M})_{0\leq t \leq T}$ is $\mathbb{F}^{W^0}$ adapted and is the regular flow of conditional distribution associated to $(X_{t\wedge \tau^M})_{0\leq t \leq T}$ in $\mathcal{P}_\alpha(\R^d)$ (Proposition \ref{prop:existcondlaw}), with $\mathbb{F}^{X_0,W,W^0}$ (resp. $\mathbb{F}^{W^0}$) be the right continuous and completed augmentation of the filtration generated by $X^0,W,W^0$ (resp. $W^0$).
\end{defi}

It is also possible to adapt the notion of strong uniqueness for explosive solutions.

\begin{defi}(Strong uniqueness up to a stopping time)
    Given a canonical probabilistic set-up $(\Omega,\mathcal{F},\mathbb{F},\mathbb{P})$ and a $\mathbb{F}$ stopping time $\tau$, we say that strong uniqueness holds up to $\tau$ for a conditional McKean Vlasov equation with coefficients $b,\sigma,\sigma^0$ of order $\alpha\geq1$ if for any initial condition $X_0$ verifying $\mathbb{E}(|X_0|^\alpha)< \infty$, if $(X_{t\wedge \tau})_{0\leq t \leq T}$ and $(\tilde{X}_{t\wedge \tau})_{0\leq t \leq T}$ are two solutions respectively associated with the sequences of $\mathbb{F}$ stopping times $(\tau_M)_{M \in \N^*}$ and  $(\tilde{\tau}_M)_{M \in \N^*}$, then $$\mathbb{P} \text{    a.s, for all   } t \in [0,T] \text{      and   } M \in \N^*, \quad X_{t\wedge \tau^M \wedge\tilde{\tau}^M}=\tilde{X}_{t\wedge \tau^M \wedge\tilde{\tau}^M} $$
\end{defi}

Whenever strong existence and uniqueness holds on a canonical set-up $(\Omega, \mathcal{F},\mathbb{F}, \mathbb{P})$ until a stopping time $\tau$, we say that the equation is strongly uniquely solvable until $\tau$ on the set-up $(\Omega, \mathcal{F},\mathbb{F}, \mathbb{P})$. The goal of this section is to justify the well posed-ness of the stochastic moment dynamic up to a blowing time $\tau$, and the following result is the main result of this section.

\begin{prop} \label{prop:stochasticmomentswelldefined}
    Suppose that the observable $f$ verifies Assumption~\ref{hyp:observableregu1},that the coefficients $a$ and $s$ verify Assumption~\ref{hyp:coeffaets} and that $\tilde{b}$ and $\tilde{\sigma}$ verify Assumption \ref{hyp:driftetdiffusionenplus}. Let $\beta\geq 0$ such that $\nabla f $ has at most a polynomial growth of order $\beta$ (Assumption \ref{hyp:driftetdiffusionenplus}). Suppose that the initial condition has finite moments of all orders\footnote{this condition can be weakened if necessary but is made to simplify the redaction of the proof}.
    \medskip{}
    
    Then, given a canonical probabilistic set-up $(\Omega, \mathcal{F}, \mathbb{F},\mathbb{P})$, and an $\mathcal{F}_0^1$ measurable initial condition , and the coefficients $b:\R^d\times \mathcal{P}_\alpha(\R^d) \rightarrow \R^d$ and $\sigma:\R^d\times \mathcal{P}_\alpha(\R^d) \rightarrow \R^{d\times p}$ defined by equation \eqref{eq:defcoefs}, the conditional McKean Vlasov equation \eqref{eq:stomomentdyn} is strongly uniquely solvable until the blowing time $\tau$ defined as $\tau=\underset{M\rightarrow \infty}{\lim}\tau^M$ and:
\begin{equation}
\label{def:tauM_explosion}
    \tau^{M}=\inf\left\{t\geq 0,  \quad \mu_t(|x|^\alpha ) \geq M , \quad d(\mu_t(f), \mathcal{S})\leq \frac{1}{M}, \quad \det(\mu_t(\nabla f ^T \nabla f))\leq \frac{1}{M} \right\}\,.
\end{equation}

\end{prop}

This result justifies that the process $(X_t,\mu_t)$ can be defined as long as it does not explode and does not reach the singularity set of the coefficients $b$ and $\sigma$. 

\medskip{}

The proof can be summarized as follows.

\begin{enumerate}
    \item We will define a sequence of increasingly larger domains $\mathcal{D}^{K,M}\subset \R^d \times \mathcal{P}_\alpha(\R^d)$ in which the coefficients $b$ and $\sigma$ will be Lipschitz and bounded. Each domain will have the form $\mathcal{D}^{K,M}=B(0,K)\times \Lambda^M$, with $B(0,K)$ a Euclidean ball on $\R^d$ and $\Lambda^M$ a bounded subset of $\mathcal{P}_\alpha(\R^d)$.
    \item For every domain $\mathcal{D}^{K,M}$, it will be possible to build regular coefficients $b^{K,M}$, $\tilde{b}^{K,M}$, $\sigma^{K,M}$ and $\tilde{\sigma}^{K,M}$  as truncations of the coefficients $b, \tilde{b}, \sigma$ and $\tilde{\sigma}$. These regular coefficients will coincide with their original versions on $\mathcal{D}^{K,M}$ and will be globally Lipschitz and bounded in $\R^d\times \mathcal{P}_\alpha(\R^d)$, ensuring strong existence and uniqueness of a solution $X_t^{K,M},\mu_t^{K,M} $ to the conditional McKean Vlasov equation directed by the coefficients $b^{K,M}$, $\tilde{b}^{K,M}$, $\sigma^{K,M}$ and $\tilde{\sigma}^{K,M}$.
    \item We can then define the sequences of exit times of $\mathcal{D}^{K,M}$ as the infimum of two stopping times $\gamma^{K,M}$ and $\tau^{K,M}$, representing respectively the exit times of $X^{K,M}_t$ from $B(0,K)$ and $\tau^{K,M}$ the exit time of $\mu^{K,M}_t$ from $\Lambda^M$. We will see that when $K$ tends to infinity, $\gamma^{K,M}> \tau^{K,M}$ a.s, meaning that the blowing time of the system can be characterized by the blowing time of the mean field limit.
    \item We can finally use a localization procedure to extend the solutions $X_t^{K,M},\mu_t^{K,M}$ until the blowing time $\tau$ expressed as a limit of stopping times $\tau^M$ only depending on the common noise. 
    
\end{enumerate}

As explained in the introduction and on the previous section, it is well known that in the non explosive case the flow of conditional laws $(\mu_t)$ solves the stochastic Fokker-Plank equation \eqref{eq:stochasticFP}. In our study, we will see that this result is also true up to a stopping time $\tau^M$ ensuring that the solution of the conditional McKean Vlasov equation has not exploded yet and is a corollary of the well posed-ness of the equation \eqref{eq:stomomentdyn}. It is worth noticing that this result is true because the sequence $(\tau^M)_{M\in \N^*} $ reducing the solution is made of $\mathbb{F}^{W^0}$ stopping times, and so the conditional Fubini theorem ensures that: $$(\mathcal{L}(X_{t\wedge \tau^M} | \mathcal{F}^0))_{0\leq t \leq T}=(\mathcal{L}(X_{t} | \mathcal{F}^0))_{0\leq t \leq \tau^M \wedge T}$$ The result won't be true if the stopping times could depend on the private sources of noise, even for the classical McKean Vlasov equations, as explained in \cite{hong2024mckean}. Then, under a weak assumption on the interaction kernel $W$ and the potential $U$ ensuring that the associated drift verify Assumption \ref{hyp:driftetdiffusionenplus}, the flow of conditional laws defined in \ref{eq:stomomentdyn} will be solution to the associated stochastic Fokker Plank equation. This equation can be interpreted as noisy version of the Wasserstein gradient flow associated to $W$ and $U$. The perturbation will correspond to the effect of the stochastic moment dynamic on the evolution of the flow.

\medskip

In order to ensure that coefficients $\tilde{b}$ and $\tilde{\sigma}$ verify Assumption \ref{hyp:driftetdiffusionenplus} in the context of the granular media equation, the associated potential $U$ and interaction kernel $W$ must verify the following assumption:

\begin{assu} \label{hyp:potentiels}
    $\nabla W$ and $\nabla U$ are locally Lipschitz, and verify one of the following assumptions:

    \begin{enumerate}
        \item[(C)] $\nabla W$ and $\nabla U$ have at most a polynomial growth of order $\alpha$, which means that there exist $C\geq0$ such that for all $x\in \R^d$:
        \begin{equation*}
            |\nabla U(x)| +|\nabla W(x)| \leq C(1+|x|^\alpha)
        \end{equation*}
        \item[(C')]  Let $\beta\geq0$ such that $\nabla f$ as at most a polynomial growth of order $\beta$. Then $\nabla U$ is coercive of order $q$ and $\nabla W$ as at most a polynomial growth of order strictly inferior to $q$, which means that there is $\lambda,C>0$ and $0\leq q' < q$ such that for all $x\in \R^d$:
        \begin{equation*}
        \langle x, \nabla U(x) \rangle \geq \lambda |x|^{q}-C \quad \text{and} \quad |\nabla W(x)| \leq C(1+|x|^{q'})
        \end{equation*}
    \end{enumerate}
\end{assu}
This assumption ensures that the coefficients $\tilde{b}$ and $\tilde{\sigma}$ associated to the granular media equation \eqref{eq:granularmedia}, verify the conditions of Assumption \ref{hyp:driftetdiffusionenplus} (depending on which assumption is verified).

 \medskip{}

 We can now present the result that formally states that the stochastic Fokker-Plank equation associated to the Stochastic Moment Dynamic \eqref{eq:stomomentdyn} is well defined up to the blowing time $\tau$ defined in Proposition \ref{prop:stochasticmomentswelldefined}. In particular, the coefficients associated to the double well potential and quadratic interaction verify assumption (C') with $q=4$.
 
\begin{cor} \label{cor:edps}
    Suppose that the observable $f$ verifies Assumption~\ref{hyp:observableregu1}, that the coefficients $a$ and $s$ verify Assumption~\ref{hyp:coeffaets}, and that the interaction kernel $W$ and the potential $U$ verify Assumption~\ref{hyp:potentiels}. Let $\tilde{b}$ and $\tilde{\sigma}$ be the coefficients associated to the granular media equation and defined by:
    \begin{equation} \label{eq:tildebetsigma}
    \tilde{b}, \tilde{\sigma}:\R^d\times  \mathcal{P}_{\alpha}(\R^d)\rightarrow \R^d \hspace{1cm} \tilde{b}(x,\mu)=-\nabla U(x) -\mu \ast \nabla W(x) \hspace{1cm} \tilde{\sigma}(x,\mu)=\tilde{\sigma}
\end{equation}    
    Let $X,\mu$ be the solution up to the blowing time $\tau$ of the conditional McKean Vlasov equation \eqref{eq:stomomentdyn}. Then $\mu$ is solution to the stochastic Fokker Plank equation up to the blowing time $\tau$, meaning that for all $M\in \N^*$, $\mathbb{P}^0$ a.s, for all $0\leq t \leq T$:
    \begin{equation}
    \begin{split}
        d \mu_{t\wedge \tau^M}=& \na \cdot \po \na U+ \na W \ast \mu_{t \wedge \tau^M} \pf \dd t + \frac{\tilde{\sigma}^2}{2} \Delta \mu_{t\wedge \tau^M}  \dd t  \\
        &-\na \cdot \po b_{\mu_t} \mu_t\pf \dd t -  \na \cdot \po    \mu_t \sigma_{\mu_t}\pf   dW_t^0+ \frac{1}{2}\sum_{i,j=1}^d \partial_{i,j}^2 \po  (\sigma_{\mu_t}\sigma_{\mu_t}^T)_{i,j} \mu_t \pf  \dd t \,.
    \end{split}
    \end{equation}
    where this equation  must be understood in the weak sense against any test function in $\mathcal{C}^2_b(\R^d)$.
\end{cor}

\begin{proof}
    The proof of this result directly follows from the proof of Theorem 1.9 in \cite{hammersley2021weak} and the fact that $\tau^M$ is a $\mathbb{F}^{W^0}$ stopping time. Indeed, because we have that for all $0\leq s\leq T$, the event $\{s\leq \tau^M\} \in \mathcal{F}^0$, and so for any $\varphi\in \mathcal{C}^2_b$:
    \begin{equation*}
        \begin{split}
            \mathbb{E}\po \int_0^{t\wedge \tau^M} \nabla \varphi(X_s) dW_s \big| \mathcal{F}^0 \pf &=\mathbb{E}\po \int_0^{t} \1_{s\leq \tau^M} \nabla \varphi(X_s) dW_s \big| \mathcal{F}^0 \pf\\
            &=\1_{s\leq \tau^M} \mathbb{E}\po \int_0^{t}  \nabla \varphi(X_s) dW_s \big| \mathcal{F}^0 \pf
        \end{split}
    \end{equation*}
    With this observation (also valid for the deterministic integrals and the one against $dW_s^0$), one can easily adapt the proof from \cite{hammersley2021weak} and prove the desired result.
\end{proof}

We can now present a propagation of chaos result for the solution of the conditional McKean Vlasov equation \eqref{eq:stomomentdyn}. Because of its explosive behavior (and of the associated particle systems), we must adapt the statement of Proposition \ref{prop:propchaos}, because intuitively it is not possible to ensure the convergence of the particle system at times arbitrarily close to the blowing time. However, conditionally on the non-explosion of the non linear process, we will be able to ensure the non explosion of the particle system and its convergence when the number of particles goes to infinity. In order to state this result, we slightly adapt Definition \ref{def:setuppropchaos} in order to present a clear set-up for the propagation of chaos.

\begin{defi}(Set-up for the conditional propagation of chaos)\label{def:setuppropchaoscond}. 

Suppose that the observable $f$ verifies Assumption~\ref{hyp:observableregu1},that the coefficients $a$ and $s$ verify Assumption~\ref{hyp:coeffaets} and that $\tilde{b}$ and $\tilde{\sigma}$ verify Assumption \ref{hyp:driftetdiffusionenplus}. Let $\beta\geq 0$ such that $\nabla f $ has at most a polynomial growth of order $\beta$ (Assumption \ref{hyp:driftetdiffusionenplus}),  let $\alpha\geq 2$ bet the parameter associated to $f$ (Assumption \ref{hyp:observableregu1}), and let $b,\sigma$ defined by equation \eqref{eq:coeffsp}.

Let $(\Omega^0,\mathcal{F}^0,\mathbb{P}^0)$ and $(\Omega^1,\mathcal{F}^1,\mathbb{P}^1)$ be two complete probability spaces endowed with two complete and right continuous filtrations $\mathbb{F}^0=(\mathcal{F}^0_t)_{0 \leq t \leq T}$ and $\mathbb{F}^1=(\mathcal{F}^1_t)_{0 \leq t \leq T}$. Let $W^0=(W_t^0)_{0\leq t \leq T}$ be a $p$ dimensional $\mathbb{F}^0$ Brownian motion defined on $\Omega^0$ and let $(W^n)_{n\geq 1}$ be a sequence of independent $d$ dimensional $\mathbb{F}^1$ Brownian motion defined on $\Omega^1$. Let $(X_0^n)_{n \geq 1}$ be a sequence of independent and identically distributed $\mathcal{F}_0^1$ random variables defined on $\Omega^1$, with finite moments of all orders\footnote{Again, this assumption can be weakened if necessary}. Let $\Omega$ be the canonical probabilistic set-up associated to the two probabilistic spaces $(\Omega^0,\mathcal{F}^0,\mathbb{P}^0)$ and $(\Omega^1,\mathcal{F}^1,\mathbb{P}^1)$.

Because of Proposition \ref{prop:stochasticmomentswelldefined}, it is possible to define for any $n\geq 1$ $\bar{X}^n=(\bar{X}_{t\wedge \tau^n}^n)_{0\leq t \leq T}$ the solution of the conditional McKean Vlasov equation defined on $(\Omega,\mathcal{F},\mathbb{F},\mathbb{P},X_0^n,W^n,W^0)$ and $\bar{\mu}_{t \wedge \tau^n }^n$ the regular flow of conditional distribution associated, up to the blowing time $\tau_n:=\underset{M\rightarrow \infty}{\lim} \tau^{M}_n $, $\tau^M_n=\inf\{t\geqslant 0, \ \bar{\mu}_t^n\notin  \Lambda^M\}$. For all $n\geq 1$, the process $(\bar{\mu}^n)$ is indistinguishable from $(\bar{\mu}^1)$, so we can use the notation $\bar{\mu}$ for the mean field distribution common for all typical particles $\bar{X}^n$ and $\tau$ the associated blowing time.

Moreover, for all $N\geq 1$, the particle system \eqref{eq:commonnoisePS} made of $N$ particles $(X^{n,N})_{1\leq n\leq N}$ is also strongly uniquely solvable up to a blowing time $\tau^{(N)}$ defined as $\tau^{(N)}:=\underset{M\rightarrow \infty}{\lim} \tau^{(N)}_M $, $\tau^{(N)}_M=\inf\{t\geqslant 0, \ \pi^N_t\notin \Lambda^M\}$ on the probabilistic set-up $( \Omega,\mathcal{F},\mathbb{F},\mathbb{P},(X^{n,N}_0)_{1\leq n\leq N}, (W^k)_{0\leq k\leq N})$, with $ \pi^N_t=\pi((X_t^{n,N})_{1\leq n \leq N})=\frac{1}{N}\sum_{n=1}^{N}\delta_{X_t^{n,N}}$ the empirical distribution of the particle's system.
\end{defi}

The justification of the statements made in this definition is presented in the next section: they are consequences of the classical results for Lipschitz coefficients and of Proposition \ref{prop:stochasticmomentswelldefined}. We can now state the propagation of chaos result.

\begin{prop}\label{prop:PoC}
Let $0<t\leq T$ be such that $\mathbb P \po \tau>t\pf >0$. Then,
\begin{equation}
\label{eq:tauN->tau}
\mathbb P \po \tau^{(N)} \leqslant \tau \ |\ \tau>t\pf \underset{N\rightarrow \infty}\longrightarrow 0
\end{equation}
and, for any $\delta>0$,
\[\mathbb P \po \left. \tau^{(N)} > t\text{ and } \sup_{s\in[0,t]} \mathcal W_\alpha(\pi^N_t,\bar{\mu}_t) \geqslant \delta  \ \right|\ \tau>t\pf \underset{N\rightarrow \infty}\longrightarrow 0\,.\]

\end{prop} 

The proof of this result is presented is the next section, and rely on the propagation of chaos for Lipschitz coefficients and on the localization procedure behind the proof of Proposition \ref{prop:stochasticmomentswelldefined}.

\section{Proofs of Section \ref{sec:def-process}}\label{sec:proof-main}

\subsection{Proof of Proposition \ref{prop:stochasticmomentswelldefined}}

The first step of the proof consists in identifying bounded subsets of $\R^d\times \mathcal{P}_\alpha(\R^d)$ in which the coefficients $b$ and $\sigma$ will be Lipschitz and bounded. These subsets will be indexed by two parameters (one for the position and one for the mean field distribution), and are increasingly bigger as the parameters tend to infinity. The solution of the equation will be defined as long as it stays in one of these subsets, which naturally leads to an explosive behavior and a blowing time corresponding to the exit time of all of these subsets.

\begin{lem} \label{lem:regucoefsursomaines}
    Suppose that the observable $f$ verifies Assumption \ref{hyp:observableregu1} with an associated parameter $\alpha$ depending on its regularity, and that the coefficients $a$ and $s$ verify Assumption~\ref{hyp:coeffaets}, with $\mathcal{S}$ the set of singularities of the coefficients $a$ and $s$. For all $M> 0$ and $K> 0$, we define the following domains:
    \begin{equation*}
        \mathcal{D}^{K,M}:=\{ x,\mu \in \R^d \times \mathcal{P}_\alpha(\R^d): \hspace{3pt} |x|< K, \hspace{3pt} m_\alpha(\mu ) < M , \hspace{3pt} d(\mu(f), \mathcal{S})> \frac{1}{M}, \hspace{3pt} \mathrm{det}(\mu(\nabla f ^T \nabla f)) > \frac{1}{M} \}
    \end{equation*}
    Then the coefficients $b:\R^d\times \mathcal{P}_\alpha(\R^d) \rightarrow \R^d$ and $\sigma:\R^d\times \mathcal{P}_\alpha(\R^d) \rightarrow \R^{d\times p}$ defined by:
    \begin{equation} 
     \left\{
        \begin{array}{lll}
        \sigma_{\mu}(x) &=& \nabla f (x)\po \mu[\nabla f^T \nabla f]\pf ^{-1} s(\mu(f))\\
    
        b_{\mu}(x) &=& \nabla f (x) \po \mu[\nabla f^T \nabla f]\pf ^{-1} [a(\mu(f))-\frac{1}{2}\mu(\sigma_{\mu}\sigma_{\mu}^T :\nabla^2 f)] \,.
        \end{array}\right.
\end{equation}
are Lipschitz and bounded on $\mathcal{D}^{K,M}$ for all $K> 0$ and $M> 0$.
\end{lem}

The assumption made on the observable $f$ and its derivatives ensures that in the appropriate Wasserstein space, functions of type $\mu \mapsto \mu(g)$ with $g$ based on $f$ and its derivatives are locally Lipschitz and bounded outside of the singularities in the coefficients. The domains $\mathcal{D}^{K,M}$ are constructed to be contained on bounded subsets of $\R^d \times \mathcal{P}_\alpha(\R^d)$, and to be distant from the singularities of $b$ and $\sigma$ in $\mu$ (the inverse of the matrix $\mu(\nabla f^T\nabla f)$ is not necessarily defined and $\mu(f)$ must not be in the set of singularities of $a$ and $s$). The technical details  of the proof of this lemma is  deferred to Appendix~\ref{app:proof1stlemma}.

\medskip{}

From this lemma and Assumption \ref{hyp:driftetdiffusionenplus}, we can multiply the coefficients $b,\tilde b,\sigma,\tilde \sigma$ by a cut-off function $\chi^{K,M}$ (a smooth approximations of the   indicator functions of the domains $\mathcal{D}_{K,M}$) in order to get the following:


\begin{lem} \label{lem:troncatures}
In the settings of Lemma~\ref{lem:regucoefsursomaines}, assuming furthermore Assumption \ref{hyp:driftetdiffusionenplus}, let $K>0$ and $M>0$. We can build truncated versions of the coefficients named $b^{K,M}$, $\tilde{b}^{K,M}$, $\sigma^{K,M}$, $\tilde{\sigma}^{K,M}:\R^d \times \mathcal P_\alpha(\R^d) \rightarrow \R^d$ that coincide with $b, \tilde{b},  \sigma , \tilde{\sigma}$ on $\mathcal{D}^{K,M}$ and that are globally Lipschitz and bounded.
\end{lem}

The details are deferred to  Appendix \ref{app:prooftroncature}.  

\medskip{}

We can now use the existence and uniqueness result for regular coefficients presented in Proposition \ref{prop:existcoefflipschitz}. Because for all $K>0$ and $M>0$ the coefficients $b^{K,M}$, $\tilde{b}^{K,M}$, $\sigma^{K,M}$ and $\tilde{\sigma}^{K,M}$ verify Assumption \ref{hyp:lipschitzbornés}, the conditional McKean Vlasov equation:
\begin{equation} \label{eq:condMVwithKandM}
    dX_t=[b^{K,M}(X_t,\mu_t)+\tilde{b}(X_t,\mu_t)]dt+\tilde{\sigma} (X_t,\mu_t) dW_t+\sigma^{K,M}(X_t,\mu_t)dW_t^0
\end{equation}
 is strongly uniquely solvable on $[0,T]$. We note $(X_t^{K,M})_{0\leq t \leq T}$ its unique strong solution and $(\mu_t^{K,M})_{0\leq t \leq T}$ its regular version of its conditional law knowing $\mathcal{F}^0$. We define the following stopping times:
\begin{align*}
    \gamma^{K,M}&= \text{inf}\{0 \leq  t \leq T, \quad |X_t^{K,M}| \geq  K\}\\
    \tau^{K,M}&= \text{inf}\{0 \leq  t \leq T, \quad \mu_t^{K,M} \notin \Lambda^M\}
\end{align*}
and 
$$\gamma^{M}= \underset{K\rightarrow \infty}{\text{lim}}\gamma^{K,M} \hspace{1cm} \tau^{M}= \underset{K\rightarrow \infty}{\text{lim}}\tau^{K,M}.$$
Since  $X^{K,M}$ is $\mathbb{F}^{X_0,W^0,W}$ adapted with continuous paths and because $B(0,K)^c$ is closed, $\gamma^{K,M}$ (and by taking the limit also $\gamma^M$) is a $\mathbb{F}^{X_0,W^0,W}$ stopping time. With the same argument, because $\mu^{K,M}$ is $\mathbb{F}^{W^0}$ adapted with continuous paths and $(\Lambda^M)^c$ is closed, $\tau^{K,M}$ and $\tau^{M}$ are $\mathbb{F}^{W^0}$ stopping times. By construction, the stopping times $\gamma^{K,M}$ and $\tau^{K,M}$ are non decreasing in $K$ and $M$, and the stopping time $\gamma^{K,M} \wedge \tau^{K,M}$ correspond to the first time when the process $(X^{K,M},\mu^{K,M})$ leaves the domain $\mathcal{D}^{K,M}$. It is also worth noticing that if we compare what happens for index $0<M \leq M'$ and $0<K \leq K'$, we have that $\mathbb{P}$ a.s, for all $t\leq T$:

\begin{equation} \label{eq:unicitejusqueTA}
    X^{K,M}_{t \wedge \gamma^{K,M} \wedge \tau^{K,M}  }= X^{K',M'}_{t \wedge \gamma^{K,M} \wedge \tau^{K,M}  } \quad \text{and} \quad \mu^{K,M}_{t \wedge \gamma^{K,M} \wedge \tau^{K,M}  }= \mu^{K',M'}_{t \wedge \gamma^{K,M} \wedge \tau^{K,M}  }
\end{equation}

This follows from the fact that $(X^{K,M},\mu^{K,M})$ and $(X^{K',M'},\mu^{K',M'})$ are both solutions to Equation~\eqref{eq:condMVwithKandM} until the stopping time $\gamma^{K,M} \wedge \tau^{K,M}$, and that this equation is uniquely solvable. The idea behind the localization procedure is then to let $K$ and $M$ to go to infinity, and to construct the solution as the limiting behaviors of these stopped processes. 

\medskip{}

The key idea for the following is that it is not really necessary to control the explosion in $x$ (corresponding to the effect of the stopping time $\gamma^{K,M}$), because knowing that the mean field limit has not exploded yet is enough to ensure that a typical particle cannot explode by itself, and so necessarily if the evolution explodes it can be described at the mean field level. To formalize this idea, we need to justify bounds on the moments of $X_t^{K,M}$ that are independent of the $t$ and $K$.

\begin{lem}\label{lem:boundmoments}
    The following inequalities are verified by $X^{K,M}$. 

    \begin{enumerate}
        \item Suppose that $\tilde{b} $ satisfy assumption (B.2'). Then, for all $M,r>0$, there exists $C>0$ such that for all all $K>0$:
        \begin{equation*}
            \mathbb{E}\po \underset{t \in [0,T]}\sup|X^{K,M}_{t\wedge \tau^{K,M}\ \wedge \gamma^{K,M}}|^r\pf  \leq C
        \end{equation*}
        
        \item Suppose that $\tilde{b} $ satisfy assumption (B.2).  Then, for all $M>0$, there exists $C>0$ such that for all all $K>0$:
        \begin{equation*}
            \mathbb{E}\po \underset{t\in [0,T]} \sup|X^{K,M}_{t\wedge \tau^{K,M} \wedge \gamma^{K,M} }|^2  \pf  \leq C
        \end{equation*}
    \end{enumerate}
\end{lem}

In order to clarify the presentation of the overall proof of the main results, the proof of this lemma is postponed to Appendix \ref{app:proofmomentbounds}. The heart of this result is that the bounds does not depend on the parameter $K$, and so we will be able to make this parameter tend to infinity and keep a localization depending only on the behaviour of the mean field limit. Indeed, the second bound will allow us to justify that a particle $X^{K,M}$ cannot go at infinity in finite time as long as the mean field limit $\mu^{K,M}$ is controlled. In other words, the effect of the stopping times $\gamma^{K,M}$ vanishes when $K$ goes to infinity. This idea is formalized by the following lemma.

\begin{lem} \label{lem:relationsT.A}
   Let $M>0$. Then: $$\mathbb{P}\po \bigcup_{K\in \N^*} \bigcap_{k \geq K} \{ \gamma^{k,M} > \tau^{k,M}\wedge T \}\pf =1$$
 Moreover,  the sequence $(\tau^{K,M})_{K\in \N^*}$ becomes constant after a certain term $\mathbb{P}$ a.s.
\end{lem}

\begin{proof}
 Let $M>0$. From Markov's inequality and Lemma~\ref{lem:boundmoments},
    \[
           \mathbb{P}( \gamma^{K,M}\leq \tau^{K,M} \wedge T) =  \mathbb{P}(\underset{t\in[0,T]}\sup |X^{K,M}_{t\wedge \tau^{K,M}\wedge \gamma^{K,M}}|^2 \geq K^2) \leq \frac{C}{K^2}
       \]
    The first part of the lemma then follows from Borel Cantelli's lemma.

 We now turn to the proof of  the second statement. For $M>0$, we define the $\mathcal{F}$ random variable $\mathcal{K}_M$ corresponding to the first index $K$ (random), such that for all $k\geq K$, $T\wedge \tau^{k,M} < \gamma^{k,M}$, namely:
 $$\mathcal{K}_M=\inf\{K \in \N^*, \, \forall k \geq K,  \,T \wedge \tau^{k,M}<\gamma^{k,m}\}$$
 The previous result justify that $\mathcal{K}_M < \infty$ almost surely. We are going to show that the sequence $(\tau^{K,M})_{K\in \N^*}$ is constant after the rank $\mathcal{K}_M$.

 Let $k\in \N^*$, we are going to show that almost surely $\tau^{\mathcal{K}_M+k,M}=\tau^{\mathcal{K}_M,M}$. By construction, we already know that $\tau^{\mathcal{K}_M+k,M}\geq \tau^{\mathcal{K}_M,M}$. Moreover, by equation \eqref{eq:unicitejusqueTA}, $\mathbb{P}$ a.s, for all $K,k \in \N^*$:

 \begin{equation} \label{eq:constantapcr}
     \mu_{T\wedge \tau^{K,M} \wedge \gamma^{K,M}}^{K+k,M}=\mu_{T\wedge \tau^{K,M} \wedge \gamma^{K,M}}^{K,M}
 \end{equation}

Then, if we define the sequence of random variables $(\nu_n)_{n\in \N^*}$ by $\nu_n=\mu_{T\wedge \tau^{n,M} \wedge \gamma^{n,M}}^{n,M}$, by composition the sequence $(\nu_{\mathcal{K}_M+k})_{k\in \N^*}$ is also a sequence of random variables which is a.s constant by Equation \ref{eq:constantapcr}. Because of the definition of $\mathcal{K}_M$, we know that $T\wedge \tau^{\mathcal{K}_M,M} \wedge \gamma^{\mathcal{K}_M,M}=T\wedge \tau^{\mathcal{K}_M,M} $, and so:

 \begin{equation} \label{eq:comparaisonmu}
 \mu_{T\wedge \tau^{\mathcal{K}_M,M}}^{\mathcal{K}_M+k,M}=\mu_{T\wedge \tau^{\mathcal{K}_M,M}}^{\mathcal{K}_M,M}
 \end{equation}
 
 We can now distinguish cases depending on the value of $T\wedge \tau^{\mathcal{K}_M,M} $.  If $T < \tau^{\mathcal{K}_M,M} $, then $\tau^{\mathcal{K}_M,M}=\infty$ and the inequality $\tau^{\mathcal{K}_M,M}\geq \tau^{\mathcal{K}_M+k,M}$ is automatically verified. Otherwise, $T \geq \tau^{\mathcal{K}_M,M} $ and because $(\mu^{\mathcal{K},M}_t)$ has continuous paths and $(\Lambda^M)^c$ is closed in $\mathcal{P}_\alpha(\R^d)$, we know that $\mu_{T\wedge \tau^{\mathcal{K}_M,M}}^{K_M,M} \in  (\Lambda^M)^c$. Then, equation \eqref{eq:comparaisonmu} implies that $\mu_{T\wedge \tau^{\mathcal{K}_M,M}}^{\mathcal{K}_M+k,M} \in (\Lambda^M)^c$ and then by definition of $\tau^{\mathcal{K}_M+k,M}$, almost surely $\tau^{\mathcal{K}_M+k,M} \leq \tau^{\mathcal{K}_M,M}$, which concludes the proof.

\end{proof}

The previous Lemma is at the heart of the overall proof because it justifies that the sequence of processes $(\mu^{K,M}_{\cdot \wedge \tau^{K,M}\wedge \gamma^{K,M}})_{K\in \N^*}$ (and also $(X^{K,M}_{\cdot \wedge \tau^{K,M}\wedge \gamma^{K,M}})_{K\in \N^*}$) becomes constant after a certain rank, and the limit corresponds to the solution of the conditional McKean Vlasov equation up to a stopping time $\tau^M$ only depending on the conditional law. We can now use a localization procedure to define the solution of the conditional McKean Vlasov equation directed by the coefficients $b$, $\tilde{b}$, $\sigma$, $\tilde{\sigma}$ until the blowing time $\tau=\underset{M\rightarrow\infty}{\text{lim}}\tau^M$

\begin{lem} \label{lem:lastlemma}
    We define the following processes:
    \begin{equation}
        X_{\cdot\wedge \tau^M}:=\underset{K\rightarrow\infty}{\text{lim}}X^{K,M}_{\cdot \wedge \tau^{K,M}\wedge \gamma^{K,M} }\hspace{1cm} \mu_{\cdot\wedge \tau^M}:=\underset{K\rightarrow\infty}{\text{lim}}\mu^{K,M}_{\cdot\wedge \tau^{K,M}\wedge \gamma^{K,M} }
    \end{equation}
    These processes are well defined and $(X_{t\wedge \tau^M}, \mu_{t\wedge \tau^M})_{0\leq t \leq T}$ is the unique solution to the conditional McKean Vlasov equation \eqref{eq:stomomentdyn} up to the blowing time $\tau=\lim_{M\rightarrow \infty} \tau^M$.
\end{lem}

Establishing this lemma is the last step of the overall proof. We first show that the processes $(X_{t\wedge \tau^M})_{0\leq t \leq T}$ and $(\mu_{t\wedge \tau^M})_{0\leq t \leq T}$ are well defined (and that they do not depend on the choice of $M$ up to time $\tau^M$), and then we show that $(X_{t\wedge \tau^M}, \mu_{t\wedge \tau^M})_{0\leq t \leq T}$ is the unique strong solution of the conditional McKean Vlasov equation \eqref{eq:stomomentdyn} until the blowing time $\tau$. Again, in order to clarify the presentation of the overall proof, the proof of this lemma is done in Appendix \ref{app:prooflastlemma}. This ends the proof of Proposition \ref{prop:stochasticmomentswelldefined}.

\subsection{Proof of Proposition \ref{prop:PoC}}

We begin the proof by justifying the statements made in Definition \ref{def:setuppropchaoscond}. For $K,M>0$, let $b^{K,M}, \sigma^{K,M}, \tilde{b}^{K,M}, \tilde{\sigma}^{K,M}$ be the regular truncations of the coefficients defined in Lemma~\ref{lem:troncatures}, and defined to coincide with the original coefficients on $\mathcal{D}^{K,M}$. These coefficients satisfy Assumption \ref{hyp:lipschitzbornés}, and we will make use of the results for Lipschitz coefficients presented in Section~\ref{sec:definition} for these coefficients. In the following, we will denote $\bar{X}^{n,K,M},\bar{\mu}^{n,K,M}$ the solution of the conditional McKean Vlasov equation with the truncated coefficients directed by the common noise $W^0$ and the personal bownian motion $W^n$. Because of Proposition \ref{prop:uniqueMFD}, the processes $(\bar{\mu}^{n,K,M})_n$ are indistinguishable, and thus so are the $(\bar{\mu}^{n})_n$ defined by their extension for $K,M \rightarrow \infty$ (this property is conserved through the localization procedure). Then, the blowing times $\bar{\tau}^n:=\underset{M\rightarrow \infty}{\lim}\underset{K\rightarrow \infty}{\lim} \bar{\tau}^{n,K,M}$ , $\bar{\tau}^{n,K,M}= \inf \{0\leq t \leq T , \quad \bar{\mu}_t^{n,K,M} \notin  \Lambda^M\}$ also coincide, justifying the statement made in Definition \ref{def:setuppropchaoscond}. This justifies the notation $\bar{\mu}=\bar{\mu}^1$ and $\bar{\tau}=\bar{\tau}^1$

Moreover, the map $$x \in (\R^d)^N \rightarrow \pi^N(x)=\frac{1}{N}\sum_{i=1}^N\delta_{x_i} \in \mathcal{P}_\alpha(\R^d)$$ is Lipschitz, so $x\mapsto b^{K,M}_{\pi^N(x)}(x)$ is also lipschitz by composition (and similarly for the other coefficients of the particle system). Then, the classical theory of SDE ensures that:
\begin{equation*} 
    dX_t^{i,N,K,M}=b^{K,M}_{\pi^N(X)}(X_t^{i,N}) \dd t+ \tilde{b}^{K,M}_{\pi^N(X)}(X_t^{i,N}) \dd t+\sigma^{K,M}_{\pi^N(X)}(X_t^{i,N}) \dd W_t^0 +\tilde{\sigma}^{K,M}_{\pi^N(X)}(X_t^{i,N}) \dd W_t^{i} \,,
\end{equation*}
is strongly uniquely solvable. A usual localization procedure ensures the well definition of the particle system until a blowing time $\tau^{(N)}:=\underset{K,M }{\sup} \tau^{(N),K,M}$ with 
\[\tau^{(N),K,M}=\inf \{0\leq t \leq T , \quad \pi_t^{N,K,M} \notin  \Lambda^M\}\]
(the explosion resulting from $X^{i,N} \notin B(0,K)$ for any index $i$ is also described by the explosion of $m_\alpha(\pi^{N})$, because if the norm of a particle explodes the moment of the empirical measure will also explode).

\medskip{}

We can now prove the conditional propagation of chaos.

 \begin{proof}[Proof of Proposition~\ref{prop:PoC}]
 
Fix $\varepsilon>0$. Take $K,M$ large enough so that $\mathbb P (\bar{\tau}^{K,M} \leqslant t < \bar{\tau}) \leqslant \varepsilon$, which is possible since $\bar{\tau}^{M}\rightarrow \bar{\tau}$ a.s. as $M\rightarrow \infty$ and $\bar{\tau}^{K,M}\rightarrow \bar{\tau}^M$ a.s as $K\rightarrow \infty$ . Since, besides, $\tau^{(N),K,M} \leq \tau^{(N)}$ for any $K,M>0$, we bound
\begin{eqnarray*}
\mathbb P \po \tau^{(N)} \leqslant  t <\bar{\tau} \pf & \leqslant & \mathbb P \po \tau^{(N)} \leqslant  t < \bar{\tau}^{K,M} \pf + \mathbb P \po \bar{\tau}^{K,M}\leqslant  t < \tau \pf   \\
& \leqslant & \mathbb P \po \tau^{(N),K,2M} \leqslant  t <\bar{\tau}^{K,M} \pf + \varepsilon\,.
\end{eqnarray*}
Let $r>0$ be the $\mathcal W_\alpha$ distance between $\mathcal D_M$ and $ (\mathcal{D}_{2M})^c$. The event $\{\tau^{(N),K,2M} \leqslant  t <\bar{\tau}^{K,M} \}$ implies that, at time $s=\tau^{(N),K,2M}$, $\bar{\mu}^{K,2M}_s$ and $\pi^{N,K,2M}_s$ are at a distance at least $r$.

Hence
\begin{equation}
\label{eq:demoPoc}
 \mathbb P \po \tau^{(N),K,2M} \leqslant  t <\bar{\tau}^{K,M} \pf  \leqslant \mathbb P \po \sup_{s\in[0,t]} \mathcal W_\alpha \po \bar{\mu}_t^{K,2M} ,\pi_s^{N,K,2M}\pf  \geqslant r \pf  \underset{N\rightarrow \infty}\longrightarrow 0\,,
\end{equation}
by standard propagation of chaos with the truncated coefficients with parameter $K,2M$ (Proposition~\ref{prop:propchaos}). As a conclusion, for any arbitrary $\varepsilon>0$, $\limsup_N\mathbb P \po \tau^N \leqslant  t < \bar{\tau} \pf \leqslant \varepsilon$, which concludes the proof of \eqref{eq:tauN->tau}.

Keeping the same $K$ and $M$ and using that $\mathcal W_\alpha(\pi^N_s,\bar{\mu}_s)=\mathcal W_\alpha(\pi_s^{N,K,2M},\bar{\mu}_s^{K,2M})$ for all $s\leqslant \tau^{(N), K,2M}\wedge \bar{\tau}^{K,M}$,
\begin{eqnarray*}
\lefteqn{\mathbb P \po \tau^{(N)} > t,\ \bar{\tau}> t \text{ and } \sup_{s\in[0,t]} \mathcal W_\alpha(\pi^N_s,\bar{\mu}_s) \geqslant \delta  \pf }\\
& \leqslant & \mathbb P \po \bar{\tau}^{K,M} \leqslant  t < \tau \pf + \mathbb P \po \tau^{(N),K,2M} \leqslant  t < \bar{\tau}^{K,M} \pf  + \mathbb P \po  \sup_{s\in[0,t]} \mathcal W_2(\pi_s^{N,K,2M},\bar{\mu}_s^{K,2M}) \geqslant \delta  \pf \\
& \leqslant &   \varepsilon + \underset{N\rightarrow\infty}o(1)\,,
\end{eqnarray*}
where we treated the last term again by classical propagation of chaos and the second term by \eqref{eq:demoPoc}. Again, this concludes.
 \end{proof}
 
\section{Controlling Explosion: Lyapunov condition and Regularized Dynamics}\label{sec:explosion} 

The potentially explosive behavior of the stochastic moment dynamics is theoretically intriguing but undesirable from a practical standpoint. The SMD is intended as a meaningful noisy counterpart to the Wasserstein gradient flow \eqref{eq:granularmedia}, where the added stochasticity should enhance controllability—allowing the process to escape the attraction basins of the underlying deterministic dynamics. However, as in the finite-dimensional setting, the resulting process remains metastable, and transitions between basins corresponding to different minimizers occur rarely. From an optimization perspective, this is problematic: identifying global minimizers of the energy \eqref{eq:energy} requires the process to evolve over a sufficiently long time to allow such transitions. Yet, if the solution ceases to exist due to explosion at time $\tau$, the process fails to explore the landscape fully—thus undermining one of the key advantages of SMD.

\medskip{}

This section addresses this issue using two complementary strategies. First, since explosion occurs when the measure $\mu_t$ enters a region where the coefficients $b$ and $\sigma$ are undefined, we propose a criterion to prevent such behavior. This criterion relies on the existence of a Lyapunov function tailored to the problem, ensuring that the process remains within a well-behaved region of the space. Second, since explosion results from the singular nature of $b$ and $\sigma$, we consider an alternative formulation in which these coefficients are regularized, leading to a globally defined and non-explosive evolution. Depending on the observable $f$ and on the coefficients $a,s$, one of the two strategies may be more adapted to the problem, as this will be illustrated in Section \ref{sec:exemples}.

\subsection{Non-explosion Criterion} \label{subsec:Lyap}

As explained in the introduction, the explosive behavior of the SMD is a consequence of the fact that the process $\mu_t(f)$ cannot take any possible value and must respect some constraints depending on the observable $f$ (like positivity for $f(x)=x^2$). Then, when the coefficients $a$ and $s$ are designed so that $\mu_t(f)$ always live in region of the space with no singularities, it is natural to expect the process to be non-explosive. In order to formulate this idea, we provide a criterion ensuring that the blowing time $\tau$ is almost surely infinite. It corresponds to the existence of a Lyapunov function for $(\mu_t)$ which implies that the process never reaches the region of the space where the coefficients are singular. Before presenting this non-explosion criteria, we first present a lemma that can be seen as an extension of \ref{cor:edps} for observables non necessarily bounded. 

\begin{lem} \label{lem:evolutionmoment}
    Under the settings of Proposition \ref{prop:stochasticmomentswelldefined}, let $M>0$ and $X_{\cdot \wedge \tau^M},\mu_{\cdot \wedge \tau^M}$ be the solution of equation \eqref{eq:stomomentdyn} up to $\tau^M$. Suppose that $g \in \mathcal{C}^2(\R^d,\R^n)$ with $n\in \N^*$, with at most a polynomial growth of order $r$, with $r>0$ such that:
    \begin{itemize}
    \item under Assumption~(B.2), \(r\) must satisfy \(0 < r \leq 2\);
    \item under Assumption~(B.2'),  \(r > 0\) can be taken arbitrarily.
\end{itemize}
    Then, 
    \begin{equation}
        \mu_{t\wedge \tau^M}(g)= \mu_0(g)+\int_0^{t\wedge \tau ^M}\mu_sL_{\mu_s}g(X_s)ds+\int_0^{t\wedge \tau^M} \mu_s(\nabla g ^T\sigma_{\mu_s})dW_s^0
    \end{equation}
\end{lem}

The proof of this lemma is based on the estimate given in Lemma \ref{lem:boundmoments} and on  \cite[Lemma~(A.5)]{hammersley2021weak}.

\begin{proof}
    We use the same notations than the one defined in the proof of Proposition \ref{prop:stochasticmomentswelldefined}. Let $K,M>0$. First, the application of Fatou's Lemma in the estimates given by Lemma \ref{lem:boundmoments} justifies that for all $M>0$:
    
\begin{enumerate}
    \item If $\tilde{b}$ verifies (B.2'), for all $r \geq0$: $\mathbb{E}(\underset{t \in [0,T]}\sup|X_{t\wedge \tau^{M} } |^r) \leq C$
        \item If $\tilde{b} $ satisfy assumption (B.2): $\mathbb{E}(\underset{t\in [0,T]} \sup|X_{t\wedge \tau^{M} }|^2) \leq C$

\end{enumerate}

    Moreover, the application of Ito's Formula to $g(X_{t\wedge \tau^{K,M} \wedge \gamma^{K,M}})$ gives that:
\begin{multline*}
            g(X_{t\wedge \tau^{K,M} \wedge \gamma^{K,M}})=g(X_0)+ \int_0^{t\wedge \tau^{K,M} \wedge \gamma^{K,M}} L_{\mu_s}g(X_s)ds \\
            +\int_0^{t\wedge \tau^{K,M} \wedge \gamma^{K,M}} \nabla g(X_s)^T\sigma_{\mu_s}(X_s) dW_s^0+\int_0^{t\wedge \tau^{K,M} \wedge \gamma^{K,M}} \nabla g(X_s)^T\tilde{\sigma}_{\mu_s}(X_s) dW_s
        \end{multline*}
    Applying the conditional expectation with respect to $\mathcal{F}_t^0$ on both side of the equation, and because every integrand is bounded before $\tau^{K,M} \wedge \gamma^{K,M}$, we can apply Lemma (A.5) from \cite{hammersley2021weak} to get:
     \begin{equation*}
        \begin{split}
            \mathbb{E} \Big(g(X_{t\wedge \tau^{K,M} \wedge \gamma^{K,M}}) | \mathcal{F}^0_t \Big)=& \mathbb{E} \Big(g(X_0) | \mathcal{F}^0_t \Big)+ \int_0^{t\wedge \tau^{K,M}}\mathbb{E} \Big( \1_{s\leq \gamma^{K,M}} L_{\mu_s}g(X_s) | \mathcal{F}^0_s \Big) ds \\
            &+  \int_0^{t\wedge \tau^{K,M} }\mathbb{E} \Big( \1_{s \leq \gamma_{K,M}} \nabla g(X_s)^T\sigma_{\mu_s}(X_s) | \mathcal{F}^0_s \Big) dW_s^0 +0\\
        \end{split}
    \end{equation*}

    We are now going to let $K\rightarrow \infty$ and justify the a.s convergence of every term of the equation. We will justify the convergences with Lemma \ref{lem:boundmoments}, to prove the domination hypothesis of the following dominated convergence theorems (classical, conditionnal and stochastic).

\medskip{}
    First, since $g(X_{t\wedge \tau^{K,M} \wedge \gamma^{K,M}}) \leq \underset{s\in [0,T]}  \sup C(1+|X_{s\wedge \tau^{M} }|^r) $ and thanks to the previous remark which shows that the right term of the equation is integrable, the conditional dominated convergence theorem ensures that:
    $$\underset{K\rightarrow \infty} \lim \mathbb{E} \Big(g(X_{t\wedge \tau^{K,M} \wedge \gamma^{K,M}}) | \mathcal{F}^0_t \Big)=\mathbb{E} \Big(g(X_{t\wedge \tau^{M} }) | \mathcal{F}^0_t \Big)= \mu_{t\wedge \tau^M}(g)\,,$$
where the last equality follows from conditional Fubini theorem.

Second, if $\tilde{b}$ satisfy assumption (B.2), then we can find a constant C (depending on $M$) such that for all $s$: $|L_{\mu_s}g(X_{\wedge \tau^{K,M} \wedge \gamma^{K,M}}) |\leq \underset{s\in [0,T]}  \sup C(1+|X_{s\wedge \tau^{M} }|^2) $, and again the dominated convergence theorem concludes. If $\tilde{b}$ satisfy assumption (B.2'), we can bound $|L_{\mu_s}g(X_{\wedge \tau^{K,M} \wedge \gamma^{K,M}}) |\leq  C(1+|X_{s\wedge \tau^{M} }|^{r'}) $  with $r'\geq2$ and the first statement of Lemma \ref{lem:boundmoments} justifies its integrability. The dominated convergence theorem to get that $\mathbb{P}$ a.s:

$$\int_0^{t\wedge \tau^{K,M}}\mathbb{E} \Big( \1_{s\leq \gamma^{K,M}} L_{\mu_s}g(X_s) | \mathcal{F}^0_s \Big) ds  \underset{K \rightarrow \infty}{\longrightarrow} \int_0^{t\wedge \tau^{M}}\mu_sL_{\mu_s}g \, ds \,.$$

\medskip{}

Finally, because $$|\mathbb{E} \Big( \1_{s \leq \gamma_{K,M}} \nabla g(X_s)^T\sigma_{\mu_s}(X_s) | \mathcal{F}^0_s \Big)| \leq \mathbb{E}( \underset{t\in [0,s]}   \sup C(1+|X_{t\wedge \tau^{M} }|^r) | \mathcal{F}^0_s)$$ and because the dominated process is locally bounded and progressive (again because of the remark done at the beginning of the proof), we can apply the dominated convergence theorem for stochastic integrals to get that:

$$ \int_0^{t\wedge \tau^{K,M} }\mathbb{E} \Big( \1_{s \leq \gamma_{K,M}} \nabla g(X_s)^T\sigma_{\mu_s}(X_s) | \mathcal{F}^0_s \Big) dW_s^0 \underset{K \rightarrow \infty}{\longrightarrow}  \int_0^{t\wedge \tau^{M} }\mu_s( \nabla g^T\sigma_{\mu_s})  dW_s^0 $$ in probability. Then, we can extract a subsequence that converges almost surely, and taking the limit for that subsequence we can conclude that a.s:

$$\mu_{t\wedge \tau^M}(g)= \mu_0(g)+\int_0^{t\wedge \tau ^M}\mu_sL_{\mu_s}gds+\int_0^{t\wedge \tau^M} \mu_s(\nabla g ^T\sigma_{\mu_s})dW_s^0\,,$$
which proves the result.
\end{proof}

Under the settings of Proposition \ref{prop:stochasticmomentswelldefined}, we can define \[\Lambda:= \bigcup_{M \in \N^*} \Lambda^M=\{ \mu \in \mathcal{P}_\alpha(\R^d), \  d(\mu(f), \mathcal{S}) >0  \}\]
and $\mathcal{O}:=\text{Ev}_f(\mathcal{P}_2(\R^d)) \subset \R^p$ with $\text{Ev}_f: \mu \mapsto \mu(f)$. The previous lemma will allow us to study the evolution along the random flow $\mu_{\cdot\wedge \tau}$ of functions $\mathbb{V}: \mu \in  \Lambda\rightarrow \R^+$ depending on $\mu$ only through its moments $\mu(g)$, with $g$ verifying the assumptions of the previous lemma. Then, if we can find an appropriate Lyapunov function $\mathbb{V}$ for the evolution, we will able to justify that the blowing time $\tau=\infty$ almost surely. In order to state our non-explosion criteria, we need to make the following assumption on the observable $f$ and on its relation with the coefficients $a$ and $s$.

\begin{assu} \label{hyp:nonexplosion}
    Suppose that $f\in \mathcal{C}^2(\R^d,\R^p)$ verify either Assumption (A.1) or (A.2) (in both cases $\alpha=2$). Suppose that the coefficients $a$ and $s$ are such  that the following holds:
    $$ \text{For all } \mu \in \mathcal{P}_2(\R^d), \quad  d(\mu(f), \mathcal{S}) >0 \Longrightarrow \det( \mu(\nabla f^T \nabla f))>0$$ 
\end{assu}

\begin{prop} \label{prop:nonexplosioncriteria}
    
    Under the setting of Proposition \ref{prop:stochasticmomentswelldefined}, suppose that $f$ verify Assumption \ref{hyp:nonexplosion}. Suppose that it exist $n\in \N^*$ and $g\in \mathcal{C}^2(\R^d,\R^n)$ verifying the conditions of Lemma \ref{lem:evolutionmoment} and $V\in \mathcal{C}^2(\R^n, \R^+)$ such that $\mathbb{V}:\mu \in \Lambda \mapsto V(\mu(g))$ verify the following conditions:

    \begin{enumerate}
        \item $\mathbb{V}(\mu)\rightarrow \infty$ when $d(\mu(f),\mathcal{S})\rightarrow0$ and when $m_2(\mu)\rightarrow \infty$
        \item For all $M>0$, $d\mathbb{V}(\mu_{t \wedge \tau^M})=G(\mu_{t \wedge \tau^M})dt +dA_{t\wedge \tau^M}$ with $A_{\cdot \wedge \tau^M}$ a local martingale, and it exists a positive constant $C\geq 0$ such that for $\mu \in \Lambda$:
        \begin{equation} \label{eq:lyapG}
            G(\mu)\leq C \mathbb{V}(\mu)\,.
        \end{equation}
    \end{enumerate}
    
    Then $\tau=\infty$ $\mathbb{P}$ a.s.
\end{prop}

The proof of this result is a direct consequence of the previous lemma and of a classical Lyapunov argument for semi-martingales. The choice of the observable $g$ and of the function $V$ is voluntarily left general so that this criteria can be applied for various stochastic moment dynamics processes.

\begin{proof}

The proof of this proposition is based on the previous lemma. Indeed, it justifies that we can apply Ito's formula to $\mathbb{V}(\mu_{t\wedge \tau^M})=V(\mu_{t\wedge\tau^M}(g))$, and we can apply a standard Lyapunov argument to justify that it cannot explode in finite time. Indeed, because of the previous Lemma, if we take $C$ such that $G(\mu)\leq C \mathbb{V}(\mu)$ for $\mu \in \Lambda$, the application of Ito's formula to $e^{-C(t\wedge \tau^M)}\mathbb{V}(\mu_{t\wedge \tau^M})$ gives that:

\begin{equation*}
    e^{-C(t\wedge \tau^M)}\mathbb{V}(\mu_{t\wedge \tau^M}) = \mathbb{V}(\mu_0)+\int_0^te^{-C(t\wedge \tau^M)} \big[ G(\mu_{s\wedge \tau^M})-C\mathbb{V}(\mu_s)  \big]ds +A_t
\end{equation*}
where $(A_t)$ is a local martingale. Thanks to~\eqref{eq:lyapG}, we get that $e^{-C(s\wedge \tau^M)}\mathbb{V}(\mu_{t\wedge \tau^M})$ is a (local) supermartingale. Then, writing $\tau_r^\mathbb{V}=\inf \{t \geq 0, \quad \mathbb{V}(\mu_t) \geq r \}$, because of the first assumption made on $V$, we know that for all $r>0$, $\tau_r^\mathbb{V} \leq \tau$. Then, if $(T^N)_{N \in \N^*}$ is a sequence of stoping times reducing $A_t$ we get that:

\begin{equation*}
    \begin{split}
        \mathbb{V}(\mu_0) \geq& \mathbb{E}(e^{-C(T\wedge \tau^M \wedge \tau^\mathbb{V}_r \wedge T^N)}\mathbb{V}(\mu_{T\wedge \tau^M \wedge \tau^\mathbb{V}_r \wedge T^N})) \\
        \geq & r e^{-CT} \mathbb{P}(\tau^{\mathbb{V}}_r \leq T\wedge \tau^M \wedge T^N) \\
        \geq & r e^{-CT}\mathbb{P}(\tau^{\mathbb{V}}_r \leq T)\,,
    \end{split}
\end{equation*}
where the last inequality is obtained by monotone convergence. Taking the limit for $r\rightarrow \infty$, we get that $\tau^\mathbb{V}_r=\infty $ a.s. Then, because $\tau \geq \tau^\mathbb{V}_r$, we get that a.s $\tau=\infty$, which proves the result.

\end{proof}

Even for simple (but not trivial) examples, finding an appropriate function $V$ that satisfies the previous proposition is very difficult for dynamics combining both a Wasserstein Gradient descent ($\tilde{b},\tilde{\sigma}\ne 0$) and stochastic moment part ($b,\sigma \ne 0$). However, considering a dynamic without additional gradient descent part ($\tilde{b},\tilde{\sigma}=0$, it is possible to deduce from the previous criteria that the evolution is non explosive if the control on the moment is well designed.

\begin{cor} \label{cor:nonexplosion}
    Under the settings of proposition \ref{prop:stochasticmomentswelldefined} with $\tilde{b},\tilde{\sigma}=0$, suppose that $f$ and the coefficients $a,s$ satisfy Assumption \ref{hyp:nonexplosion}, and suppose that it exists $V \in \mathcal{C}^2(\mathcal{O} \setminus \mathcal{S}, \R^+)$ and a positive constant $C\geq 0$ such that for all $z\in \R^p \setminus \mathcal{S}$:
    \[ a(z)\cdot \nabla V(z)+\frac{1}{2} s(z)s(z)^T : \nabla ^2V(z) \leq C \,V(z) \] \label{eq:lyapV}
    Then $\tau=\infty$ a.s.
\end{cor}

The corollary is a direct application of Proposition \ref{prop:nonexplosioncriteria}, because under the simplifying assumption that $\tilde{b},\tilde{\sigma}=0$, the condition \eqref{eq:lyapG} is rewritten as equation \eqref{eq:lyapV}. A consequence of this result is that knowing the constraints on the evolution for an observable $f$ (i.e. understanding when $\mu(\nabla f^T\nabla f)$ is singular) is enough to define a non-explosive process with some controllability on this specific moment, because it is possible to choose an appropriate control that verifies the conditions of the previous corollary. 

\medskip{}

The non-explosion criteria \ref{prop:nonexplosioncriteria} is well suited when the constraints required by the moments are easy to describe (for instance positivity), and when there is no additional Wasserstein gradient descent terms. However, when a large number of moments is considered, or when considering more complex observables than polynomials, having a clear description of these constraints may be difficult, so designing a suitable control on the moments may not be possible. That is why in the following subsection we introduce a regularized version of the dynamics, designed not to explodes in finite time.

 \subsection{Regularized process}\label{ref:regularise}
 
 In this section we modify the dynamics initially derived in Section~\ref{sec:derivation} to remove the singularities. Obviously, this is not possible while ensuring that the moments $\mu_t(f)$ satisfy exactly a prescribed autonomous SDE, in general. But the fact that the moments $\mu_t(f)$ satisfy a closed equation is actually not a critical point in our situation in view of our initial motivation. What is useful is to have some elliptic noise on these moments, to ensure some controllability. By taking the stochastic moments dynamics of Section~\ref{sec:derivation} as our starting point and then modifying it to suppress the explosive behavior, basically, we ensure the presence of an elliptic noise on $\mu_t(f)$ \emph{as long as} the process stays within the domain defined by the constraints intrinsically  associated to the observables $f$. This is exactly what we need in practice.
 
 There would be several ways to modify the dynamics. The modification could be confined  close to the boundary of the admissible domain, using reflection, a repulsive force, or some truncation of the singular coefficients, so that $\mu_t(f)$ exactly solves the prescribed SDE when $\mu_t$ is far away from this boundary. However, in the following we consider a simpler regularization, which consist in replacing the semi-definite matrix $\mu \co (\na f)^T \na f\cf$ by the non-singular one $\eta I_p + \mu \co (\na f)^T \na f\cf$ for some $\eta>0$.
 
 In other words, for $\eta>0$, the regularized stochastic moment coefficients are defined as 
 \begin{equation} \label{eq:coeffsp-regularized}
     \left\{
        \begin{array}{lll}
        \sigma_{\mu}^\eta(x) &=& \nabla f (x)\po \eta I_p + \mu[\nabla f^T \nabla f]\pf ^{-1} s(\mu(f))\\
    
        b_{\mu}^\eta(x) &=& \nabla f (x) \po \eta I_p + \mu[\nabla f^T \nabla f]\pf ^{-1} [a(\mu(f))-\frac{1}{2}\mu(\sigma_{\mu}^\eta(\sigma_{\mu}^\eta)^T :\nabla^2 f)] \,.
        \end{array}\right. 
\end{equation}

The following proposition states that under some suitable assumptions the regularized stochastic moment dynamic does not explode.

\begin{prop}
    Suppose that $f$ verify assumptions (A.1) or (A.2) ($\alpha=2$). Suppose that the coefficients $a$ and $s$ are globbaly Lipschitz and bounded ($\mathcal{S}=\emptyset$), and that $\tilde{b}$ and $\tilde{\sigma}$ verify assumption \ref{hyp:driftetdiffusionenplus}. Then, for $\eta>0$, the conditional McKean Vlasov equation with coefficients $b^\eta, \sigma^\eta, \tilde{b}, \tilde{\sigma}$ is strongly uniquely solvable in $[0,T]$ (the solution is not explosive).
\end{prop}

The proof of this proposition is based on the observation that under the assumption made on $a$ and $s$, the functions 
\begin{align*}
   \mu & \mapsto (\eta I_p + \mu[\nabla f^T \nabla f] )^{-1} s(\mu(f)) \\
   \mu & \mapsto  \po \eta I_p + \mu[\nabla f^T \nabla f]\pf ^{-1} [a(\mu(f))-\frac{1}{2}\mu(\sigma_{\mu}^\eta(\sigma_{\mu}^\eta)^T :\nabla^2 f)]
\end{align*}
are globally Lipchitz and bounded on $\mathcal{P}_2(\R^d)$. Then, because there is no singularity associated to either $\mu(\nabla f^T\nabla f)^{-1}$ or to $a$ and $s$, the localization procedure justifies that this evolution is well defined until $m_2(\mu)=\infty$. Then, one can apply a similar Lyapunov argument to the one presented in the previous section with $\mathbb{V}(\mu)=1+m_2(\mu)$ to justify that $m_2(\mu_t)< \infty$ for all $0\leq t \leq T$, and so that the solution is non-explosive.

\begin{proof}
    First, one can notice that the map $\psi:A \in \mathcal{S}_p^+ \mapsto (\eta I +A)^{-1}$ is globally Lipschitz and bounded. This follows from the fact that the spectral norm of $\psi(A)$ is bounded by $\frac{1}{\eta}$ and that its differential is equal to $$D\psi(A)(H)=-\psi(A)H \psi(A)$$
    Then, because of the assumptions made on $f$, $a$ and $s$ and because a product of globally Lipschitz and bounded function is also Lipschitz and bounded, the coefficients $b$ and $\sigma$ are globally Lipschitz and bounded in $\R^d \times \mathcal{P}_2(\R^d)$. Then, because $\tilde{b}$ and $\tilde{\sigma}$ are supposed to be locally Lipschitz doing a similar localization procedure than the one done for Proposition \ref{prop:stochasticmomentswelldefined}, we can define a solution to the associated conditional McKean Vlasov equation $X^\eta, \mu^{\eta}$ up to time $\tau^\eta= \inf \{t \in [0,T], \, m_2(\mu^\eta_t)=\infty \}$. Moreover, doing the exact same proof than for Lemma \ref{lem:evolutionmoment}, because of the at most linear growth of $\nabla f$ and because of the assumption made on $\tilde{b}$ and $\tilde{\sigma}$, we can show that for $\phi(x)=1+|x|^2$:
    $$ L_\mu \phi = (\na \phi)^T (b_\mu + \tilde b_\mu) +\frac12 \po \sigma_{\mu}\sigma_{\mu}^T + \tilde \sigma_{\mu}\tilde \sigma_{\mu}^T\pf : \na^2 \phi$$
    verify $L_\mu \phi \leq C \phi$, which imply that:
    $$G^\eta(\mu)=\mu L_\mu^\eta \phi \leq C\mathbb{V}(\mu)= C(1+m_2(\mu))\,,$$
with $C$ a constant independent of the localization procedure. Then, using the same Lyapunov argument than the one presented in the previous subsection, one can deduce that for all $t \in [0,T]$, $\mathbb{V}(\mu_t^\eta)<\infty$, and then that $\tau^\eta=\infty$ a.s, which imply that the regularized process is non explosive.
\end{proof}

In the following, we present an application of the non-explosion criteria and of the regularized dynamic in some simple examples.
 
 \section{Some simple explicit examples}\label{sec:exemples}

 The goal of this section is to present simple examples of the stochastic moment dynamic, and to show that in these specific cases the processes are non explosive. Note that the process introduced in \cite{delarue2024ergodicity} where the expectation evolves like a Brownian motion is also an example of a non explosive stochastic moment dynamic. The following examples are defined in dimension $d=1$. For these examples, we will not consider an additional Wasserstein gradient descent term $\tilde{b}=\tilde{\sigma}=0$, corresponding to the study of the SMD alone. The first example that we will present is designed to ensure a control on the second order moment, where the second one is designed to ensure a control on both the mean and the variance of the process.

 \subsection{A control on the second order moment} \label{sec:premierexemple}

 In this first example we want to control the second order moment of the process, so we define the observable $f:x \mapsto x^2$ in dimension $d=1$ (and $p=1$ because we control only one observable). As explained previously, because the second order moment of a probability distribution is necessary positive ($\mathcal{O}=\R^+$), it is natural to force it to be an $\R^*_+$ valued process to avoid a possible explosion of the system (happening when $\mu_t(f) \in \partial\mathcal{O}$. To do so, we have to choose the coefficient $a$ and $s$ that will control the evolution of $\mu_t(f)$. The simplest example we can think of is based on Bessel processes, which are positive stochastic processes that are known to never touch $0$ when their dimension parameter $\delta$ is greater than $2$ (when $\delta$ is an integer they correspond to the norm of a $\delta$-dimensional Brownian motion). We refer to \cite{lawler2018notes} for more details on these processes. For $\delta \geq 2$, Bessel processes are defined as the unique strong solution of the following SDE:

 \begin{equation} \label{eq:bessel}
     dZ_t=\frac{\delta-1}{2 Z_t}dt+dW_t^0\quad Z_0>0
 \end{equation}

 Then, we have to define the coefficients $a$ and $s$ controlling the evolution of $\mu_t(f)$ by:

 $$ \text{For \hspace{4pt}}  z \in \R_+^{*}, \quad a(z)=\frac{\delta-1}{2z}, \quad s(z)=1\,,$$
 which verify assumption \ref{hyp:coeffaets} with $\mathcal{S}=\{0\}$. The coefficients~\eqref{eq:stomomentdyn}  associated to the evolution of the stochastic moment dynamic 
then read
\begin{equation}
    b_{\mu}(x)=\frac{(\delta-\frac{3}{2})}{4m_2(\mu)^2}x \,,\quad \quad \sigma_{\mu}(x)=\frac{x}{2m_2(\mu)}\,.
\end{equation}

Suppose that $\delta>2$, and set $q>0$ such that $2+q<\delta$.
We can apply Corollary \ref{cor:nonexplosion} with observable $g(x)=f(x)=|x|^2$ and function $V(z)=z+ \frac{1}{z^q}$ because for all $z>0$:

\[a(z) V'(z)+\frac{1}{2} s(z)^2 V''(z)=\frac{\delta-1}{2m_2(\mu)}+\frac{q(q+2-\delta)}{2m_2(\mu)^{2+q}}\,.\]
Then, because $ 2+q < \delta$, we can find a constant $C \geq0 $ such that for all $z>0$:
\[ \frac{\delta-1}{2z}+\frac{q(q+2-\delta)}{2z^{2+q}} \leq C\,.\]
Hence, since $V(z)\geq1$, Corollary \ref{cor:nonexplosion} ensures that the resulting process is non explosive.

\subsection{A control on the mean and the variance} \label{sec:secondexemple}

For now, we saw that it was possible to build two simple processes to control either the mean or the second order moment of the process. It is natural to try to combine these two in order to have a control on both the mean and the variance of the process (and so $p=2$). As explained in the introduction, controlling the mean as a Brownian motion does not induce a singularity in the evolution, while the non-negativity of the variance imposes the use of a non negative process to avoid an explosion in zero. Thus, it seems reasonable to control the mean as a Brownian motion and the variance as a Bessel process. Because there is no constraint relating the variance and the mean of a probability distribution, it is natural to require both observables to evolve independently. Then, given a $2$-dimensional Brownian motion $W^0=(W^{0,1},W^{0,2})$, if we define $f_1:x \mapsto x$,  $f_2:x \mapsto x^2$, because $\text{Var}(\mu)=\mu(f_2)-\mu(f_1)^2$ we require that:

 \begin{equation*} 
     \left\{
        \begin{array}{rcl}
        \dd\mu_t(f_1) &=& \dd W_t^{0,1}\\
       \displaystyle \dd(\mu_t(f_2)-\mu_t(f_1)^2 ) &=& \displaystyle \frac{\delta-1}{2(\mu_t(f_2)-\mu_t(f_1)^2 )} \dd t+ \dd W_t^{0,2}
        \end{array}\right.
\end{equation*}
which can equivalently be written as
 \begin{equation*} 
     \left\{
        \begin{array}{rcl}
        \dd\mu_t(f_1) &=& \dd W_t^{0,1}\\
       \displaystyle \dd \mu_t(f_2) &=&\displaystyle  2 \mu(f_1) \dd W_t^{0,1} +dt+\frac{\delta-1}{2(\mu_t(f_2)-\mu_t(f_1)^2 )} \dd t+ \dd W_t^{0,2}
        \end{array}\right.
\end{equation*}

Then, if we define the observable $f=(f_1,f_2)$, the coefficients $a$ and $s$ directing the evolution of $\mu(f)$ must be:
 \begin{equation}
     a(z_1,z_2)= \begin{pmatrix} 
0 \\ \displaystyle
1+\frac{\delta-1}{2(z_2-z_1^2)} \\
\end{pmatrix}, \quad s(z_1,z_2)= \begin{pmatrix}
1 & 0 \\
2z_1 & 1 
\end{pmatrix}\,.
 \end{equation}
These coefficients verify assumption \ref{hyp:coeffaets}, with $\mathcal{S}=\{(z_1,z_2) \in \R^2, \ z_2-z_1^2=0\}$, and $\mathcal{O}=\{(z_1,z_2) \in \R^2, \ z_2-z_1^2\geq0\}$. Using the expression~\eqref{eq:stomomentdyn} of the coefficients $b$ and $\sigma$, defining $m(\mu)=\mu(f_1)$ the mean of $\mu$, one can check that their expression can be simplified as:

 \begin{equation}
    b_{\mu}(x)=\frac{\delta-\frac{3}{2}}{4\text{Var}(\mu)^2}(x-m(\mu)) \quad \quad \sigma_{\mu}(x)=\Big(1 \hspace{3pt}, \hspace{3pt}\frac{\delta-\frac{3}{2}}{2\text{Var}(\mu)} (x-m(\mu)) \Big)
\end{equation}

Suppose that $\delta>2$, and set $q>0$ such that $2+q<\delta$. We can apply almost the same argument that for the previous example. Let $h(z_1,z_2):=z_2-z_1^2$, and $V(z_1,z_2)=1+ z_2+ \frac{1}{h(z_1,z_2)^q}$. Since
\begin{align*}
\nabla V(z_1, z_2) &= \begin{pmatrix}
2q z_1 h^{-q - 1} \\
1 - q h^{-q - 1}
\end{pmatrix}    \\
\nabla^2 V(z_1, z_2) 
&=
\begin{pmatrix}
2q h^{-q - 1} + 4q(q+1) z_1^2 h^{-q - 2} & -2q(q+1) z_1 h^{-q - 2} \\
-2q(q+1) z_1 h^{-q - 2} & q(q+1) h^{-q - 2}
\end{pmatrix}\,,
\end{align*}
we get
\[
a^T \nabla V + \frac{1}{2} s s^T : \nabla^2 V = 1 + \frac{\delta-1}{2h}+\frac{q}{2} \cdot \frac{2+q-\delta}{h^{q+2}}
\]

Then, again, because $ 2+q < \delta$, we can find a constant $C \geq0 $ such that the right term of the equation is bounded by $C$ for all $(z_1,z_2) \in \mathcal{O}\setminus \mathcal{S}$. Finally, because $V(z_1,z_2)\geq 1$, we have the desired inequality, and thus $\tau=\infty $ a.s.

\medskip{}

In the two previous examples, the main reason why we were able to easily prove that the evolutions were non explosive was that the blowing time was depending on $\mu_t$ only through the controlled observable $\mu_t(f)$. Hence, choosing an appropriate control on this process was enough to ensure stability. However, this is not true on the general case even for simple observables, in which case one should use the regularized version of SMD to avoid the explosion.

 \section{Numerical illustrations}\label{sec:numerique}

 The goal of this section is to illustrate the behavior of the Stochastic Moment Dynamic for several different observables and control on their evolution. As explained in the previous sections, some of them will have an explosive behavior if the constraints verified by $\mu(f)$ are not verified by the diffusion directing its evolution. We will show that numerically the regularized processes are non-explosive and that even if they do not verify the prescribed SDE, they still bring elliptic noise on the moments of the particle system (which was our original motivation). This section is organized as follows. First, we illustrate the possible explosive behavior of the SMD with the example introduced in Section \ref{sec:secondexemple}: depending on the value of the parameter of the associated Bessel process, the singularity may (or not) be reached, so the process may (or not) explode. Second, we show for two other observables and control that numerically, the regularized processes introduced in Section~\ref{ref:regularise} are stable and still add elliptic noise on the moment's evolution. Finally, we show that when suitable observables are controlled, the addition of a SMD term in a Wasserstein gradient descent allows for transitions between the bassins of attraction of the different minimizers. This will be illustrated for the granular media equation~\eqref{eq:granularmedia} with the double potential and the attractive quadratic interaction kernel.

 All simulations are based on an Euler-Maruyama scheme applied to the particle system~\eqref{eq:commonnoisePS}.

 \subsection{Explosive behavior of the SMD}

In this section, we illustrate through the example of Section \ref{sec:secondexemple} that depending on the observable and the control chosen, the stochastic moment dynamic process may (or not) be explosive.

We simulate the system of $N=1000$ particles, initialized as independent variables with Gaussian distribution $\mathcal{N}(0,1)$, with coefficients:
 \begin{equation}
    b_{\mu}(x)=\frac{\delta-\frac{3}{2}}{4\text{Var}(\mu)^2}(x-m(\mu)) \quad \quad \sigma_{\mu}(x)=\Big(1 \hspace{3pt}, \hspace{3pt}\frac{\delta-\frac{3}{2}}{2\text{Var}(\mu)} (x-m(\mu)) \Big) \quad  \quad \tilde{b}=\tilde{\sigma}=0
\end{equation}
corresponding to the SMD process in dimension 1 with observable $f(x)=(x,x^2)$, where the mean evolves like a Brownian motion and the variance as a $\delta$ Bessel process (independent). The result, displayed in Figure~\ref{fig:1}, show a very different behavior of the process depending on the value of $\delta$.

\begin{figure}[!ht] \centering
  \includegraphics[width=180mm, scale=1]{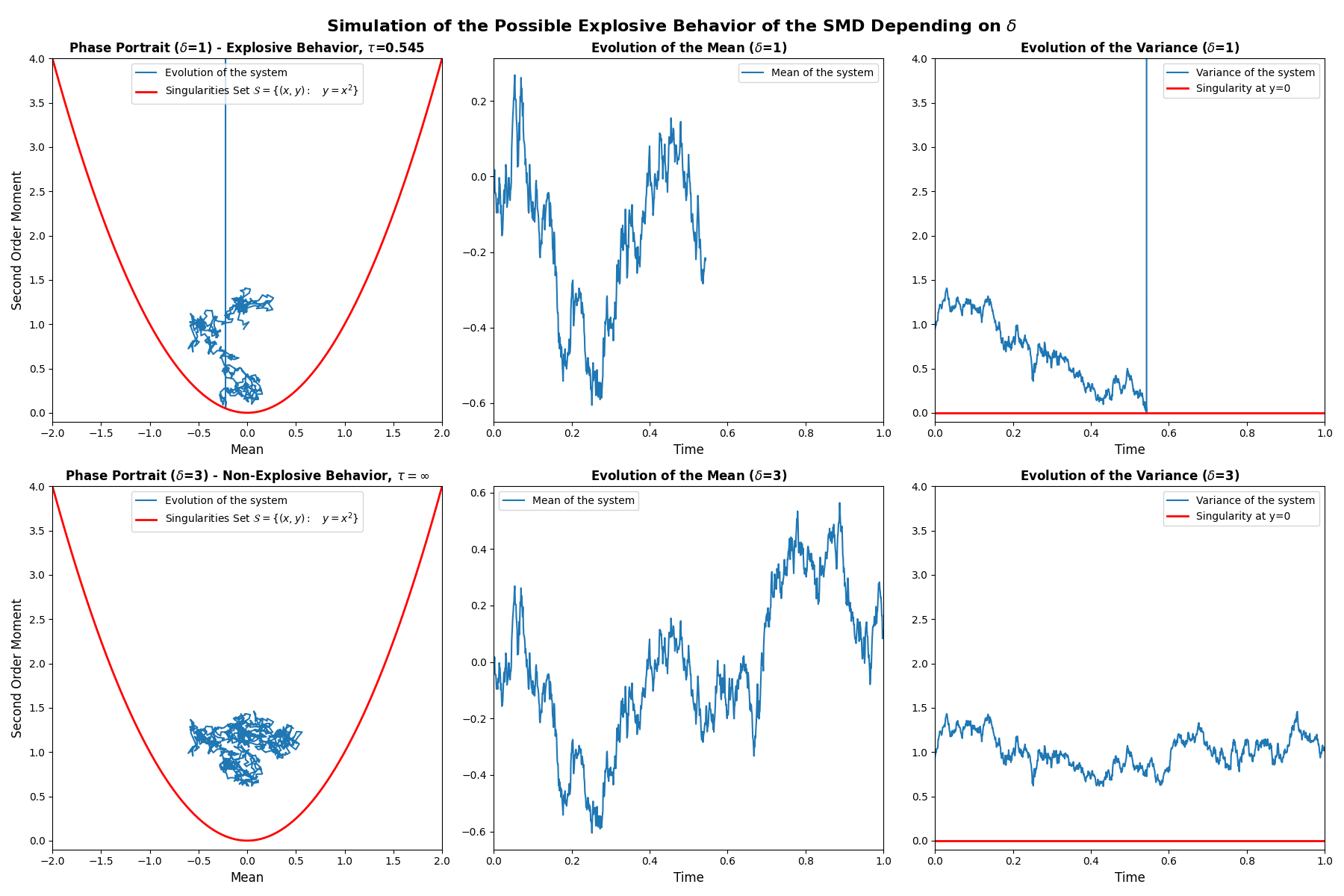}
  \caption{Explosive/Non explosive behavior of the SMD depending on the parameter $\delta$ (top: $\delta=1$; bottom: $\delta=3$).}\label{fig:1}
\end{figure}

The process $\mu(f)$ is constrained (because of Jensen's inequality) to live in $\mathcal{O}=\{(x,y)\in \R^2, \, y\geq x^2 \}$ (in red in the left pictures of Figure~\ref{fig:1}), and the evolution is no longer defined when it reaches the boundary of the domain (the process then explodes). By construction, the evolution is designed so that the explosion can (resp. cannot) happen if $\delta$ is smaller (resp. greater) than $2$. The simulation illustrates this phenomenon: two SMD processes initialized with the same initial conditions and driven by the same Brownian motion can either explode or not depending on the value of $\delta$.



\subsection{Regularized version of the SMD}

In this section, we present a numerical illustration of the non-explosive behavior of the regularized SMD from Section~\ref{ref:regularise}, even if the control on the moment is not well designed (it may reach the singularity of the non-regularized SMD). We will present two different cases of study, each corresponding to a different observable $f$ in dimension $1$, and is both cases directed as a Brownian motion ($a(z)=0$ and $s(z)=1$). In both case, the particles will be initialized as independent variables with Gaussian distribution $\mathcal{N}(0,1)$, and we will not consider an additional McKean Vlasov term, so that $\tilde{b}=\tilde{\sigma}=0$. 
The following simulations will be associated to the two different cases $f(x)=x^2$ (in Figure~\ref{fig:2}) and $f(x)=\tanh(x)$ (in Figure~\ref{fig:3}), implying that $\mu_t(f)$ is constrained to live respectively on $\R^+$ and $(-1,1)$. Defining $\eta\geq 0$ the regularization parameter, for $f(x)=x^2$, the coefficients $b^{\eta}$ and $\sigma^{\eta}$ are respectively defined by:
\begin{equation}
    b_\mu^\eta(x)=-\frac{8x m_{2}(\mu)}{(\eta+4m_2(\mu))^3} \quad \sigma_\mu^{\eta}=\frac{2x}{\eta+4m_2(\mu)}\,.
\end{equation}

\begin{figure}[!ht] \centering
  \includegraphics[width=180mm, scale=1]{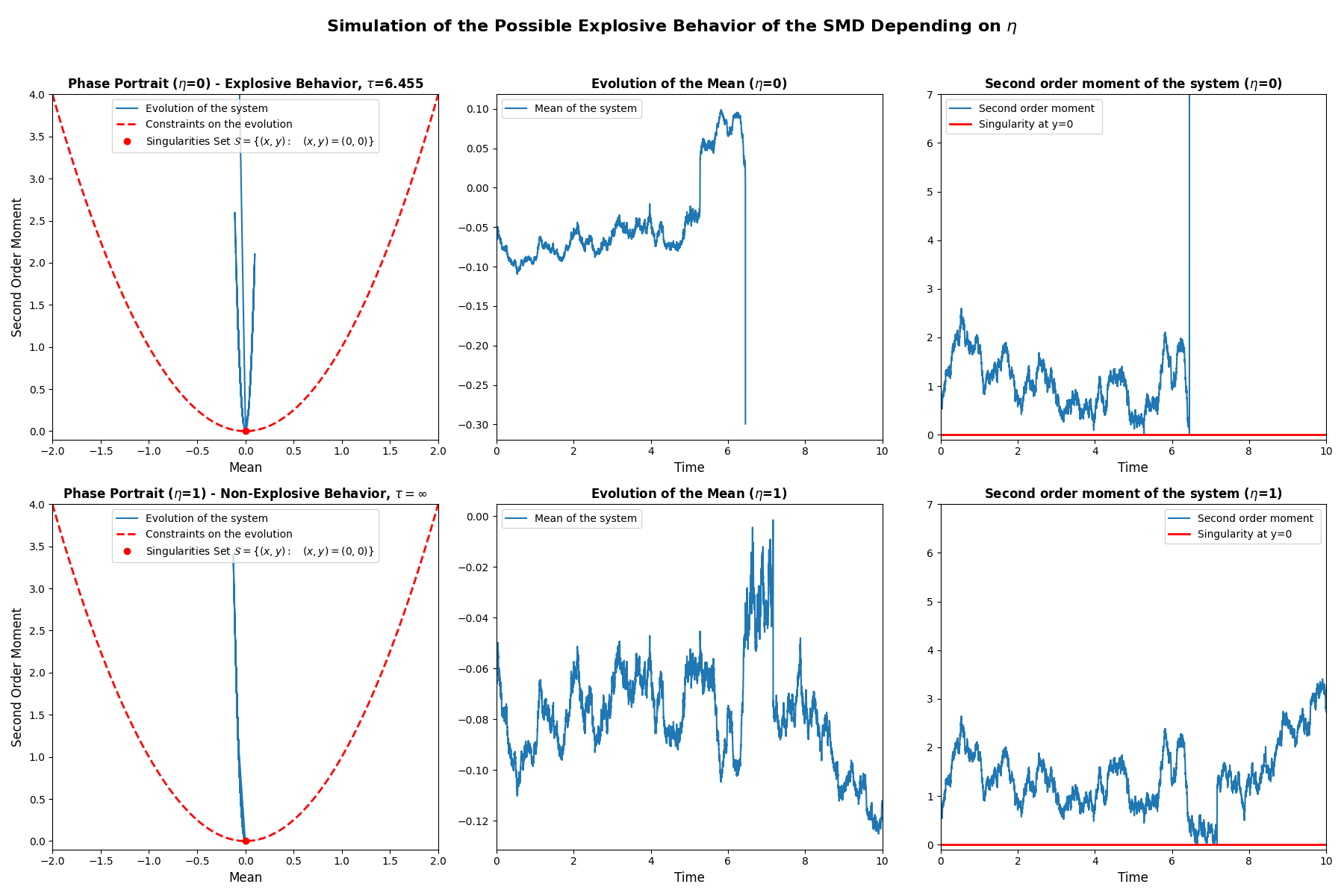}
  \caption{Explosive/Non explosive behavior of the regularized SMD with $f(x)=x^2$ depending on the parameter $\eta$ (top: $\eta=0$; bottom: $\eta=1$)}\label{fig:2}
\end{figure}
For $f(x)=\tanh(x)$, the coefficients $b^{\eta}$ and $\sigma^{\eta}$ are defined by:
\begin{equation}
 b^\eta_\mu(x)=\frac{1}{\cosh(x)^2} \cdot \frac{ \mu(\frac{\tanh}{\cosh^6} )}{(\eta+\mu(\frac{1}{\cosh^4}))^3} \quad \sigma^\eta_\mu(x)=\frac{1}{\cosh(x)^2} \cdot \frac{ 1}{\eta+\mu(\frac{1}{\cosh^4})}\,.
\end{equation}

\begin{figure}[!ht] \centering
  \includegraphics[width=180mm, scale=1]{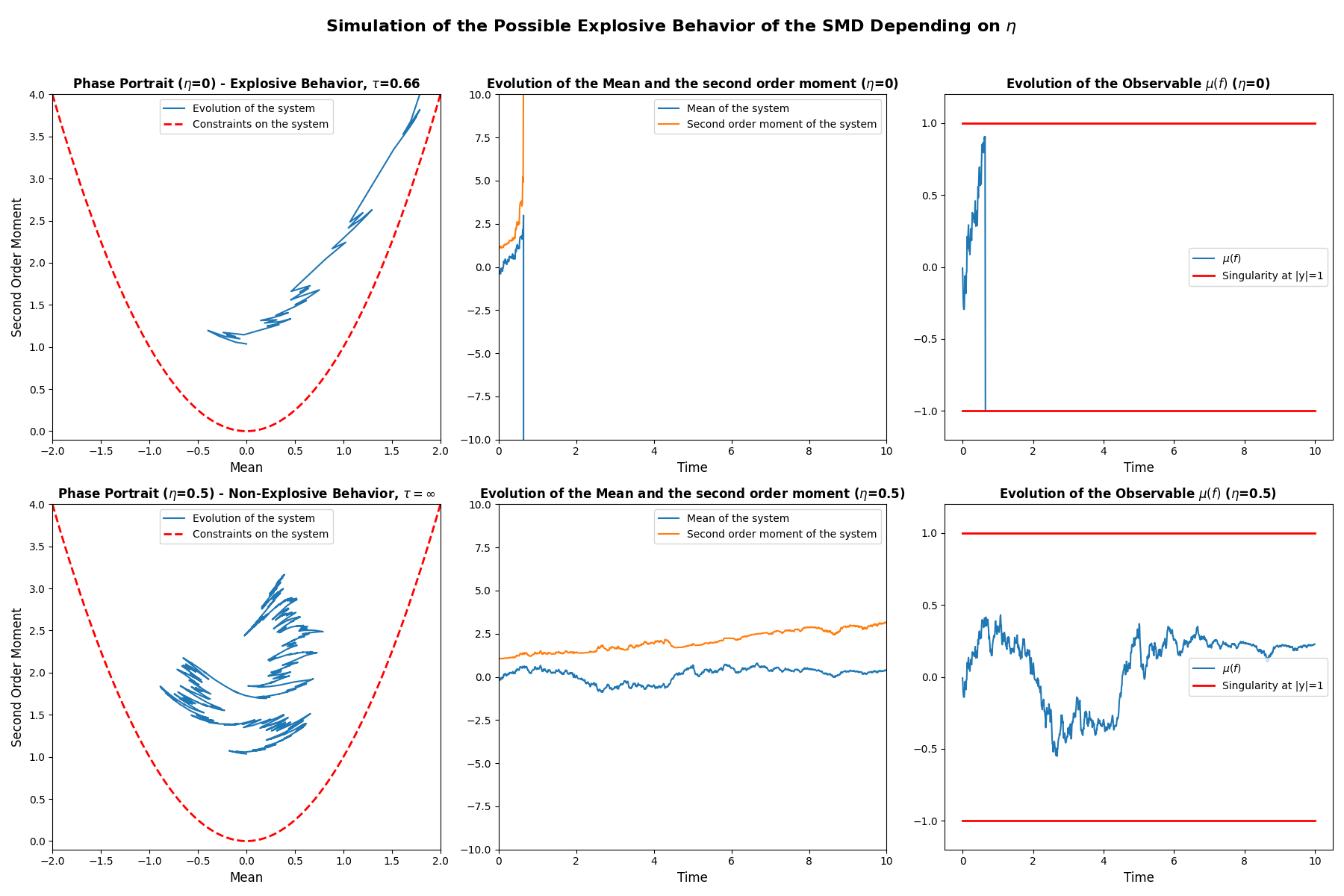}
  \caption{Explosive/Non explosive behavior of the regularized SMD with $f(x)=\tanh(x)$ depending on the parameter $\eta$  (top: $\eta=0$; bottom: $\eta=0.5$)}\label{fig:3}
\end{figure}

 The constraints on the evolution depends on the observable: the first is constrained to be positive while the second must live in $(-1,1)$. Because $\mu_t(f)$ is designed to evolve as a brownian motion when $\eta=0$ (which may reach the singularities of both processes), both non-regularized dynamics are explosive. Nevertheless, the regularized evolutions are non-explosive (but $\mu(f)$ is no longer a brownian motion). These simulations show that from a practical point of view it is interesting to numerically implement the regularized versions of the SMD if one wants to perturb the moments of a particle system and still get a stable dynamic.

Looking at the phase portrait of both processes (left pictures in Figures~\ref{fig:2} and \ref{fig:3}, representing the trajectory of the mean and second order moment of $\mu_t$, independently from the observable $f$), we confirm the intuition that the choice of the perturbed observable strongly influences the region of the probability space explored by the SMD. Thus, one should adapt the choice of $f$ depending on the problem considered to get the most of the exploration properties of the SMD. The following subsection illustrates the interest of the addition of a SMD term in the McKean Vlasov dynamic \eqref{eq:mcKeanPS}.

\subsection{SMD is a suitable exploration noise for the double well McKean Vlasov dynamic}
 
Finally, we can illustrate the interest of the addition of an SMD term in the McKean Vlasov dynamic as a suitable exploration term for the minimization of the free energy \ref{eq:energy}. Indeed, we are going to show that numerically the ``noisy version" of the Wasserstein gradient flow can make transitions between the basins of attraction of the two stable global minimizers $\rho_{-}$ and $\rho_{+}$. In order to chose what observable $f$ should be controlled, it is important to first understand the qualitative behavior of the McKean Vlasov dynamic. When the thermal agitation is small enough, the diffusion phenomenon is not enough to compensate the deterministic drift, and so the particles tend to aggregate in one of the two potential well. The limiting distributions ($\rho_{-}$ and $\rho_{+}$) are then concentrated around one of the minima of the potential with a small variance. In order to allow transitions between $\rho_{-}$ and $\rho_{+}$, one should encourage the particles not to stay grouped at the same place. To do so, we can control the mean of the particle system so that it does not stay close to the same minima of the potential, and the variance of the system to avoid that the particles stay aggregated together. That is why we chose to control the observable $f(x)=(x,x^2)$ exactly like in Section \ref{sec:secondexemple}, but one should keep in mind that for a different problem another observable may be more adapted (besides, in this specific example, randomizing the mean alone would be sufficient to get an ergodic process, as shown in \cite{delarue2024ergodicity}).

We simulate the evolution of the particle system associated to the following coefficients:
\begin{equation}
\left\{
        \begin{array}{lll}
        b_{\mu}(x)= \gamma \cdot \frac{\delta-\frac{3}{2}}{4\text{Var}(\mu)^2}(x-m(\mu)) \quad \quad \sigma_{\mu}(x)= \gamma \cdot  \Big(1 \hspace{3pt}, \hspace{3pt}\frac{\delta-\frac{3}{2}}{2\text{Var}(\mu)} (x-m(\mu)) \Big)\\
        \\
    
        \tilde{b}_\mu(x)=-\nabla U(x) -\mu \ast \nabla W(x) \hspace{1cm} \tilde{\sigma}_\mu(x)=\tilde{\sigma} \,,
        \end{array}\right.
\end{equation}
where $\gamma \geq0$ is a parameter corresponding to the intensity of the exploration noise and $\delta\geq 2$ the dimension of the underlying Bessel process (see Section \ref{sec:secondexemple}). The system is made of $N=1000$ particles with a diffusive parameter  $\tilde{\sigma}=0.7$, and the particles are initialized with independent variables with Gaussian distribution $\mathcal{N}(-\frac{3}{2},\frac{1}{2}$), so that the non perturbed process naturally converges to $\rho_-$. We noticed that the evolutions had a non-explosive behavior so we did not add a regularization parameter $(\eta=0)$. The results are displayed in Figure~\ref{fig:4}.

\begin{figure}[!ht] \centering
  \includegraphics[width=180mm, scale=1]{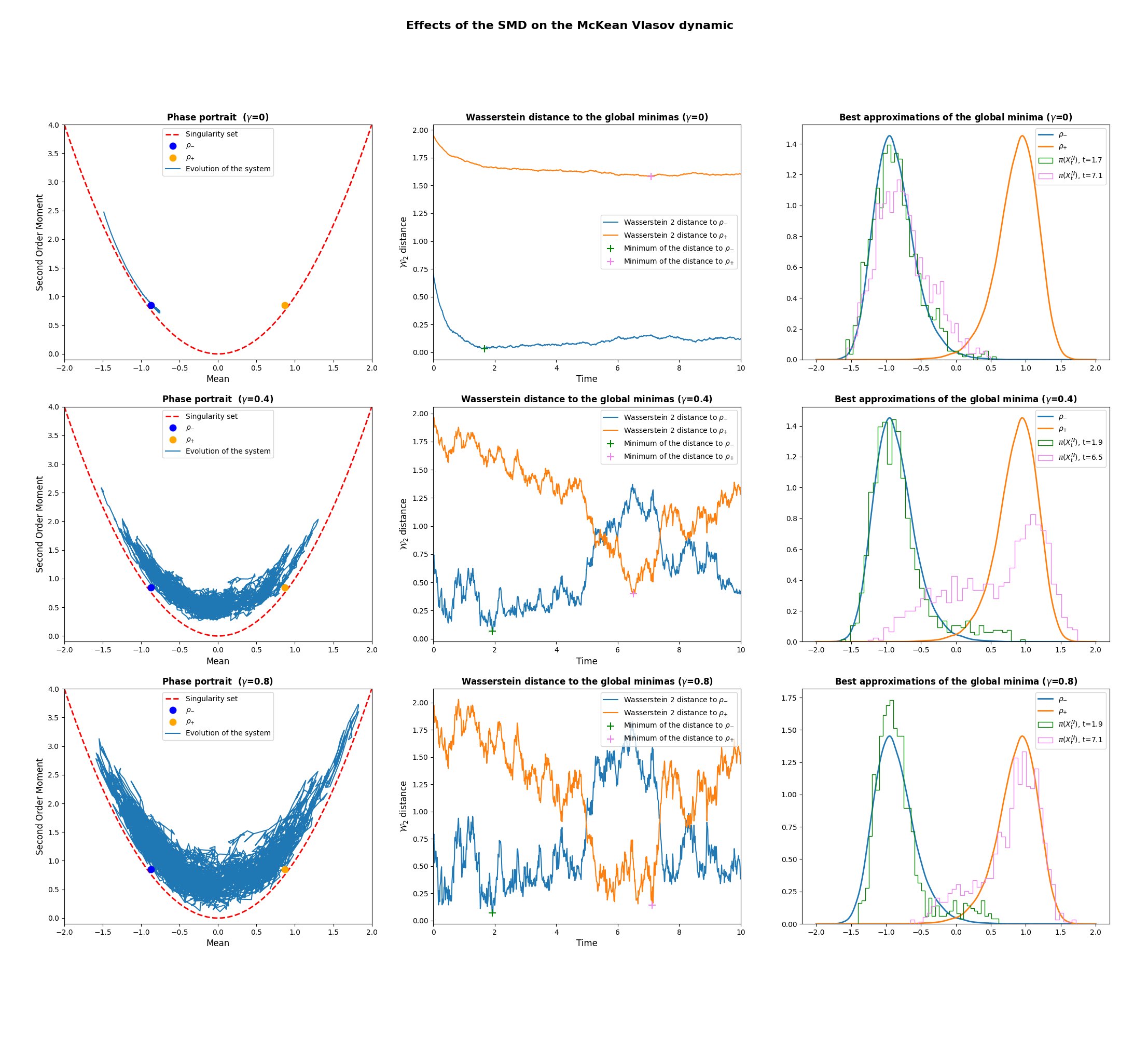}
  \caption{Evolution of the "noisy" McKean Vlasov dynamic depending on the intensity of the noise $\gamma$}\label{fig:4}
\end{figure}

In the figure, each line corresponds to a different intensity of the exploration noise $\gamma$ ($0$, $0.4$ and $0.8$ from top to bottom). The first column corresponds to the evolution of the particle system over time in the phase portrait mean/ second order moment, with the position of the two global minimizers of the free energy represented with a colored dot. The dashed red line corresponds to the intrinsic constraints given by Jensen's inequality, and the process must stay above the line to avoid explosion. The second column represents the Wasserstein 2 distance of the particle system to the two global minimizers over time (which have been estimated in advance through the long term convergence of the classical McKean Vlasov particle system). The two cross-points (green and pink) represent the moment where the distance was minimal, and the distribution of the particle system at these instants is represented next to the two global minima in the last column of the figure.

When the intensity of the exploration noise $\gamma=0$, the evolution corresponds to the classical McKean Vlasov dynamic: the system evolves almost deterministically (up to a small random perturbation resulting from the approximation of the deterministic dynamic by a particle system) to one of the global minimizers. The process is not  exploring the space at all and never approaches the other global minimizer during the simulation time. 

When $\gamma$ increases, the process explores the space of probability measures and makes transitions between the basins of attraction of the two global minimizers. As a consequence, the process passes closer to the second minimizer $\rho_{+}$, resulting in a better approximation of $\rho_{+}$ along the trajectory (column 3). However, this exploration noise induces a bias in the deterministic dynamic which is getting bigger as $\gamma$ grows: if one wishes to determine the minima of a more complicated energy, the research of the minima should be done in two step. First, one should use the exploration properties of the SMD by simulating a noisy version of gradient flow and store the minimizers of the energy along the trajectory. Then, in order to refine their expression, one should start a deterministic gradient descent initialized from these candidates, in order to avoid the additional bias coming from the SMD.

This numerical experiment illustrates the exploration properties of the stochastic moment dynamic, and shows on a toy example how it can be used to solve a non-convex optimization problem over the space of probability measures.

\section*{Acknowledgments}

The research of P. Monmarché is supported by the project  CONVIVIALITY (ANR-23-CE40-0003) of the French National Research Agency  and by the European Research Council (ERC) under the European Union’s
Horizon 2020 research and innovation program, project EMC2 (grant agreement N°810367).

 \section{Appendix}\label{sec:appendix}

\subsection{Proof of Proposition \ref{prop:stochasticmomentswelldefined}}

\subsubsection{Proof of Lemma \ref{lem:regucoefsursomaines} }
\label{app:proof1stlemma}

Before starting the proof of the first lemma, we recall basic results on functions of $\mu$ of the form $\mu \rightarrow \mu(g)$ with $g$ a function with at most a polynomial growth. To do so, in order to clarify the proof, we give a name to the set of functions with at most polynomial growth:

\begin{defi}\label{def:ensemblel}
    Let $\theta \geq 0$, $q,r\in \N^*$. We define  $L^{(\theta)}(\R^d,M_{q,r}) $ the set of functions $f:\R^d \rightarrow M_{q,r}$ such that it exists a constant $C\geq0$ such that $$\forall x,y \in \R^d: \quad |f(x)-f(y) |\leq C|x-y|(1+|x|^{\theta}+|y|^{\theta})$$
\end{defi}

The choice of this notation is justified by the basic observation that the parameter $\theta$ can be seen as the growth rate of the Lipschitz constant of the -function. Then in dimension $d=1$, the monomials $f:x \mapsto x^{\theta+1}$ belong to the set $L^{(\theta)}(\R,\R)$ (because of the mean value theorem). Moreover, because of the equivalence of the norms for finite dimensional vector spaces, the choice of the norm has no influence on the definition of the set. This notation will be useful to recall some basic properties of these sets of functions.

\begin{lem}\label{lem:croissancepolynomiale}
    Let $a$ and $b$ be positive, and let $m,n,l$ be three positive integers. Suppose that $g\in L^{(\theta_1)}(\R^d,M_{m,n})$ and $h\in L^{(\theta_2)}(\R^d,M_{m,n})$. The following properties hold true.
    \begin{enumerate} \label{lem:propensemble}
        \item (Product) $ gh \in L^{(\theta_1+\theta_2+1)}(\R^d,M_{m,l})$ 
        \item (Product by a constant matrix) For $A \in M_{m,n}(\R)$ a constant matrix,  $A h \in L^{(\theta_2)}(\R^d,M_{m,l})$, and the associated constant grows at most linearly in the norm of $A$.
        \item (Evaluation on a coordinate) For all $i,j \in \{1,..m\}\times \{1,..,n\}$ , $g_{i,j} \in L^{(\theta_1)}(\R^d,\R)$ 
        \item (Stability) Suppose that $\theta_1\leq \theta_2$ , then  $g\in L^{(\theta_2)}(\R^d,M_{m,n})$ 
        \item(Transposition) $g^T \in L^{(\theta_1)}(\R^d,M_{n,m})$.
    \end{enumerate}
\end{lem}

The first result follows from the convexity of the logarithm, ensuring that expressions of type $|x|^{\theta_1}|y|^{{\theta_2}+1} $ are smaller than $|x|^{\theta_1+\theta_2+1}+|y|^{\theta_1+\theta_2+1}$. The proof of the other results is straightforward. We will use these results later in the proof, to justify that some observables will have a polynomial growth. We will also use the notion of moment of order $\gamma \geq 1$ of law defined by:
$$\displaystyle m_{\gamma}:\mathcal{P}_{\gamma}(\R^d) \rightarrow \R^{+} , \quad m_{\gamma}: \mu \mapsto \int_{\R^d} |x|^{\gamma} d\mu(x)=\mathcal{W_\gamma}(\mu,\delta_0)^\gamma$$

The next lemma is the reason justifying the choice of the hypothesis on the observables.

\begin{lem} \label{lem:evaluationobservable}
    Let $\theta\geq 0$, $m,n$ be two positive integers, and $g\in L^{(\theta)}(\R^d,M_{m,n})$. Then the map defined on $\mathcal{P}_{\theta+1}(\R^d)$ by $\mu \mapsto \mu(g)$ is Lipschitz and bounded on the bounded subsets of $\mathcal{P}_{\theta+1}(\R^d)$. In other words, for all $M\geq 0$ the map $\mu \mapsto \mu(g)$ is Lipschitz bounded on the set $\{\mu \in \mathcal{P}_{\theta+1}(\R^d), \quad m_{\theta+1}(\mu) \leq M \}$
\end{lem}

\begin{proof}
    The proof of this result is an application of the Holder inequality. Let $\theta\geq 0$, $m,n$ be two positive integers, $M\geq 0$, and $g\in L^{(\theta)}(\R^d,M_{m,n}(\R^d))$. Let $\mu,\nu \in \mathcal{P}_{\theta+1}(\R^d)$, such that $m_{\theta+1}(\mu)$ and $m_{\theta+1}(\nu)$ are bounded by $M$. Let $\pi \in \mathcal{C}(\mu,\nu)$ be an optimal coupling between $\mu$ and $\nu$ for the $\mathcal W_{\theta+1}$ distance:
    $$\mathcal W_{\theta+1}(\mu,\nu)^{\theta+1}=\int_{\R^d\times \R^d}|x-y|^{\theta+1}d\pi(x,y)$$ 
    The following inequalities hold true:
    \begin{equation*}
            \begin{split}
                |\mu(g)-\nu(g)|&=| \int_{\R^d}g(x)d\mu(x)-\int_{\R^d}g(y)d\nu(y)|\\
                &=|\int_{\R^d\times \R^d} g(x)-g(y)d\pi(x,y) |\\
                &\leq \int_{\R^d\times \R^d}C |x-y|(1+|x|^\theta+|y|^\theta) d\pi(x,y) \\
                & \leq C\Bigl[\int_{\R^d\times \R^d}|x-y|^{\theta+1} d\pi(x,y)\Bigr]^{\frac{1}{\theta+1}} \Bigl[\int_{\R^d\times \R^d}(1+|x|^\theta+|y|^\theta)^{\frac{\theta+1}{\theta}} d\pi(x,y)\Bigr]^{\frac{\theta}{\theta+1}} \\
            \end{split}
        \end{equation*}
    Because of the Holder inequality applied with $\theta+1$ and $\frac{\theta+1}{\theta}$. Moreover, by convexity, we can find a constant $C_\theta$ such that for all $x$ and $y$:
    $$(1+|x|^\theta+|y|^\theta)^{\frac{\theta+1}{\theta}}\leq C_\theta(1+|x|^{\theta+1}+|y|^{\theta+1}) $$
    Finally, we have :

    \begin{equation*}
       \begin{split}
            |\mu(g)-\nu(g)| &\leq C C_\theta \mathcal W_{\theta+1}(\mu,\nu) (1+m_{\theta+1}(\mu)+m_{\theta+1}(\nu))^{\frac{\theta}{\theta+1}} \\
            & \leq C C_\theta( 1+2M )^{\frac{\theta}{\theta+1}}\mathcal W_{\theta+1}(\mu,\nu)\,,
       \end{split}
    \end{equation*}
    Which proves the result.
    
\end{proof}

Before starting the main step of the proof, let us recall that because the Wasserstein distance of order $\theta$ between two measure is non-decreasing in $\theta$, then if a map defined on $\mathcal{P}_{\theta}(\R^d)$ is Lipschitz bounded on bounded subsets of $\mathcal{P}_{\theta}(\R^d)$, then it is still the case in every Wasserstein space of order $\theta'$ with $\theta'\geq \theta$. Then, the previous result, even if stated in the Wasserstein space $\mathcal{P}_{\theta+1}(\R^d)$, is also true in every Wasserstein space of order bigger than $\theta+1$. We can now prove Lemma \ref{lem:regucoefsursomaines}.

\begin{proof}
    Let $K>0$ and $M>0$. We first show the result for $\sigma$ then we show it for $b$.
    \medskip{}

    \underline{\textit{1) Study of $\sigma$}}: $\sigma$ is written as the product of three function of $x$ and $\mu$, respectively $x, \mu \mapsto \nabla f(x)$,  $x,\mu \mapsto \mu(\nabla f^T \nabla f)^{-1}$ and $x,\mu \mapsto s(\mu(f))$. We show that they are all Lipschitz and bounded on $\mathcal{D}^{K,M}$, and we conclude by the fact that a product of bounded Lipschitz functions is Lipschitz bounded.

    \medskip{}
    
    \textit{a) Study of $\nabla f$:} Because for all $x,\mu$ in $\mathcal{D}^{K,M}$, the norm of $x$ is bounded, and because $\nabla f$ is Lipschitz with at most a polynomial growth,  $x,\mu \mapsto \nabla f(x)$ is Lipchitz bounded on $\mathcal{D}^{K,M}$.

    \medskip{}
    
    \textit{b) Study of $\mu(\nabla f^T \nabla f)^{-1}$:} If $f$ verify $(A.1)$ or $(A.2)$, $\nabla f^T \nabla f \in  L^{(1)}(\R^d,M_{p,p})$ and if $f$ verify $(A.3)$, $\nabla f^T \nabla f \in  L^{(2\alpha_2+1)}(\R^d,M_{p,p})$ (Lemma \ref{lem:croissancepolynomiale}). Then, in any case Lemma \ref{lem:evaluationobservable} ensures that $\mu \mapsto \mu(\nabla f^T \nabla f)$ is Lipschitz bounded on bounded subsets of $\mathcal{P}_\alpha(\R^d)$ (because $2\alpha_2+1\leq \alpha$). We can then find $A_M \geq 0$ such that for all $\mu$ in $\{\mu \in  \mathcal{P}_\alpha(\R^d): m_\alpha(\mu) \leq M \}$, $|\mu(\nabla f^T\nabla f)| \leq A_M$.

    \medskip{}

    We can now define the compact set $G^M:=\{ B \in M_{p,p}(\R^d) : |B| \leq A_M , \quad \text{det}(B) \geq \frac{1}{M} \}$. We consider the map:
    $$ \begin{array}{ccccc}
\text{Inv}^M& : & G^M & \longrightarrow & M_{p,p}(\R^d) \\
 & & B & \longmapsto &\displaystyle \frac{1}{\text{det}(B)}\text{com}^T(B) \\
\end{array}$$ 
This map is smooth on the compact set $G^M$ so it is naturally Lipschitz and bounded on it. Then, by composition, the map $\mu \mapsto \text{Inv}^M(\mu(\nabla f^T \nabla f))$ is Lipschitz and bounded on the set:
$$\{\mu \in \mathcal{P}_\alpha(\R^d): \hspace{3pt} m_\alpha(\mu) \leq M ,\hspace{3pt} \text{det}(\mu(\nabla f ^T \nabla f)) \geq \frac{1}{M} \}\,.$$ 

Moreover, we have that $\mu \mapsto \mu(\nabla f^T \nabla f)^{-1}$ coincide with $\mu \mapsto \text{Inv}^M(\mu(\nabla f^T \nabla f))$ on this set. Then, by definition of $\mathcal{D}^{K,M}$, the map $x,\mu \mapsto \mu(\nabla f^T \nabla f)^{-1}$ is Lipschitz bounded on $\mathcal{D}^{K,M}$.

\medskip{}

\textit{c) Study of $s(\mu(f)) $}: Similarly, the assumption on $f$ and the Lemma \ref{lem:evaluationobservable} ensure that $\mu \mapsto \mu(f)$ is Lipschitz bounded on bounded subsets of $\mathcal{P}_\alpha(\R^d)$, and so we can find $\tilde{A}_M\geq 0$ such that for all $\mu$ in $\{\mu \in  \mathcal{P}_\alpha(\R^d): m_\alpha(\mu) \leq M \}$, $|\mu(f)| \leq \tilde{A}_M$.

Thanks to the assumption on $s$, with $\mathcal{S}$ the set of singularity points, the restriction of $s$ to the set $$(\mathcal{S}_{\frac{1}{M}})^c \cap B(0,\tilde{A}_M) \subset \R^p$$ is Lipschitz and bounded. By composition, we finally have that $x,\mu \mapsto s(\mu(f))$ is Lipschitz and bounded on 
$$\{ x,\mu \in \R^d \times \mathcal{P}_\alpha(\R^d): \hspace{3pt} m_{\alpha}(\mu) \leq M , \hspace{3pt} d(\mu(f), \mathcal{S})\geq \frac{1}{M} \}$$ Finally, by definition of $\mathcal{D}^{K,M}$, $x,\mu \mapsto s(\mu(f))$ is Lipschitz and bounded on $\mathcal{D}^{K,M}$, which concludes the study of $\sigma$.

\medskip{}

\underline{\textit{2) study of $b$:}} We do the same reasoning than for the study of $\sigma$: we decompose $b$ as a product of functions, and we show that each is Lipschitz bounded on $\mathcal{D}^{K,M}$. We already showed that $x,\mu \mapsto \nabla f (x) \mu(\nabla f^T\nabla f)^{-1}$ is Lipschitz bound on $\mathcal{D}^{K,M}$. By doing the same proof than for $s(\mu(f))$, we can show that $x,\mu \mapsto a(\mu(f))$ is Lipschitz and bounded on $\mathcal{D}^{K,M}$. We only need to study of $x,\mu \mapsto \mu(\sigma_{\mu}\sigma_{\mu}^T :\nabla^2 f) $: in order to simplify the notations we will denote $C$ a positive constant (independent of $x$ and $\mu$), that may change from line to line. 

\medskip{}

First, we develop the expression of $x,\mu \mapsto \mu(\sigma_{\mu}\sigma_{\mu}^T :\nabla^2 f) $, in order to express it as a composition of several maps on $\mathcal{D}^{K,M}$.

\medskip{}

\textit{a) Development of $\mu(\sigma_{\mu}\sigma_{\mu}^T :\nabla^2 f)$:} this map has values in $\R^p$, and for all index $1\leq k \leq p$:
\begin{equation} \label{eq:defLetT}
     \begin{split}
         (\mu(\sigma_{\mu}\sigma_{\mu}^T :\nabla^2 f))_k&=\mu(\sigma_{\mu}\sigma_{\mu}^T :\nabla^2 f_k)\\
         &=\sum_{i,j=1}^d \mu((\sigma_{\mu}\sigma_{\mu}^T)_{i,j} \hspace{3pt} \partial^2_{i,j} f_k) \\
         &= \sum_{i,j=1}^d \int_{\R^d} [\nabla f (x) A(\mu) \nabla f ^T(x))_{i,j} \hspace{3pt} \partial^2_{i,j} f_k(x) \hspace{3pt}] d\mu(x)
     \end{split}
 \end{equation}
where $A:\mathcal{P}_\alpha(\R^d)\rightarrow M_{p,p}$ is defined by:
$$A:\mu \mapsto \mu(\nabla f^T \nabla f)^{-1} s(\mu(f)) s(\mu(f))^T (\mu(\nabla f^T \nabla f)^{-1})^T\,.$$
If we define the set $$\Lambda^M:=\{\mu \in \mathcal{P}_\alpha(\R^d):\hspace{3pt} m_\alpha(\mu) < M , \hspace{3pt} d(\mu(f), \mathcal{S})< \frac{1}{M}, \hspace{3pt} \text{det}(\mu(\nabla f ^T \nabla f)) > \frac{1}{M} \}$$
by definition of $\mathcal{D}^{K,M}$, we naturally have that $$\mathcal{D}^{K,M}=B(0,K) \times \Lambda^M$$ where $B(0,K)\subset \R^d$ is the open centered euclidean ball of radius $K$.

Since $A(\mu)$ is defined as a product of bounded Lipschitz maps on $\Lambda^M$, it is also Lipschitz and bounded on this set, and so the map

\begin{equation} \label{eq:defT}
    T:\mu \in \mathcal{P}_\alpha(\R^d) \rightarrow M_{p,p} \times \mathcal{P}_\alpha(\R^d), \quad T:\mu \mapsto (A(\mu),\mu)
\end{equation}

is also Lipschitz and bounded on $\Lambda^M$. Let $V_M$ be the maximum value of $A$ on $\Lambda_M$, so that $A$ is bounded by $V_M$ on this set.

In order to clarify the expression of $(\mu(\sigma_{\mu}\sigma_{\mu}^T :\nabla^2 f))_k$, we define the map:

$$L^M_k:\{B \in M_{p,p}(\R), \hspace{3pt} |B| \leq V_M \} \times \Lambda^M \rightarrow \R$$
by
\begin{equation} \label{eq:defL}
    L^M_k(B,\mu)= \sum_{i,j=1}^d \int_{\R^d} [\nabla f (x) B \nabla f ^T(x))_{i,j} \hspace{3pt} \partial^2_{i,j} f_k(x) \hspace{3pt}] d\mu(x)
\end{equation}

We naturally denote $L^M(B,\mu):=(L^M_k(B,\mu))_{k=1}^p$, and thanks to Equations~\eqref{eq:defLetT}, \eqref{eq:defT} and \eqref{eq:defL}, we have the following equality for all $\mu \in \Lambda^M$:

\begin{equation} \label{eq:compositionLT}
    \mu(\sigma_{\mu}\sigma_{\mu}^T :\nabla^2 f)=L(A(\mu),\mu)=L\circ T(\mu)
\end{equation}

As we already showed that $T$ is Lispchitz bounded on $\Lambda^M$, we have to study $L$ to prove the desired result. We will prove that for all index $1\leq k \leq p$, we have that $L_k^M$ is Lipschitz and bounded, which will conclude the overall proof.

\medskip{}

\textit{b) Study of $L_k^M$}: we now want to show that $L_k^M$ is Lipschitz bounded on the set $\{B \in M_{p,p}(\R), \hspace{3pt} |B| \leq V_M \} \times \Lambda^M$. To do so, let $\mu_1,\mu_2 \in \Lambda^M$ and $B_1,B_2 \in \{B \in M_{p,p}(\R), \hspace{3pt} |B| \leq V_M \}$ et $k \in \{1,..,p\}$. We have that:
\begin{equation} \label{eq:deuxtermes}
    \begin{split}
        |L_k^M(B_1,\mu_1)-L_k^M(B_2,\mu_2)|  \leq & |L_k^M(B_1,\mu_1)-L_k^M(B_1,\mu_2)|+ |L_k^M(B_1,\mu_2)-L_k^M(B_2,\mu_2)| \\
        \leq & \sum_{i,j=1}^d|  \hspace{3pt} \mu_1((\nabla f B_1 \nabla f ^T)_{i,j} \hspace{3pt} \partial^2_{i,j} f_k )- \mu_2((\nabla f B_1 \nabla f ^T)_{i,j} \hspace{3pt} \partial^2_{i,j} f_k ) \hspace{3pt} | \\
        & + \sum_{i,j=1}^d \mu_2( \hspace{2pt} |( \hspace{2pt} \nabla f (B_1- B_2)\nabla f ^T)_{i,j} \hspace{3pt} \partial^2_{i,j} f_k  \hspace{2pt} | \hspace{2pt})
    \end{split}
\end{equation}
We deal with each of these two terms independently.

\medskip{}

\textit{i)First term of equation \eqref{eq:deuxtermes}}: Because $|B_1|\leq V_M$, Lemma \ref{lem:croissancepolynomiale} ensures that for all index $i$ and $j$ the observable $x \mapsto (\nabla  f(x) B_1 \nabla f (x)^T)_{i,j} \hspace{3pt} \partial^2_{i,j} f_k(x) \in L^{(\alpha)}(\R^d,\R)$, with a constant independent of $B_1$ (because $B_1$ is bounded). Then Lemma \ref{lem:evaluationobservable} justify that we can find a constant $C$ independent of $B_1$ such that for all index $i$ and $j$:
$$|  \hspace{3pt} \mu_1((\nabla f B_1 \nabla f ^T)_{i,j} \hspace{3pt} \partial^2_{i,j} f_k )- \mu_2((\nabla f B_1 \nabla f ^T)_{i,j} \hspace{3pt} \partial^2_{i,j} f_k ) \hspace{3pt} | \leq C \, W_\alpha(\mu_1,\mu_2)$$
and thus, (without changing the name of the constant for ease of notations):
\begin{equation} \label{eq:terme1}
    \sum_{i,j=1}^d|  \hspace{3pt} \mu_1((\nabla f B_1 \nabla f ^T)_{i,j} \hspace{3pt} \partial^2_{i,j} f_k )- \mu_2((\nabla f B_1 \nabla f ^T)_{i,j} \hspace{3pt} \partial^2_{i,j} f_k ) \hspace{3pt} | \leq C \, W_\alpha(\mu_1,\mu_2)\,,
\end{equation}
which corresponds to the desired upper bound on the first term of Equation \eqref{eq:deuxtermes}.

\medskip{}

\textit{ii) Second term of equation \ref{eq:deuxtermes}}: For $i$ and $j$ two indices, because of the equivalence of the norms in finite dimensions, we can find a multiplicative constant $C$ so that:
\begin{equation*}
             \begin{split}
                 \mu_2( \hspace{2pt} | \hspace{2pt} \nabla f (B_1- B_2)\nabla f ^T)_{i,j} \hspace{3pt} \partial^2_{i,j} f_k  \hspace{2pt} |) &\leq  \mu_2( C \, |\nabla f | |B_1-B_2| |\nabla f ^T| |\nabla^2 f_k|) \\
                 & \leq C \, |B_1-B_2| \, \mu_2( |\nabla f | |\nabla f ^T| |\nabla^2 f_k|)\,.
             \end{split}
\end{equation*}
Then, because of the hypothesis on $f$ and its derivatives, we can find a constant such that $\mu_2( |\nabla f | |\nabla f ^T| |\nabla^2 f_k|) \leq C(1+m_\alpha(\mu_2)) $ and because $\mu_2\in \Lambda^M$ this term is uniformly bounded in $\mu_2$. As a consequence,
$$\mu_2( \hspace{2pt} | \hspace{2pt} \nabla f (B_1- B_2)\nabla f ^T)_{i,j} \hspace{3pt} \partial^2_{i,j} f_k  \hspace{2pt} |) \leq C \, |B_1-B_2|$$
Next, for all index $k$, $L_k^M$ is Lipschitz bounded on the set $\{B \in M_{p,p}(\R), \hspace{3pt} |B| \leq V_M \} \times \Lambda^M$, and so by Equation \eqref{eq:compositionLT}, $\mu \mapsto \mu(\sigma_{\mu}\sigma_{\mu}^T :\nabla^2 f)$ is Lipschitz bounded on $\Lambda^M$. Finally, because $b$ is a product of Lipschitz bounded functions on $D^{K,M}$,  it is itself a Lipschitz bounded function on $\mathcal{D}^{K,M}$, which concludes the proof.

\end{proof}

\subsubsection{Proof of Lemma \ref{lem:troncatures}} \label{app:prooftroncature}

In order to prove the result, we first present the function $\chi^{K,M} $ that will play the role of the indicator function and then we show that the truncated coefficients are indeed Lipschitz and bounded on the whole space $\R^d\times \mathcal{P}_\alpha(\R^d)$.

\begin{proof}

First, we define for any $r>0$, the following functions defined from $\R^+$ to $[0,1]$, that correspond to Lipschitz approximations of respectively $\1_{[0,r]}$ and $\1_{[\frac{1}{r},+\infty[}$: 
    $$H_r(u):=\Biggl\{
        \begin{array}{llll}
        \displaystyle
        1 \quad \text{if $u\leq r$} \\
        r+1-u \quad \text{si $r<u< r$}\\
        0 \quad \text{si  $u\geq r+1$}
        \end{array}
        \Biggr.$$
    and 
     $$I_r(u):=\Biggl\{
        \begin{array}{llll}
        \displaystyle
        0 \quad \text{if $u\leq \frac{1}{r+1}$} \\
        r(r+1)u-r \quad \text{si $\frac{1}{r+1}<u<\frac{1}{r}$}\\
        1 \quad \text{si  $u\geq \frac{1}{r}$}
        \end{array}
        \Biggr.$$

   We now define $\chi^{K,M}:\R^d \times \mathcal{P}_\alpha(\R^d)\rightarrow [0,1]$ by:
    \begin{equation}
        \label{eq:def_truncation}
    \chi^{M,K}:x , \mu \mapsto H_K(|x|) \hspace{2pt} \cdot \hspace{2pt}  H_M(m_\alpha(\mu)) \hspace{2pt} \cdot \hspace{2pt}  I_M(d(\mu(f),\mathcal{S}))\hspace{2pt} \cdot \hspace{2pt} I_M(\text{det}(\mu(\nabla f ^T \nabla f))) 
    \end{equation}

    By construction, $\chi^{M,K}$ is bounded, equal to $1$ on $\mathcal{D}^{K,M}$, equal to $0$ on $(\mathcal{D}^{K+1,M+1})^c$, and Lipschitz on $\overline{\mathcal{D}^{K+1,M+1}}$. It follows that the truncates versions of the coefficients $b,\tilde{b},\sigma,\tilde{\sigma}$ defined on $\R^d \times \mathcal{P}_\alpha(\R^d)$ by:
        \begin{equation} \label{eq:truncated}
        b^{M,K}(x,\mu)=\chi^{M,K}(x,\mu)b_\mu(x)\quad \sigma^{M,K}(x,\mu) =\chi^{M,K}(x,\mu)\sigma_\mu(x)
        \end{equation}
         (same definition for $\tilde{b}^{K,M}$ and $\tilde{\sigma}^{K,M}$) are Lipschitz on $\overline{\mathcal{D}^{K+1,M+1}}$ and equals to $0$ outside of $\mathcal{D}^{K+1,M+1}$ and coincide with $b$ and $\sigma$ on $\mathcal{D}^{K,M}$. We are going to show that the truncation ensures that the coefficients are globally Lipschitz: we show this result for a generic function $F:\R^d \times \mathcal{P}_\alpha(\R^d) \rightarrow \R$ Lipschitz on $\overline{\mathcal{D}^{K+1,M+1}}$ and equals to $0$ outside of $\mathcal{D}^{K+1,M+1}$.
         
         \medskip{}
         
         Let $\mu_0,\mu_1 \in \mathcal{P}_\alpha(\R^d) $ and $x_0,x_1 \in \R^d$, we want to show that we can find a constant $C$ so that
        \begin{equation}\label{eq:distinctioncasdomaine}
            |F(x_0,\mu_0)-F(x_1,\mu_1)|\leq C \Bigl[ |x_0,x_1|+\mathcal W_\alpha(\mu_0,\mu_1) \Bigr]
        \end{equation}
        
        We show that this result holds with $C$ the Lipschitz constant of $F$ on $\mathcal{D}^{K+1,M+1}$ 

        \medskip{}

        If $x_0,\mu_0$ and $x_1,\mu_1$ are both inside or outside of $\mathcal{D}^{K+1,M+1}$, equation \ref{eq:distinctioncasdomaine} is directly verified.

        Suppose that  $x_1,\mu_1 \in \mathcal{D}^{K+1,M+1}$ and $x_0,\mu_0 \notin \mathcal{D}^{K+1,M+1}$ (the symetric case can be proven the same way). If $F(x_1,\mu_1)= 0 $, the inequality is obviously true. Otherwise $F(x_1,\mu_1)\neq  0 $,  and for $t\in[0,1]$, let $x_t:=(1-t)x_0+tx_1$ and $(\mu_t)$ be a geodesic from $\mu_0$ to $\mu_1$ in $\mathcal{P}_\alpha(\R^d)$. The classical theory of optimal transport ensures that such a geodesic always exists and verify the following equality:
            \begin{equation} \label{eq:geodesique}
                \mathcal W_\alpha(\mu_1,\mu_t)=(1-t)\mathcal W_\alpha(\mu_0,\mu_1)
            \end{equation}
            Moreover, the hypothesis on $F$ ensures that the map $$t\in [0,1] \mapsto F(x_t,\mu_t)$$ is continuous, equal to zero for $t=0$ and non-zero for $t=1$. Let $t_0=\text{inf}\{0\leq t \leq 1: \, F(x_t,\mu_t) \neq 0 \}$ ($t_0$ exists because $F(x_1,\mu_1)\neq  0$ ). By continuity and because $\overline{\mathcal{D}^{K+1,M+1}}$ is closed, $(x_{t_0},\mu_{t_0}) \in \overline{\mathcal{D}^{K+1,M+1}}$ and $F(x_{t_0},\mu_{t_0})=0$. Then:
            \begin{equation*}
                \begin{split}
                    |F(x_0,\mu_0)-F(x_1,\mu_1)|= |F(x_{t_0},\mu_{t_0})-F(x_1,\mu_1)|
                \end{split}
            \end{equation*}
            Then, because $F$ is Lipschitz in $\overline{\mathcal{D}^{K+1,M+1}}$ and because of Equation \eqref{eq:geodesique}:
            \begin{equation*}
                \begin{split}
                    |F(x_{t_0},\mu_{t_0})-F(x_1,\mu_1)| & \leq C \Bigl[ |x_{t_0},x_1|+\mathcal W_\alpha(\mu_{t_0},\mu_1) \Bigr] \\
                    & \leq C (1-t) \,\Bigl[  |x_{0},x_1|+\mathcal W_\alpha(\mu_{0},\mu_1) \Bigr] \\
                    & \leq C\Bigl[  |x_{0},x_1|+\mathcal W_\alpha(\mu_{0},\mu_1) \Bigr] 
                \end{split}
            \end{equation*}
            which implies that $$|F(x_0,\mu_0)-F(x_1,\mu_1)| \leq C\Bigl[  |x_{0},x_1|+\mathcal W_\alpha(\mu_{0},\mu_1) \Bigr]$$
            and this give the desired result.

\end{proof}

\subsection{Proof of Lemma \ref{lem:boundmoments}} \label{app:proofmomentbounds}

\begin{proof}
    First, suppose that $\tilde{b} $ satisfy assumption (B.2'), with parameter $q\geq 2$, let $\gamma \geq 2$ and $\phi:x \in \R^d \mapsto |x|^{\gamma}$. Applying Ito's formula to $\phi(X^{K,M}_{t\wedge \tau^{K,M} \wedge \gamma^{K,M} })$, because the truncated coefficients coincide with the original coefficients on $\mathcal{D}^{K,M}$ we get that:
    \begin{equation} \label{eq:phiX}
        \phi(X^{K,M}_{t\wedge \tau^{K,M} \wedge \gamma^{K,M} })= \phi(X_0)+ \int_0^{t\wedge \tau^{K,M} \wedge \gamma^{K,M} } L_{\mu^{K,M}_s}\phi(X^{K,M}_s)ds+M_t
    \end{equation}
    where $M_t$ is a local martingale and for all $x,\mu \in \mathcal{D}^{K,M}$, and $L_\mu\phi(x)=\hat{L}_\mu\phi(x)+\tilde{L}_\mu\phi(x)$ is the generator of the process, with:
    \begin{equation} 
     \left\{
        \begin{array}{lll}
        \hat{L}_\mu\phi(x)&=\nabla \phi(x) \cdot b_\mu(x)+ \frac{1}{2}\nabla^2 \phi(x) : \sigma_\mu(x)\sigma_\mu(x)^T\\
    
        \tilde{L}_\mu\phi(x)&=\nabla \phi(x) \cdot \tilde{b}_\mu(x) + \frac{1}{2}\nabla^2 \phi(x) : \tilde{\sigma}_\mu(x)\tilde{\sigma}_\mu(x)^T
        \end{array}\right.
\end{equation}

Because we assumed that $\nabla f(x)$ has at most a polynomial growth of order $\beta$ and by definition of $b$ and $\sigma$, for $x,\mu \in \mathcal{D}^{K,M}$ it exist a constant $C_M$ only depending on $M$ such that: $|b_\mu(x)|+|\sigma_\mu(x)| \leq C_M (1+|x|^\beta)$. Then:
\begin{equation} \label{eq:momentL1}
\begin{split}
    \hat{L}_\mu\phi(x) &\leq C_M(1+|x|^{\beta+\gamma-1}+|x|^{2\beta+\gamma-2}) \\
\end{split}
\end{equation}

Moreover, because $\tilde{b}$ is assumed to be coercive (Assumption (B.2')) and because $\tilde{\sigma}$ has at most a polynomial growth of order $\beta$:
\begin{equation*}
\begin{split}
    \tilde{L}_\mu\phi(x) &\leq C_M(1-c_M|x|^{q+\gamma-2}+|x|^{2\beta+\gamma-2}) \\
    &\leq C_M(1-c_M|x|^{q+\gamma-2})\,,
\end{split}
\end{equation*}
which implies that: $$L_\mu \phi(x) \leq C_M(1-c_m|x|^{q+\gamma-2}) \leq C_M$$ because $q \geq \max(\beta+1, 2\beta) $ and $C_M, c_M> 0$ are changing from line to line but only depend on $M$.

Then, if we introduce a sequence of stopping times $(T^N)_{N \in \N}$ reducing the local martingale $(M_t)$, we get from equation \eqref{eq:phiX} that:
\begin{equation} \label{eq:majunif}
    \mathbb{E}\big(|X^{K,M}_{t\wedge \tau^{K,M} \wedge \gamma^{K,M} \wedge T^N}|^\gamma\big) \leq \mathbb{E}(|X_0|^\gamma)+ C_MT
\end{equation}
The application of Fatou's lemma (in order to simplify $T^N$) justifies that for all $\gamma \geq 2$, we can bound for all $\mathbb{E}\big(|X^{K,M}_{t\wedge \tau^{K,M} \wedge \gamma^{K,M} \wedge T^N}|^\gamma\big)$ uniformly in $K>0$ and $t\in [0,t]$ .

In order to get the desired estimate, we can now use the explicit expression of the local martingale $(M_t)$ and Burkholder-Davis-Gundy's inequality to get:
\begin{equation*}
    \begin{split}
        \mathbb{E}({\underset{t\in [0,T]}{\text{sup}}|X_{t\wedge \tau^{K,M} \wedge \gamma^{K,M}}^{K,M}|^\gamma}) & \leq C \mathbb{E}\Bigl[(|X_0|^\gamma+ \int_0^{T\wedge \tau^{K,M} \wedge \gamma^{K,M}\wedge t} [ |b_{\mu^{K,M}_s}(X^{K,M}_s)|^\gamma+|\tilde{b}_{\mu^{K,M}_s}(X^{K,M}_s)|^\gamma]ds\\
        &+\int_0^{T\wedge \tau^{K,M} \wedge \gamma^{K,M}} |\tilde{\sigma}_{\mu^{K,M}_s} (X^{K,M}_s)|^\gamma ds + \int_0^{T\wedge \tau^{K,M} \wedge \gamma^{K,M}} |\sigma_{\mu^{K,M}_s}(X^{K,M}_s)|^\gamma ds\Bigr]\\
    \end{split}
\end{equation*}
Then, since all coefficients have at most a polynomial growth and because $X_0$ has finite moments of every order, we can find $\gamma'\geq 2$ such that:
\begin{equation*}
    \begin{split}
        \mathbb{E}({\underset{t\in [0,T]}{\text{sup}}|X_{t\wedge \tau^{K,M} \wedge \gamma^{K,M}}^{K,M}|^\gamma}) & \leq C \po 1+ \int_0^T \mathbb{E} ( |X^{K,M}_{t \wedge \tau^{K,M} \wedge \gamma^{K,M}})|^{\gamma'}) dt \pf\\
    \end{split}
\end{equation*}
The previous estimate (equation \eqref{eq:majunif}) then justifies that the left term of the equation is bounded independently of $K$, which proves the result.
\medskip{}

In order to prove the second estimate of the lemma, because the result follows from the previous estimate if $\tilde{b}$ satisfies assumption (B.2'), we can assume that $\tilde{b}$ satisfy (B.2). We make use of Burkholder-Davis-Gundy inequality to get:

\begin{equation*}
    \begin{split}
       \mathbb{E}({\underset{t\in [0,T]}{\text{sup}}|X_{t\wedge \tau^{K,M} \wedge \gamma^{K,M}}^{K,M}|^2}) & \leq C \mathbb{E}\Bigl[(|X_0|^2+ \int_0^{T\wedge \tau^{K,M} \wedge \gamma^{K,M}\wedge t} [ |b_{\mu^{K,M}_s}(X^{K,M}_s)|^2+|\tilde{b}_{\mu^{K,M}_s}(X^{K,M}_s)|^2]ds\\
       &+\int_0^{T\wedge \tau^{K,M} \wedge \gamma^{K,M}} |\tilde{\sigma}_{\mu^{K,M}_s} (X^{K,M}_s)|^2 ds + \int_0^{T\wedge \tau^{K,M} \wedge \gamma^{K,M}} |\sigma_{\mu^{K,M}_s}(X^{K,M}_s)|^2ds\Bigr]\\
       &\leq C  \Big(1+\mathbb{E}\Bigl[ \int_0^{T\wedge \tau^{K,M} \wedge \gamma^{K,M}} |X^{K,M}_s|^{2\beta} +|X^{K,M}_s|^{\alpha} ds\Big] \Big)\\
       &\leq C  \Big(1+\mathbb{E}\Bigl[ \int_0^{T\wedge \tau^{K,M} \wedge \gamma^{K,M}} |X^{K,M}_s|^{\alpha}ds  \Big] \Big)\\
    \end{split}
\end{equation*}

Where $C$ is independent of $K$ and may change from line to line. The last inequality is a consequence of Jensen's inequality because $\alpha\geq 2\beta$.

Using that
\begin{equation} \label{eq:controlemoment1}
    \begin{split}
        \mathbb{E}[ \int_0^{T\wedge \tau^{K,M} \wedge \gamma^{K,M}} |X^{K,M}_s|^\alpha ds ] &\leq \mathbb{E}[ \int_0^{T} \1_{s\leq \tau^{K,M}}  |X^{K,M}_s|^\alpha ds] \\
        &=\mathbb{E}\Bigl[ \mathbb{E}\bigl[\int_0^T \1_{s\leq \tau^{K,M}}  |X^{K,M}_s|^\alpha  ds \big| \mathcal{F}^0 \bigr] \Bigr]\,,
    \end{split}
\end{equation}
by the conditional Fubini-Tonelli theorem, and because $\1_{s\leq \tau^{K,M}} $ is $\mathcal{F}^0$ measurable,
\begin{equation}  \label{eq:controlemoment2}
    \begin{split}
        \mathbb{E}\bigl[\int_0^T \1_{s\leq \tau^{K,M}}  |X^{K,M}_s|^\alpha ds \big| \mathcal{F}^0 \bigr] &= \int_0^T \1_{s\leq \tau^{K,M}}  \mathbb{E} \bigl[|X^{K,M}_s|^\alpha  \big| \mathcal{F}^0]ds \\
        &= \int_0^{T\wedge \tau^{K,M}} m_\alpha(\mu_s^{K,M})ds \\
        & \leq C_M
    \end{split}
\end{equation}
Since  $m_\alpha(\mu^{K,M}_s)$ is uniformly bounded by $M$ until $\tau^{K,M}$ (and independently of $K$). Again, this concludes.

\end{proof}

\subsection{Proof of Lemma \ref{lem:lastlemma}} \label{app:prooflastlemma}

\begin{proof}
    We first show that the processes $(X_{t\wedge \tau^M})_{t\in[0,T]}$ and $(\mu_{t\wedge \tau^M})_{t\in[0,T]}$ are well defined. This will be a consequence of the previous Lemma and of the construction of the processes $X^{K,M}$ and $\mu^{K,M}$.

    Let $M \in \N^*$ be fixed and $m \in \N$. Because of the previous lemma, the sequence $(\tau^{K,M}\wedge \gamma^{K,M})_{K\in \N^*}$ is a.s constant after the (random) index $\mathcal{K}_M$ and stabilizes to $ \tau^{\mathcal{K}_M,M}$. Combining this result with Equation \eqref{eq:unicitejusqueTA}, and \eqref{eq:comparaisonmu}, we get that a.s for all $k,m\in \N$:
    $$(\mu_{t\wedge \tau^{\mathcal{K}_M+k,M}\wedge \gamma^{\mathcal{K}_M+k,M}}^{\mathcal{K}_M+k,M+m})_{t\in[0,T]}=(\mu_{t\wedge \tau^{\mathcal{K}_M,M} \wedge \gamma^{\mathcal{K}_M,M}}^{\mathcal{K}_M,M})_{t\in[0,T]}=(\mu_{t\wedge \tau^{\mathcal{K}_M,M}}^{\mathcal{K}_M,M})_{t\in[0,T]}=(\mu_{t\wedge \tau^{M}}^{\mathcal{K}_M,M})_{t\in[0,T]},,$$
    so a.s the sequence of processes $(\mu_{\cdot\wedge \tau^{K,M}\wedge \gamma^{K,M}}^{K,M+m})_{K\in \N^*}$ is constant after a certain rank, and its limit is a process that does not depend on the value of $m$. We can define $$\mu_{\cdot\wedge\tau^{M}}^{M+m}= \lim_{K\rightarrow \infty} \mu_{\cdot\wedge \tau^{K,M}\wedge \gamma^{K,M}}^{K,M+m} $$ and conclude that $\mu_{\cdot\wedge\tau^{M}}^{M+m}=\mu_{\cdot\wedge\tau^{M}}^{M}$. We can finally use the notation $\mu_{\cdot \wedge \tau^{M}}=\mu_{\cdot\wedge\tau^{M}}^{M}$ because its definition depends on $M$ only through the stopping time $\tau^M$ (otherwise there would be a dependency in $m$). Finally, $\mu_{\cdot\wedge \tau^M}$ is well defined, and we can use the exact same reasoning to justify that $X_{\cdot\wedge\tau^M}$ is also well defined.

    By definition of $\tau^M$, we have that:
\begin{equation*}
    \begin{split}
        \tau^{M}&=\tau^{\mathcal{K}_M,M}\\
        &=  \inf \{t\in [0,T], \,  \mu_t^{\mathcal{K}_M,M} \in (\Lambda^M)^c\} \\
        &=\inf \{t\in [0,T], \,  \mu_t\in (\Lambda^M)^c\} \\
        &=\inf\{t\in [0,T], \, \mu_t(|x|^\alpha ) \geq M , \quad d(\mu_t(f), \mathcal{S})\leq \frac{1}{M}, \quad \det(\mu_t(\nabla f ^T \nabla f))\leq \frac{1}{M} \}
    \end{split}
\end{equation*}
Hence, the blowing time $\tau=\lim_{M\rightarrow \infty}\tau^M$ is well defined and corresponds to the definition given in the statement of Proposition \ref{prop:stochasticmomentswelldefined}.

\medskip{}

    We are now going to show that the process $(X_{t}, \mu_{t})_{t\in [0,T]}$ is the unique solution of of the conditional McKean Vlasov equation with coefficients $b,\tilde{b},\sigma,\tilde{\sigma}$ up to the blowing time $\tau$.

    We first need to check that the processes are adapted, continuous and that $\mu$ is a continuous version of the conditional laws of $X$. By definition, $\mu_{\cdot\wedge \tau^M}$, (resp. $X_{\cdot\wedge \tau^M}$) is defined as a limit of sequence of $\mathbb{F}^{W^0}$ (resp. $\mathbb{F}^{W_0,W,W^0}$) adapted processes, so is also adapted to this filtration. Moreover, because $X_{\cdot \wedge \tau^M}$ and $\mu_{\cdot \wedge \tau^M}$ are limits of sequences of continuous processes that are a.s stationary after a certain rank, they are also a.s continuous.

    The process $\mu_{\cdot \wedge \tau^{M}}$ is also a version of the flow of conditional law of $X_{\cdot\wedge \tau^{M}}$ with respect to $\mathcal{F}^0$. Indeed, for $h$ a continuous and bounded function on $\R^d$, for all $0\leq t \leq T$ and $K>0$, the conditional Fubini Theorem ensures that:
    $$\mathbb{E}\bigl(h(X^{K,M}_{t\wedge \tau^{K,M}}) \big| \mathcal{F}^0 \bigr)=\mu^{K,M}_{t\wedge \tau^{K,M}}(h)$$
    By the conditional dominated convergence theorem, the left side of the equation $\mathbb{P}^0$ a.s converges when $K$ goes to infinity to $\mathbb{E}\bigl(h(X_{t\wedge \tau^{M}}) \big| \mathcal{F}^0 \bigr)$. We also know that $(\mu_{t\wedge \tau^{M}})_{0\leq t \leq T}$ is $\mathbb{P}^0$ a.s continuous in $\mathcal{P}_\alpha(\R^d)$ and because the convergence for the Wasserstein distance imply the convergence in law:
    $$\mathbb{P}^0 \text{a.s,  } \quad \mu^{K,M}_{t\wedge \tau^{K,M}}(h) \underset{K\rightarrow \infty }{\longrightarrow} \mu_{t\wedge \tau^{M}}(h)$$
    Then, $\mathbb{P}^0$ a.s:$$ \mathbb{E}\bigl(h(X_{t\wedge \tau^{M}}) \big| \mathcal{F}^0 \bigr)=\mu_{t\wedge \tau^{M}}(h)$$ and we have that $(\mu_{t\wedge \tau^{M}})_{0\leq t \leq T}$ is a (continuous) version of the flow of conditional law of  $(X_{t\wedge \tau^{M}})$ with respect to $\mathcal{F}^0$.

    The justification that:
    $$\mathbb{E}[\hspace{2pt} \int_0^{T\wedge \tau^M} (|b(X_t,\mu_t)|+ |\tilde{b}(X_t,\mu_t)|+|\sigma(X_t,\mu_t)|^2+|\tilde{\sigma}(X_t,\mu_t)|^2)dt ] < \infty$$
    directly follows from the polynomial growth of the coefficients and of lemma \ref{lem:boundmoments}.
    
    To verify that the last condition is verified, we note 
    \begin{equation*}
        \begin{split}
            A^M:= \Bigl\{ \, \forall 0\leq t \leq T, \quad X_{t\wedge \tau^M}&=X_0+\int_0^{t\wedge \tau^M}b(X_s,\mu_s)+\tilde{b}(X_s,\mu_s)ds\\
            &+\int_0^{t\wedge \tau^M}\sigma(X_s,\mu_s)dW_s+\int_0^{t\wedge \tau^M}\sigma^0(X_s,\mu_s)dW^0_s \Bigr\}
        \end{split}
    \end{equation*}
    and for all $K \in \N^*$: 
    \begin{equation*}
        \begin{split}
            A_K^M:= \Bigl\{ \, \forall 0\leq t \leq T, \quad &X_{t\wedge \tau^{K,M}\wedge \gamma^{K,M}}=X_0+\int_0^{t\wedge \tau^{K,M}\wedge \gamma^{K,M}}b(X^{K,M}_s,\mu_s)+\tilde{b}(X^{K,M}_s,\mu_s)ds\\
            &+\int_0^{t\wedge \tau^{K,M}\wedge \gamma^{K,M}}\sigma(X^{K,M}_s,\mu_s)dW_s+\int_0^{t\wedge \tau^{K,M}\wedge \gamma^{K,M}}\sigma^0(X^{K,M}_s,\mu_s)dW^0_s \Bigr\}
        \end{split}
    \end{equation*}

    Since $(X^{K,M},\mu^{K,M})$ and $(X,\mu)$ coincide until $T\wedge \tau^{K,M}\wedge \gamma^{K,M}$, and because of Lemma~\ref{lem:relationsT.A}, we have that:
    $$A^M=\underset{K \in \N^*}{\bigcap}A_K^M$$
    Then, because for all $K \in \N^*$ $\mathbb{P}(A_K^M)=1$ by definition of $X^{K,M},\mu^{K,M}$, we then have that: $$\mathbb{P}(A^M)=1\,.$$ 
    This finishes to prove that $X,\mu$ is a solution to the conditional McKean Vlasov equation until the blowing time $\tau$.

    \medskip{}

    We finally prove that $X,\mu$ is the unique solution to the conditional McKean Vlasov equation until the blowing time $\tau$. Let $\tilde{X},\tilde{\mu}$ be another solution reduced by the sequence of stopping times $(\tilde{\tau}^M)_{M\in \N^*}$. 
    In the same spirit than $\tau^{K,M}$ and $\gamma^{K,M}$, we define the following stopping times:
    \begin{align*}
    \tilde{\gamma}^{K,M}&= \text{inf}\{0 \leq  t \leq T, \quad |\tilde{X}_t| \geq  K\}\\
    \tilde{\tau}^{K,M}&= \text{inf}\{0 \leq  t \leq T, \quad \tilde{\mu}_t \notin \Lambda^M\}
\end{align*}

Then, the processes $(X_t,\mu_t)$ and $(\tilde{X}_t,\tilde{\mu})$ are both solutions until $\tilde{\tau}^M \wedge \tau^{K,M}\wedge\gamma^{K,M}\wedge\tilde{\tau}^{K,M}\wedge\tilde{\gamma}^{K,M} $ of the conditional McKean Vlasov equation \eqref{eq:condMVwithKandM} directed by the coefficients $b^{K,M},\tilde{b}^{K,M},\sigma^{K,M},\tilde{\sigma}^{K,M}$. By strong unicity of this equation until a stopping time, the processes $(X_t,\mu_t)$ and $(\tilde{X}_t,\tilde{\mu})$ coincide until $\tau^{K,M}\wedge\gamma^{K,M}\wedge\tilde{\tau}^{K,M}\wedge\tilde{\gamma}^{K,M}$. A direct consequence is that:
$$\mathbb{E}\Bigl( \, \underset{0\leq t \leq T}{\text{sup}}|X_{t\wedge\tau^{K,M}\wedge\gamma^{K,M}\wedge\tilde{\tau}^{M}}- \tilde{X}_{t\wedge\tau^{K,M}\wedge\gamma^{K,M}\wedge\tilde{\tau}^{M}}| \Bigr)=0$$
Then, thanks to Lemma \ref{lem:relationsT.A}:
\begin{equation*}
    \underset{0\leq t \leq T}{\text{sup}}|X_{t\wedge\tau^{M}\wedge\tilde{\tau}^{M}}- \tilde{X}_{t\wedge\tau^{M}\wedge\tilde{\tau}^{M}}|= \underset{K\rightarrow \infty}{\text{lim}} \, \underset{0\leq t \leq T}{\text{sup}}|X_{t\wedge\tau^{K,M}\wedge\gamma^{K,M}\wedge\tilde{\tau}^{M}}- \tilde{X}_{t\wedge\tau^{K,M}\wedge\gamma^{K,M}\wedge\tilde{\tau}^{M}}| 
\end{equation*}
By Fatou's Lemma:
\begin{equation*}
\begin{split}
    \mathbb{E}\Bigl( \, \underset{0\leq t \leq T}{\text{sup}}|X_{t\wedge\tau^{M}\wedge\tilde{\tau}^{M}}- \tilde{X}_{t\wedge\tau^{M}\wedge\tilde{\tau}^{M}}| \Bigr)&\leq  \underset{K\rightarrow \infty}{\text{liminf}} \, \mathbb{E}\Bigl( \, \underset{0\leq t \leq T}{\text{sup}}|X_{t\wedge\tau^{K,M}\wedge\gamma^{K,M}\wedge\tilde{\tau}^{M}}- \tilde{X}_{t\wedge\tau^{K,M}\wedge\gamma^{K,M}\wedge\tilde{\tau}^{M}}| \Bigr) \\
    &\leq 0
\end{split}
\end{equation*}
As a consequence, strong uniqueness until the blowing time holds for the conditional McKean Vlasov equation \eqref{eq:stomomentdyn}, which concludes the overall proof.

\end{proof}

\bibliographystyle{plain}
\bibliography{biblio}
\end{document}